\documentclass[11pt]{article}

\usepackage{geometry}
\usepackage{amsmath, amssymb, amsfonts, bm, mathtools}
\usepackage{amsthm}
\usepackage{xcolor}
\definecolor{darkblue}{rgb}{0,0,.5}
\usepackage{graphicx}
\usepackage{subfigure}

\usepackage[numbers]{natbib}

\usepackage[colorlinks=true,allcolors=darkblue]{hyperref}       
\usepackage{url}            
\usepackage{booktabs}       
\usepackage{amsfonts}       
\usepackage{nicefrac}       
\usepackage{microtype}      

\allowdisplaybreaks

\usepackage{float}
\usepackage{multirow}
\usepackage{footnote}
\usepackage{dsfont}

\usepackage{algorithm}
\usepackage{algorithmic}
\usepackage{nicefrac}

\usepackage{hyperref}
\usepackage[capitalize]{cleveref}
\usepackage{crossreftools}

\usepackage{tikz}
\usepackage{overpic}

\definecolor{cc}{RGB}{125,0,0}
\definecolor{jd}{RGB}{0,125,0}
\definecolor{ha}{RGB}{0,0,200}

\usepackage{dsfont}

\newcommand{\ind}{\mathds{1}}
\newcommand{\var}{\mathsf{Var}}
\newcommand{\cov}{\mathsf{Cov}}
\newcommand{\Ep}{\mathbb{E}}
\newcommand{\E}{\mathbb{E}}
\newcommand{\Prb}{\mathbb{P}}
\newcommand{\R}{\mathbb{R}} 
\newcommand{\N}{\mathbb{N}} 
\renewcommand{\P}{\mathbb{P}}
\newcommand{\mc}[1]{\mathcal{#1}}
\newcommand{\brc}[1]{\left\{{#1}\right\}}
\newcommand{\prn}[1]{\left({#1}\right)} 
\newcommand{\prns}[1]{({#1})} 
\newcommand{\brk}[1]{\left[{#1}\right]} 
\newcommand{\norm}[1]{\left\|{#1}\right\|} 
\newcommand{\what}[1]{\widehat{#1}}
\newcommand{\wt}[1]{\widetilde{#1}}
\newcommand{\wb}[1]{\overline{#1}}

\newcommand{\majY}{Y^+}

\newcommand{\sgn}{\mathsf{sign}}

\newcommand{\normal}{\mathsf{N}}
\newcommand{\bindist}{\mathsf{Binomial}}
\newcommand{\maj}{\mathsf{maj}}

\newcommand{\lr}{^{\textup{lr}}}
\newcommand{\sem}{^{\textup{sp}}}

\newcommand{\mv}{^{\textup{mv}}}
\newcommand{\ds}{^{\textsc{ds}}}
\newcommand{\dsp}{^{\textsc{ds-p}}}
\newcommand{\glad}{^{\textsc{glad}}}
\newcommand{\gladp}{^{\textsc{glad-p}}}

\newcommand{\proj}{\mathsf{P}}

\newcommand{\projperp}{\proj_{u\opt}^\perp}
\newcommand{\Flink}{\mathcal{F}_{\mathsf{link}}}
\newcommand{\lipconst}{\mathsf{L}}
\newcommand{\loss}{\ell}
\newcommand{\poploss}{L}

\newcommand{\flink}{\mathcal{F}_{\mathsf{link}}}
\newcommand{\fsymlink}{\mathcal{F}_{\mathsf{link}}^0}
\newcommand{\fsymlinkL}{\mathcal{F}_{\mathsf{link}}^{\mathsf{L}}}
\newcommand{\fsymlinksp}{\mathcal{F}_{\mathsf{link}}^{\mathsf{sp}}}
\newcommand{\ltwo}[1]{\norm{#1}_2} 
\newcommand{\linf}[1]{\norm{#1}_\infty} 

\newcommand{\norms}[1]{\|{#1}\|} 
\newcommand{\ltwos}[1]{\norms{#1}_2} 
\newcommand{\linfs}[1]{\norms{#1}_\infty} 

\newcommand{\<}{\langle} 
\renewcommand{\>}{\rangle}

\DeclareMathOperator*{\argmax}{argmax}
\DeclareMathOperator*{\argmin}{argmin}

\newcommand{\simiid}{\stackrel{\textup{iid}}{\sim}}
\newcommand{\indic}[1]{1\!\left\{{#1}\right\}}

\newcommand{\cd}{\rightsquigarrow}
\newcommand{\cp}{\stackrel{p}{\rightarrow}}
\newcommand{\cas}{\stackrel{a.s.}{\rightarrow}}

\makeatletter
\long\def\@makecaption#1#2{
  \vskip 0.8ex
  \setbox\@tempboxa\hbox{\small {\bf #1:} #2}
  \parindent 1.5em  
  \dimen0=\hsize
  \advance\dimen0 by -3em
  \ifdim \wd\@tempboxa >\dimen0
  \hbox to \hsize{
    \parindent 0em
    \hfil 
    \parbox{\dimen0}{\def\baselinestretch{0.96}\small
      {\bf #1.} #2
    } 
    \hfil}
  \else \hbox to \hsize{\hfil \box\@tempboxa \hfil}
  \fi
}
\makeatother

\newcommand{\defeq}{\coloneqq}

\newcommand{\LtwoP}[1]{\norm{#1}_{L^2(\Prb)}}
\newcommand{\LtwoPbold}[1]{\norm{#1}_{L^2(\Prb)}}
\newcommand{\LtwoPn}[1]{\norm{#1}_{L^2(\Prb_n)}}
\newcommand{\LtwoPns}[1]{\norms{#1}_{L^2(\Prb_n)}}
\newcommand{\LtwoQ}[1]{\norm{#1}_{L^2(Q)}}
\newcommand{\LtwoQn}[1]{\norm{#1}_{L^2(Q_n)}}
\newcommand{\LtwoQns}[1]{\norms{#1}_{L^2(Q_n)}}
\newcommand{\lipnorm}[1]{\norm{#1}_{\textup{Lip}}}

\newcommand{\half}{\frac{1}{2}}
\newcommand{\event}{\mc{E}}

\newcommand{\card}{\textup{card}}
\newcommand{\starhull}{\textup{Star}}
\newcommand{\hinge}[1]{\left({#1}\right)_+}

\definecolor{innerboxcolor}{rgb}{.9,.95,1}
\definecolor{outerlinecolor}{rgb}{.6,0,.2}

\newcommand{\opt}{^\star}
\newcommand{\subopt}{_\star}
\newcommand{\gme}{\what{\theta}_{n,m}} 
\newcommand{\tmle}{\what{\theta}\lr_{n,m}} 
\newcommand{\mve}{\what{\theta}\mv_{n,m}} 
\newcommand{\spe}{\what{\theta}\sem_{n,m}} 
\newcommand{\dse}{\what{\theta}\ds_{n,m}} 
\newcommand{\dspe}{\what{\theta}\dsp_{n,m}} 
\newcommand{\glade}{\what{\theta}\glad_{n,m}} 
\newcommand{\gladpe}{\what{\theta}\gladp_{n,m}} 
\newcommand{\G}{\mathbb{G}}

\hypersetup{hypertexnames=false}

\usepackage{enumitem}
\usepackage{placeins}

\newtheorem{claim}{Claim}[section]
\newtheorem{lemma}[claim]{Lemma}
\newtheorem{assumption}{Assumption}

\newtheorem{theorem}{Theorem}

\newtheorem{proposition}{Proposition}

\newtheorem{corollary}{Corollary}

\usepackage{setspace}

\usepackage{xr}

\title{How many labelers do you have? \\
  A closer look at gold-standard labels}

\author{%
  Chen Cheng\thanks{Department of Statistics, Stanford University; email: \texttt{chencheng@stanford.edu}.} \and
  Hilal Asi\thanks{Department of Electrical Engineering, Stanford University; email: \texttt{asi@stanford.edu}.}
  \and
  John Duchi\thanks{Departments of Statistics and Electrical Engineering, Stanford University; email: \texttt{jduchi@stanford.edu}.}
}

\begin{document}
\maketitle


\begin{abstract}
  The construction of most supervised learning datasets revolves around
  collecting multiple labels for each instance, then aggregating the labels
  to form a type of ``true'' label.
  We question the wisdom of this pipeline by developing a (stylized)
  theoretical model of this process and analyzing its statistical
  consequences, showing how access to non-aggregated label information can
  make training well-calibrated models more feasible than it is with
  cleaned labels.
  The entire story, however, is subtle, and the
  contrasts between aggregated and fuller label information depend on
  the particulars of the problem,
  where estimators that use aggregated information exhibit robust but slower
  rates of convergence, while estimators that can effectively
  leverage all labels converge more quickly \emph{if} they have
  fidelity to (or can learn) the true labeling process.
  The theory
  makes several predictions for real-world datasets, including when
  non-aggregate labels should improve learning performance, which we
  test to corroborate the validity of our predictions.
\end{abstract}



\section{Introduction}

Challenge datasets drive advances in statistical machine
learning~\cite{MarcusSaMa94, LewisYaRoLi04, AsuncionNe07, KrizhevskyHi09,
  DengDoSoLiLiFe09, RussakovskyDeSuKrSaMaHuKaKhBeBeFe15},
making the collection of
labeled data essential~\cite{Donoho17}.
While in the past, experts could label reasonably sized datasets, the cost
and size of modern datasets often precludes such annotation, motivating
crowdsourcing and other dataset generation strategies that aggregate
non-expert and expert feedback~\citep{DawidSk79,
  WhitehillWuBeMoRu09, RussakovskyDeSuKrSaMaHuKaKhBeBeFe15,
  CodellaRoTsCeDuGuHeKaLiMa19, GadreEtAl23}.
%
Frequently, published data contain aggregated label information; for
example, in the dataset~\cite{CodellaRoTsCeDuGuHeKaLiMa19}
of skin lesion images to be classified as
malignant, biopsy and histology confirm positive examples, while physician
consensus determines negative examples, yet the published data includes no
notion of disagreement or uncertainty in these consensuses.
Similarly, the benchmark CIFAR~\citep{KrizhevskyHi09} and
ImageNet~\citep{DengDoSoLiLiFe09, RussakovskyDeSuKrSaMaHuKaKhBeBeFe15}
computer vision challenges include only a single label for each example, in
spite of deriving from multiple pieces of feedback per image.
Researchers typically perform such aggregation in effort to obtain
``high-quality'' or ``true'' labels~\cite{DengDoSoLiLiFe09,
  RussakovskyDeSuKrSaMaHuKaKhBeBeFe15, DawidSk79,
  WhitehillWuBeMoRu09}.

Yet most statistical machine learning development does not investigate this
full data generating process, assuming only that data comes in the form of
$(X, Y)$ pairs of covariates $X$ and labels $Y$~\cite{Vapnik95,
  BoucheronBoLu05, BartlettJoMc06, HastieTiFr09}.
We argue for a more
holistic perspective: that analysis and algorithmic development should focus
on the complete statistical learning pipeline, from dataset construction to
model output.
We question the extent to which such label cleaning strategies are
essential, arguing that benchmark datasets ought to include noisy labels as
a rule; one may always clean the data in subsequent stages, but sometimes
the more nuanced labeling information may yield more accurate downstream
models.



We develop a stylized theoretical model to capture uncertainty in the
labeling process, allowing us to understand the limitations of and possible
improvements to using aggregated data in a statistical learning pipeline.
We model each example as a pair $(X_i, (Y_{i1},\dots,Y_{im}))$ where $X_i$
is a data point and $Y_{ij}$ are noisy binary labels.
The basic
formulation of our results compares two methods: empirical risk
minimization using all labels and empirical risk minimization using
majority vote labels $\majY$.
While this oversimplifies label
aggregation strategies~\cite{DawidSk79, RussakovskyDeSuKrSaMaHuKaKhBeBeFe15,
  RatnerBaEhFrWuRe17},
the simplicity (i) allows us to understand
fundamental limitations of algorithms based on majority-vote aggregation,
(ii) facilitates circumventing these limits via non-aggregated
information, and (iii) provides a base for future work: the purposeful
simplicity permits classical tools to apply.
In these models, our results show that estimators using non-aggregated labels
outperform those that use aggregated (majority-vote)
labels if the predictive model has fidelity to the data,
majority-vote estimators provide robust (but slower) convergence, and these
tradeoffs are fundamental.
We extend the results to include misspecification, semiparametric
scenarios, and simple models of learned annotator reliability.

While our models are stylized, they make testable predictions for real data
beyond simplified scenarios, in particular, that models fit on
non-aggregated data should both make better-calibrated predictions and have
lower classification error than those using cleaned labels.
Experiments on two real and one large semisynthetic dataset corroborate the
implications of our theory beyond logistic models: majority-vote-based
algorithms yield uncalibrated predictors, while algorithms using full label
information provide (more) calibrated predictors.
The former algorithms exhibit worse classification error in our experiments,
with the error gap depending on parameters---e.g.,
noise---we address in our theoretical models.

\subsection{Problem formulation}
\label{sec:formulation}


Consider covariates $X_i \simiid \P_X$, $X_i \in \R^d$, in a
binary classification problem with $m$ labelers,
who annotate points independently through a generalized linear
model (GLM) via link functions 
\begin{align*}
  \sigma_1\opt, \dots,
  \sigma_m\opt \in \flink := \brc{\sigma : \R \to [0, 1]
    \mid \sigma(0) = 1/2, \sgn(\sigma(t) - 1/2) = \sgn(t)}.
\end{align*} 
For example, in the logistic model where
labelers have identical links, $\sigma_j\opt(t) = 1 / \prn{1 + e^{-t}}$.
If $\sigma(t) + \sigma(-t) = 1$ for all $t \in \mathbb{R}$, we say the
link function is symmetric, denoting the class of symmetric
links by $\fsymlink \subset \flink$.
Labels then follow the distribution
\begin{align}
  \Prb_{\sigma, \theta}(Y = y \mid X = x) & =
  \sigma ( y \< \theta, x\>) ~~~ \mbox{for~} y \in \left\{\pm 1\right\},
  ~ x,\theta \in \R^d. \label{eq:general-calibration-model}
\end{align}
Key to our model, and what allows our analysis, is that we assume
labelers use the same linear classifier $\theta\opt \in \R^d$, 
so we obtain conditionally independent labels $Y_{ij} \sim \Prb_{\sigma_j\opt,
  \theta\opt}(\cdot \mid X_i)$.
 We
seek to recover $\theta\opt$ or the direction $u\opt := \theta\opt /
\ltwo{\theta\opt}$ from the observations $(X_i, Y_{ij})$.

\paragraph{Error measures.}
We evaluate an estimator $\what{\theta}$ and associated direction $\what{u}
\defeq \what{\theta} / \ltwos{\what{\theta}}$ through
\begin{enumerate}[leftmargin=10em,label=(\textsc{E}.\roman*)]
\item \label{item:class}
  The classification error: $\ltwos{u\opt - \what{u}}
  = \sqrt{2(1 - \<u\opt, \what{u}\>)}$.
\item \label{item:cal}
  The calibration error:~~ $\ltwos{\theta\opt - \what{\theta}}$.
\end{enumerate}
To motivate these errors, note that for classification, we only need to
control the difference between the directions $\what{u}$ and $u\opt$, while
calibration---that for a data point $X$, the value $\sigma_j\opt( \<
\what{\theta}, X \>)\approx \sigma_j\opt(\< \theta\opt, X \>)$---requires
controlling the error in $\what{\theta}$ as an estimate of $\theta\opt$.
If $X$ is rotationally symmetric,
then $\P(\sgn(\<X, u\>) \neq \sgn(\<X, u\opt\>)) = \frac{1}{\pi}
\cos^{-1}\<u, u\opt\>$ for any unit vectors $u, u\opt$; because $\cos^{-1} t
= (1 + o(1)) \sqrt{2(1 - t)}$ as $t \uparrow 1$, the classification
error~\ref{item:class}, $\ltwo{u\opt - u} =
\sqrt{2(1 - \<u, u\opt\>)}$, is asymptotically equivalent to the angle
between $u$ and $u\opt$.

\paragraph{Estimators.}
We consider two types of estimators: one
using processed labels $\majY_i$ for each
example $X_i$, and the other
using each label $Y_{i1}, \ldots, Y_{im}$.
To center discussion, we instantiate this temporarily via logistic
regression for the logistic link $\sigma\lr(t) = \frac{1}{1 + e^{-t}}$,
defining
\begin{align*}
  \loss_\theta\lr(y \mid x) = -\log \Prb_{\sigma\lr, \theta}(y \mid x)
  = \log(1 + e^{-y\<x,\theta\>})  .
\end{align*}
In the non-aggregated model, we let $\Prb_{n,m}$ be the empirical
measure on $\{(X_i, (Y_{i1},\dots,Y_{im}))\}$ and consider the logistic
regression estimator (i.e., the MLE assuming the logistic model)
\begin{align}
  \label{eq:maximum-likelihood}
  \tmle = \argmin_\theta \left\{\Prb_{n,m} \loss_\theta\lr
  = \frac{1}{nm} \sum_{i = 1}^n \sum_{j = 1}^m \loss_\theta\lr(Y_{ij} \mid X_i)
  \right\}.
\end{align}
We focus on the simplest
aggregation strategy, where example $i$ has majority vote label
\begin{align*}
  \majY_i = \maj(Y_{i1}, \ldots, Y_{im})
  \defeq \sgn\left(\sum_{j=1}^m Y_{ij}\right)
  ~~~\mbox{where}~~~
  \sgn(0) \sim \mathsf{Ber}(1/2)
\end{align*}
Then letting $\wb{\Prb}_{n,m}$ be the empirical measure on
$\{(X_i, \majY_i)\}$, the majority-vote estimator solves
\begin{align}
  \label{eq:logistic-majority}
  \mve = \argmin_\theta  \brc{\wb{\Prb}_{n,m}\loss_\theta\lr = \frac 1 n \sum_{i=1}^n  \loss_\theta\lr(\majY_i \mid X_i)}.
\end{align}
Method~\eqref{eq:logistic-majority} acts as our proxy for the ``standard''
data analysis pipeline, with cleaned labels, while
method~\eqref{eq:maximum-likelihood} is our proxy for non-aggregated methods
using all labels.
More general formulations than the majority
vote~\eqref{eq:logistic-majority} could allow complex aggregation
strategies, e.g., crowdsourcing, but we simplify to capture what we view as
the essential features of problems (e.g.\ CIFAR~\cite{KrizhevskyHi09} or
ImageNet~\cite{RussakovskyDeSuKrSaMaHuKaKhBeBeFe15}) where only aggregated
label information is available.

Our main results characterize the
asymptotics of the
estimators $\mve$ and $\tmle$.
Under
appropriate assumptions on the data generating
mechanisms~\eqref{eq:general-calibration-model}, which include
misspecification, we (i) provide consistency results that elucidate the
limits for $\mve$, $\tmle$, and a few more general
estimators, and (ii) evaluate their limit distributions via
asymptotic normality. The latter allows comparison
of the estimators through their limiting covariances, which
exhibit (to us) interestingly varying dependence on $m$ and the scaling of
the true parameter $\theta\opt$.

\subsection{Summary of results and implications}

Asymptotic characterization of the
estimators~\eqref{eq:maximum-likelihood}--\eqref{eq:logistic-majority}
provides our main avenue to highlighting the import of multiple labels.
Table~\ref{table:comparisons}  summarizes our results, comparing
the performance of the MLE-type~\eqref{eq:maximum-likelihood} and
majority vote~\eqref{eq:logistic-majority} estimators for the
classification~\ref{item:class} and calibration~\ref{item:cal} errors.
We highlight estimator performance in two scenarios:
``single link'' corresponds to the case that the loss is
well-specified, while ``mis-specified''
corresponds to the case that $\loss_\theta(y \mid x) \neq -\log
\Prb_{\sigma,\theta}(y \mid x)$, which may occur when labelers have
distinct link functions $\sigma$ or the model is incorrect.
In all cases, the
estimate $\what{\theta}$ and
vector $\what{u} = \what{\theta} / \ltwos{\what{\theta}}$
are asymptotically normal around a scalar multiple of $\theta\opt$,
with differing scaling and asymptotic variance.
With $m$ labelers,
each estimator yields a rate $r(m)$ and scale $t(m)$ for which
\begin{align}
\tag{\ensuremath{3'}}
  \label{eqn:limit-scaling-for-table}
  \sqrt{n}(\what{u} - u\opt)
  \cd \normal(0, r(m) \Sigma_{\textup{n}})
  ~~~ \mbox{and} ~~~
  \sqrt{n}(\what{\theta} - t(m) \theta\opt)
  \cd \normal(0, \Sigma_{\textup{u}}),
\end{align}
where $\Sigma_{\textup{n}}$ and $\Sigma_{\textup{u}}$ are ``normalized'' and
``unnormalized'' covariances, respectively.
The classification error~\ref{item:class} is small whenever $r(m) \ll 1$,
while the calibration error~\ref{item:cal} is small when $t(m) \approx 1$,
and Table~\ref{table:comparisons} shows the
scalings~\eqref{eqn:limit-scaling-for-table} for the different estimators.

\begin{table}
  \centering
  \begin{tabular}{c|c|c}
    \toprule 
    & MLE-type estimator~\eqref{eq:maximum-likelihood}
    & majority-vote estimator~\eqref{eq:logistic-majority} \\
    \hline
    single link
    & \textbf{Classification}:
    $r(m) \asymp \frac{1}{m}$
    & Classification: $r(m) \asymp \frac{1}{\sqrt{m}}$ \\
    (well-specified) & \textbf{Calibration}:
    $t(m) = 1$ &
    Calibration: $t(m) \asymp \sqrt{m}$ \\ 
    \hline
    mis-specified or
    & Calibration: $t(m) \asymp 1$ &  Calibration: $t(m) \asymp \sqrt{m}$  \\ 
    multiple links & Classification: $r(m) \asymp 1$ &
    \textbf{Classification}: $r(m) \asymp 1 / \sqrt{m}$
    \\
    \hline
  \end{tabular}
  \caption{
    Comparison of multi-label and majority-vote estimators
    for $X \sim \normal(0, I_d)$
    according to the scalings in classification error
    $r(m)$ and calibration error $t(m)$ that
    the limits~\eqref{eqn:limit-scaling-for-table} identify.
    If $t(m) \gg 1$, the model is uncalibrated; if the rate $r(m)$ is small,
    the model achieves lower asymptotic variance.
    Boldface indicates the associated estimator
    is (rate) optimal. \label{table:comparisons}}
\end{table}


\paragraph{Improved performance using multiple labels for well-specified models}
We begin in Section~\ref{sec:well-sepcified-mod} by specializing our main
results to the logistic model.
We show that the MLE~\eqref{eq:maximum-likelihood} is
calibrated and enjoys faster rates of convergence (in $m$) than the
majority-vote estimator~\eqref{eq:logistic-majority}.
The improvements depend on the underlying distribution, and
Propositions~\ref{prop:lowd-maximum-likelihood}
and~\ref{prop:lowd-majority-vote-m-noise} connect them to
Mammen-Tsybakov-type noise conditions.
In ``standard'' cases, the improvement scales as
$\sqrt{m}$; for problems with little classification noise, the majority vote
estimator becomes brittle (relatively slower), while the convergence rate
gap decreases for noisier problems.

\paragraph{Robustness of majority-vote estimators}
Our results also evidence support for majority vote
estimators~\eqref{eq:logistic-majority}.  Section~\ref{sec:mis-spec}
provides master results holding for well- and mis-specified losses, which
highlight the robustness of the majority vote estimator
(cf.\ Table~\ref{table:comparisons}).
The faster rates MLE-type estimators~\eqref{eq:maximum-likelihood} enjoy
when the model is correct break down in ways we make precise for
mis-specified models; in contrast, majority-vote-type
estimators~\eqref{eq:logistic-majority} maintain their (not quite optimal)
improved convergence guarantees, even under mis-specification, suggesting
practical value to cleaning data when correctly modeling uncertainty is
impossible.

\paragraph{Fundamental limits of majority vote}
Via local asymptotic minimax theory, Section~\ref{sec:lower-bound} develops
fundamental limits for estimation given only majority vote labels $(X,
\majY)$, which coincide with our results in
Section~\ref{sec:well-sepcified-mod}.
These results highlight how cleaned labels force slower
convergence in the number of labelers $m$ while reinforcing the claims of
robustness.
In particular, while a mis-specified link function introduces structural
bias that MLE-type estimators using all labels cannot overcome, majority
vote-based estimators achieve rate-optimal convergence for procedures
receiving cleaned labels $\majY$ (making the rate $r(m)$ in the bottom right
of Table~\ref{table:comparisons} sharp).




\paragraph{Experimental evaluation}
Our theoretical results predict two outcomes: (i) \emph{if} the
model has fidelity to the labeling process, then using all noisy labels
should yield better estimates than procedures using cleaned labels, and
(ii), that majority vote estimators should be more robust. In
Section~\ref{sec:experiments}, we thus provide real and semi-synthetic
experiments that corroborate the theoretical predictions; the experiments
also highlight a need, which some researchers are now beginning to
address~\cite{GadreEtAl23}, for more applied dataset development so that we
can evaluate the consequences of filtering and otherwise manipulating the
data we use.

\paragraph{Semi-parametric approaches}
In Section~\ref{sec:semi-param}, we develop an exemplar
approach for estimation via semi-parametric approaches, where
one estimates labelers' links and reliability.
While not our main focus, the results provide potential insights into
estimation using existing crowdsourcing techniques, where black box
algorithms provide measures of labeler reliability that one may incorporate
to achieve rate-optimal estimation.

\subsection{Related work}
\label{sec:related-work}

We briefly overview related work, acknowledging that to make our theoretical
model tractable, we capture only a few of the complexities inherent in
dataset construction. Label aggregation strategies often attempt to
evaluate labeler uncertainty, dating to~\citet{DawidSk79}, who study
labeler uncertainty estimation to overcome noise in
clinical patient measurements. With the rise of
crowdsourcing, such reliability models have
attracted substantial recent interest, with approaches for optimal budget
allocation~\cite{KargerOhSh14}, addressing untrustworthy or malicious
labelers~\cite{CharikarStVa17}, and more broadly an intensive line of work
studying crowd labeling and aggregation~\cite{WhitehillWuBeMoRu09,
  WelinderBrPeBe10, TianZh15, RatnerBaEhFrWuRe17}, with substantial
applications~\cite{DengDoSoLiLiFe09, RussakovskyDeSuKrSaMaHuKaKhBeBeFe15}.

The focus in many of these, however, is to obtain a single  clean
and trustworthy label for each example. Thus, while these
aggregation techniques have been successful and important, our work takes a
different perspective.  First, we target statistical analysis for
the full learning pipeline---to understand the theoretical landscape of the
learning problem with multiple labels for each example---as opposed to
obtaining only clean labels. Moreover, we argue for an increased focus on
calibration of the resulting predictive model, which aggregated (clean)
labels necessarily impede. We therefore adopt a perspective similar to
\citeauthor{PetersonBaGrRu19} and \citeauthor{PlataniosAlXiMi20}'s
applied work~\citep{PetersonBaGrRu19,PlataniosAlXiMi20},
which highlight ways that incorporating human
uncertainty into learning pipelines can make classification ``more
robust''~\citep{PetersonBaGrRu19}.  Instead of splitting the learning
procedure into two phases, where the first aggregates labels and the second
trains, we simply use non-aggregated labels throughout
learning. \citet{PlataniosAlXiMi20} propose a richer model than our stylized
scenarios, directly modeling labeler reliability, but the
simpler approaches we investigate allow us to be theoretically precise about
the limiting behavior and performance of the methods.

Our results complement a new strain of experimental and applied work in
machine learning.  The Neural Information Processing Systems conference, as
of 2021, includes a Datasets and Benchmarks track, recognizing that they
``are crucial for the development of machine learning
methods''~\cite{BeygelzimerLiVaDa21}. \citet{GadreEtAl23} introduce
DataComp, a testbed for experimental work on datasets, building off the long
history of dataset-driven development in machine learning~\cite{HardtRe22}.
We take a more theoretical approach, attempting to lay statistical
foundations for this line of work.

We mention in passing that our precise consistency results rely on
distributional assumptions on the covariates $X$, for example, that they are
Gaussian. That such technical conditions appear may be familiar from work,
for example, in single-index or nonlinear measurement models~\cite{PlanVe16,
  PlanVe13}. In such settings, one assumes $\E[Y \mid X] = f(\<\theta\opt,
X\>)$ for an unknown increasing $f$, and it is essential that $\E[Y X]
\propto \theta\opt$ to allow estimation of $\theta\opt$; we leverage similar
techniques.

\paragraph{Notation}
We use $\norm{x}_p$ to denote the $\ell_p$ norm of a vector $x$. For a
matrix $M$, $\norm{M}$ is its spectral norm, and $M^\dagger$ is its
Moore-Penrose pseudo inverse. For a unit vector $u \in \R^d$, the projection
operator to the orthogonal space of $\mathrm{span}\{u\}$ is $\proj_{u}^\perp
= I - u u^\top$. We use the notation $f(n) \asymp g(n)$ for $n \in
\mathbb{N}$ and $f(x) \asymp g(x)$ for $x \in \R_+$ if there exist numerical
constants $0 < c_1, c_2 < \infty$ and $n_0, x_0 \geq 0$ such that $c_1|g(n)| \leq |f(n)|
\leq c_2|g(n)|$ and $c_1|g(x)| \leq |f(x)| \leq c_2|g(x)|$ for $n \geq n_0$
and $x \geq x_0$. We also use the empirical process notation $\Prb Z = \int
z d\Prb(z)$, and $\cd$ denotes convergence in distribution.
We let $c=o_m(1)$ denote that $c \to 0$ as $m \to \infty$.

\section{The well-specified logistic model}\label{sec:well-sepcified-mod}

We begin in a setting that allows the cleanest and most precise comparisons
between a method using aggregated labels and one without by considering the
logistic model for the link~\eqref{eq:general-calibration-model},
\begin{align*}
  \sigma\lr(t) = \frac{1}{1 + e^{-t}} \in \fsymlink.
\end{align*}
This simplicity allows results that highlight many of the conclusions we
draw, and so we present initial, relatively clean, results for the
estimators~\eqref{eq:maximum-likelihood} and~\eqref{eq:logistic-majority}
here.  In particular, we assume identical links $\sigma_1\opt = \dots =
\sigma_m\opt = \sigma\lr$, have an i.i.d.\ sample $X_1, \ldots, X_n$, where
for each $i$ we draw $(Y_{i1}, \ldots, Y_{im}) \simiid
\Prb_{\sigma\lr,\theta\opt}(\cdot \mid X_i)$ for a true vector $\theta\opt$.

\subsection{The isotropic Gaussian case} \label{sec:gauss}

To deliver the general taste of our results on the performance of the full
information~\eqref{eq:maximum-likelihood} and majority
vote~\eqref{eq:logistic-majority} approaches, we start by studying the
simplest case when $X \sim \normal(0, I_d)$.

\paragraph{Performance with non-aggregated data.}

The
non-aggregated MLE estimator $\tmle$ in Eq.~\eqref{eq:maximum-likelihood}
admits a standard analysis~\cite{VanDerVaart98},
which we state as a corollary of a result where $X$ has more general
distribution in Proposition~\ref{prop:lowd-maximum-likelihood} to come.
%
\begin{corollary}
  \label{cor:lowd-maximum-likelihood-Gaussian} 
  Let $X \sim \normal(0, I_d)$ and $t\opt = \ltwos{\theta\opt}$. The maximum
  likelihood estimator $\tmle$ is consistent, with $\tmle \cp
  \theta\opt$, and for $\proj_{u\opt}^\perp = I_d - u\opt {u\opt}^\top$,
  \begin{align*}
    \sqrt{n} (\what{u}\lr_{n,m} - u^\star)
    & \cd \normal\left(0,  m^{-1}
    \cdot (t\opt)^{-2} \cdot \E[\sigma\lr(\<\theta\opt, X\>)
      (1 - \sigma\lr(\<\theta\opt, X\>))]^{-1}
    \proj_{u\opt}^\perp \right).
  \end{align*}
\end{corollary}

The first part of \Cref{cor:lowd-maximum-likelihood-Gaussian} demonstrates
that the non-aggregated MLE classifier is calibrated: it recovers both the
direction and scale of $\theta\opt$. Moreover, the second part shows that
this classifier enjoys convergence rates that roughly scale as
$O(1)/\sqrt{nm}$, so that a linear increase in the label size $m$ roughly
yields a linear increase in convergence rate.

\paragraph{Performance with majority-vote aggregation.}

The analysis of the majority vote estimator~\eqref{eq:logistic-majority}
requires more care, though the assumption that $X \sim \normal(0, I_d)$ allows
us to calculate the limits explicitly. In
brief, we show that when $X$ is Gaussian, the estimator is not calibrated
and has slower convergence rates in $m$ for classification error than the
non-aggregated classifier.
The basic idea is that a classifier fit using
majority vote labels $\majY_i$
should still point in the direction of $\theta\opt$,
but it should be (roughly) ``calibrated'' to the probability of a majority
of $m$ labels being correct.

We follow this idea and sketch the derivation here, as it is central to all
of our coming theorems, and then state the companion corollary to
Corollary~\ref{cor:lowd-maximum-likelihood-Gaussian}.  Each result depends
on the probability
\begin{equation*}
  \Prb(\majY = \sgn(\<x, \theta^\star\>) \mid x)
  = \rho_m(|\<x, \theta^\star\>|)
\end{equation*}
of obtaining a correct label using majority vote,
where $\rho_m$ is the binomial probability function
\begin{align}
  \rho_m(t) =
  \P\left(\bindist\Big(m, \frac{1}{1 + e^{-|t|}}\Big)\ge \frac{m}{2}\right)
  = 
  \sum_{i=\lceil m/2 \rceil}^m \binom{m}{i} \left(\frac{1}{1 + e^{-|t|}} \right)^i \left(\frac{e^{-|t|}}{1 + e^{-|t|}} \right)^{m-i}, \label{eq:link-function-m-majority-vote}
\end{align}
when $m$ is odd. (When $m$ is even the final sum has the additional additive
term $\frac 1 2 \binom{m}{m/2} \frac{e^{-m|t|/2}}{(1 + e^{-|t|})^m}$, which
is asymptotically negligible.) Key to the coming result is choosing a
parameter to roughly equalize binomial (majority vote) and logistic
(Bernoulli) probabilities, and so for $Z \sim \normal(0, 1)$ we define the
function
\begin{align}
  h_m(t) &  = \Ep \brk{|Z|(1-\rho_m(t^\star |Z|))}-\Ep
  \brk{\frac{|Z|}{1 + e^{t|Z|}}}.
  \label{eqn:h-centering}
\end{align}
We use $h$ to find the minimizer of population loss $\poploss\mv_m(\theta) =
\Ep[ \ell_\theta\lr(\majY \mid X)]$ by considering the ansatz
that $\theta = t u^\star$ for
some $t > 0$. Using the definition~\eqref{eq:link-function-m-majority-vote}
of $\rho_m$, we can write
\begin{align*}
  \poploss\mv_m(\theta)  
  & = \Ep\left[\log(1 + \exp(-S \<X, \theta\>)) \cdot\rho_m(t^\star| \<X, u^\star\>|) \right]  + \Ep\left[\log(1 + \exp(S\<X, \theta\>)) \cdot (1-\rho_m(t^\star|\<X, u^\star\>|)) \right],
\end{align*}
where $S = \sgn(\<X, \theta^\star\>)$, and compute the gradient
\begin{align*}
  \nabla \poploss\mv_m(\theta) & = - \Ep \left[\frac{S}{1 + \exp(S \<X, \theta\>)} \rho_m(t^\star |\<X, u^\star\>|) X \right] + \Ep \left[\frac{S\exp(S \<X, \theta\>)}{1 + \exp(S \<X, \theta\>)} (1-\rho_m(t^\star|\<X, u^\star\>|)) X \right].
\end{align*}
We set $Z = \<X, u^\star\>$ and decompose $X$ into the independent sum $X =
(X - u^\star Z) + u^\star Z$. Substituting in $\theta = t u\opt$ yields
\begin{align}
  \nabla \poploss\mv_m(t u^\star) 
  & \stackrel{\mathrm{(i)}}{=} - \Ep \left[\frac{\sgn(Z)}{1 + \exp(t|Z|)} \rho_m(t^\star |Z|) X \right]  + \Ep \left[\frac{\sgn(Z)\exp(t|Z|)}{1 + \exp(t|Z|)} (1-\rho_m(t^\star |Z|)) X \right] \nonumber	\\
  &= - \Ep \left[\frac{\sgn(Z)Z}{1 + \exp(t|Z|)} \rho_m(t^\star |Z|)\right] u^\star  + \Ep \left[\frac{\sgn(Z)Z\exp(t|Z|)}{1 + \exp(t|Z|)} (1-\rho_m(t^\star |Z|))\right]  u^\star \nonumber \\
  & = \left(\Ep \left[\frac{|Z|  \exp(t|Z|)}{1 + \exp(t|Z|)}\right] - \Ep \left[|Z| \rho_m(t^\star |Z|) \right]\right) u^\star \nonumber \\
  & = \left(\Ep \brk{|Z|(1-\rho_m(t^\star  |Z|))}-\Ep \brk{\frac{|Z|}{1 + e^{t|Z|}}}\right) u^\star = h_m(t) u^\star  , \label{eq:majority-vote-root}
\end{align} 
where in (i) we substitute $S = \sgn(Z)$. As we will present in
Corollary~\ref{cor:lowd-majority-vote-m}, $h_m(t)=0$ has a unique solution
$t_m \asymp \sqrt m$, and the global minimizer of the population loss
$\poploss\mv_m$ is thus exactly $t_m u\opt$.


By completing the calculations for the precise value of
$t_m$ above and performing a few asymptotic normality calculations,
we have the following result, a 
special case of Proposition~\ref{prop:lowd-majority-vote-m-noise} to come.

\begin{corollary} \label{cor:lowd-majority-vote-m}
  Let $X \sim \normal(0, I_d)$
  and $t\opt = \ltwos{\theta\opt}$.
  There are numerical constants $a, b > 0$ such that the following hold: for the function $h = h_m$ in~\eqref{eqn:h-centering},
  there is a unique $t_m \geq t_1 = t^\star$ solving $h(t_m) = 0$
  and 
  \begin{equation*}
    \what{\theta}\mv_{n,m} \cp t_m u\opt
    ~~ \mbox{and} ~~
    \lim_{m \to \infty} \frac{t_m}{t\opt \sqrt{m}} = a.
  \end{equation*}
  Moreover,
  there exists a function
  $C_m(t) = \frac{b}{t \sqrt{m}}(1 + o_m(1))$
  as $m \to \infty$ such that
  $\what{u}\mv_{n,m} = \mve/\ltwos{\mve}$ satisfies
  \begin{align*}
    \sqrt{n} \left(\what{u}\mv_{n,m} - u^\star \right) &\cd \normal\left(0, C_m(t\opt) {\proj_{u\opt}^\perp}\right).
  \end{align*}
\end{corollary}

It is instructive to compare the rates of this estimator to the rates
for the non-aggregated MLE
in \Cref{cor:lowd-maximum-likelihood-Gaussian}. First, the non-aggregated
estimator is calibrated in that $\tmle \to \theta\opt$, in contrast to the
majority-vote estimator, which roughly
``calibrates'' to the probability majority vote is correct
(cf.~\eqref{eq:majority-vote-root}) via the convergence
$\mve \to c \sqrt{m} \theta\opt$ as $n
\to \infty$.  The scaling of $C_m$ in
\Cref{cor:lowd-majority-vote-m} is also important: the majority-vote
estimator exhibits worse convergence rates by a factor of $\sqrt{m}$ than
the estimator $\what{\theta}\lr_{n,m}$: for constants $c\lr$ and $c\mv$ that
depend only on $t\opt = \ltwos{\theta\opt}$, we have asymptotic variances
differing by $\sqrt{m}$:
\begin{equation*}
	\sqrt{n}(\what{u}\lr_{n,m} - u\opt) \cd
	\normal\left(0, m^{-1} c\lr \proj_{u\opt}^\perp\right)
	~~ \mbox{while} ~~
	\sqrt{n}(\what{u}\mv_{n,m} - u\opt) \cd
	\normal\left(0, m^{-1/2} c\mv \proj_{u\opt}^\perp\right).
\end{equation*}

To obtain intuition for this result via comparison with
Corollary~\ref{cor:lowd-maximum-likelihood-Gaussian}, consider the Fisher
Information $\E[\sigma\lr(\<\theta\opt, X\>) (1 - \sigma\lr(\<\theta\opt,
  X\>)) XX^\top]$ for logistic regression. Letting $Z\opt = \<\theta\opt,
X\>$ be the predicted margin, there is only ``information'' available for an
estimator when $\sigma\lr(Z\opt)$ is near $\half$, as otherwise
$\sigma\lr(Z\opt)(1 - \sigma\lr(Z\opt)) \approx 0$; under majority vote, we
analogously require $\P(\bindist(m, \sigma\lr(Z\opt)) > \frac{m}{2}) \approx
\half$, which in turn means that we need $\sigma\lr(Z\opt) \in \half \pm
O(\frac{1}{\sqrt{m}})$ by standard binomial concentration, or $Z\opt =
\<\theta\opt, X\> \in \pm O(\frac{1}{\sqrt{m}})$.  This occurs on about
$1/\sqrt{m}$ fraction of the sample, so of $n \cdot m$ total observations,
majority vote ``loses'' a factor $\sqrt{m}$.

\newcommand{\mt}{$\textsc{M}_\beta$}

\subsection{Comparisons for more general feature distributions}
\label{sec:non-gauss}

The key to the preceding results---and an indication that they are
stylized---is that the covariates $X$ decompose into components aligned with
$\theta\opt$ and independent noise. Here, we abstract away the Gaussianity
assumptions to allow a more general and nuanced development carefully
tracking label noise, as margin and noise conditions play a
strong role in the relative merits of maximum-likelihood-type (full label
information) estimators versus those using cleaned majority-vote labels. The
results, as in Sec.~\ref{sec:gauss}, are consequences of the master theorems
to come in the sequel.

We first make our independence assumption.
\begin{assumption}
  \label{assumption:x-decomposition}
  The covariates $X$ have non-singular covariance $\Sigma$ and decompose as a
  sum of independent random vectors in the span of $u\opt$ and its
  complement
  \begin{equation*}
    X = W +  Z u^\star  ,
    ~~~\mbox{where}~~~
    W \perp\!\!\!\!\perp Z  , ~~
    \<W,u\opt\> = 0  ,
    ~ \Ep[W] = 0  ,
    ~ \Ep[Z] = 0  .
  \end{equation*}
\end{assumption}
Under these assumptions, we develop a characterization of the limiting
behavior of the majority vote and non-aggregated models based on
classification difficulty, adopting \citeauthor{MammenTs99}'s
perspective~\citep{MammenTs99}
and measuring difficulty of classification through the proximity of the
probability $\Prb(Y = 1 \mid X = x)$ to $1/2$.
Thus, for a \emph{noise exponent} $\beta \in (0, \infty)$, we
consider the condition
\begin{equation}
  \label{eqn:mammen-tsybakov}
  \Prb \left( \left|\Prb(Y = 1 \mid X) - \frac 1 2 \right| \leq \epsilon \right)
  = O(\epsilon^\beta)  .
  \tag{$\textsc{M}_\beta$}
\end{equation}
We see that as $\beta \uparrow \infty$ the problem becomes ``easier'' as it
is less likely to have a small margin---in particular, $\beta=\infty$ gives
a hard margin that $|\P(Y = 1 \mid X) - \half| \ge \epsilon$ for all small
$\epsilon$. Under the independent decomposition
Assumption~\ref{assumption:x-decomposition}, the noise
condition~\eqref{eqn:mammen-tsybakov} solely depends on the covariate's
projection onto the signal $Z$. We therefore consider the
following assumption on $Z$.
\begin{assumption}
  \label{assumption:z-density}
  For a given $\beta > 0$, $Z$ is \emph{$(\beta, c_Z)$-regular}, meaning
  that the absolute value $|Z|$ has density $p(z)$ on $(0, \infty)$, no
  point mass at $0$, and satisfies
  \begin{align*}
    \sup_{z \in (0, \infty)} z^{1 - \beta} p(z) < \infty,
    \qquad \lim_{z \to 0}z^{1 - \beta} p(z) = c_Z \in (0, \infty).
  \end{align*}
\end{assumption}

As the logistic function $\sigma\lr(t) = 1/(1+e^{-t})$ satisfies
${\sigma\lr}'(0) = 1/4$, for $t\opt = \ltwos{\theta\opt}$ in our logistic
model~\eqref{eq:general-calibration-model} we have $\Prb(Y = 1 \mid X = W +
u\opt Z, Z = z) = \sigma\lr(t\opt z) = 1/(1 + e^{- t\opt z})$. More
generally, for any link function $\sigma$ differentiable at $0$ with
$\sigma'(0) > 0$, we have $\Prb_\sigma(Y = 1 \mid Z = z) = \sigma(t\opt z) =
\half + \sigma'(0) t\opt z + o(t\opt z)$, so that the Mammen-Tsybakov noise
condition~\eqref{eqn:mammen-tsybakov} is equivalent to
\begin{align*}
  \Prb \left( \left|t^\star Z \right| \leq \epsilon \right) = O(\epsilon^\beta).
\end{align*}
Thus,
under Assumption~\ref{assumption:z-density},
condition~\eqref{eqn:mammen-tsybakov} holds, as
by dominated convergence we have
\begin{align*}
  \Prb\left(|t\opt Z| \le \epsilon\right)
  = \int_0^{\epsilon/t\opt} p(z) dz
  = \int_0^{\epsilon/t\opt} c_Z(1 + o_\epsilon(1)) z^{\beta - 1} dz
  = \frac{c_Z}{\beta} \epsilon^\beta \cdot (1 + o_\epsilon(1)),
\end{align*}
As a concrete case, when the features $X$ are isotropic Gaussian and so $Z
\sim \normal(0, 1)$, $\beta=1$.  We provide extensions of
Corollaries~\ref{cor:lowd-maximum-likelihood-Gaussian}
and~\ref{cor:lowd-majority-vote-m} in the more general cases the noise
exponent $\beta$ allows.


The maximum likelihood estimator retains its convergence
guarantees in this setting, and we can  be more precise for the
analogue of the final claim of Corollary~\ref{cor:lowd-maximum-likelihood-Gaussian} (see
Appendix~\ref{proof:lowd-maximum-likelihood} for a proof):
\begin{proposition}
  \label{prop:lowd-maximum-likelihood}
  Let Assumptions~\ref{assumption:x-decomposition} and
  \ref{assumption:z-density} hold for some $\beta  > 0$ and
  $t\opt = \ltwos{\theta\opt}$. Let $L\lr(\theta) = \Ep[
    \ell_\theta\lr(Y \mid X)]$ be the population logistic loss. Then the
  maximum likelihood estimator~\eqref{eq:maximum-likelihood} satisfies
	\begin{equation*}
		\sqrt{n} \left(\what{\theta}\lr_{n,m} - \theta\opt\right)
		\cd \normal(0, m^{-1}\nabla^2 L\lr(\theta\opt)^{-1})  . 
	\end{equation*}
	Moreover,
	\begin{equation*}
		\sqrt{n}\left(\what{u}\lr_{n,m} - u\opt\right)
		\cd \normal\left(0, m^{-1} \cdot (t\opt)^{-2} \proj_{u\opt}^\perp
		\nabla^2 \poploss\lr(\theta\opt)^{-1} \proj_{u\opt}^\perp\right), 
	\end{equation*}
	and there exist $C(t)$ and constant $b > 0$ such that
	$C(t) = \frac{b}{t^{2 - \beta}}(1 + o_t(1))$ and
	\begin{equation*}
		\sqrt{n} \left(\what{u}\lr_{n,m} - u\opt\right)
		\cd \normal\left(0, m^{-1} \cdot C(t\opt) \prn{\proj_{u\opt}^\perp \Sigma \proj_{u\opt}^\perp}^\dagger \right).  
	\end{equation*}
\end{proposition}

For majority-vote aggregation, we can in turn generalize
Corollary~\ref{cor:lowd-majority-vote-m}. In this case we still have $t_m \asymp \sqrt{m}$. However, the interesting factor here is that
the convergence rate now depends on the noise exponent $\beta$.

\begin{proposition} \label{prop:lowd-majority-vote-m-noise}
	Let
	Assumptions~\ref{assumption:x-decomposition} and \ref{assumption:z-density} hold for some $\beta \in (0,
	\infty)$, and $t\opt = \ltwos{\theta\opt}$. Suppose $h = h_m$ is the
	function~\eqref{eqn:h-centering} with $Z$ defined in Assumption~\ref{assumption:x-decomposition}.  There are constants $a, b > 0$,
	depending only on $\beta$ and $c_Z$, such that the following hold: there is a unique
	$t_m \geq t_1 = t^\star$ solving $h(t_m) = 0$ and for this $t_m$ we have
	both
	\begin{equation*}
		\what{\theta}\mv_{n,m} \cp t_m u\opt
		~~ \mbox{and} ~~
		\lim_{m \to \infty} \frac{t_m}{t\opt \sqrt{m}} = a.
	\end{equation*}
	Moreover,
	there exists a function
	$C_m(t) = \frac{b}{(t \sqrt m)^{2 - \beta}}(1 + o_m(1))$
	as $m \to \infty$ such that
	\begin{align*}
		\sqrt{n} \left(\what{u}\mv_{n,m} - u^\star \right)
		&\cd \normal\left(0, C_m(t\opt) \prn{\proj_{u\opt} \Sigma \proj_{u\opt}}^\dagger\right).
	\end{align*}
\end{proposition}
\noindent
We defer the proof to Appendix~\ref{proof:lowd-majority-vote-m-noise}.

Paralleling the discussion in Section~\ref{sec:gauss}, we may compare the
performance of the MLE $\what{\theta}\lr_{n,m}$, which uses all labels, and
the majority-vote estimator $\what{\theta}\mv_{n,m}$ using only the cleaned
labels. When the classification problem is hard---meaning that
$\beta$ in Condition~\eqref{eqn:mammen-tsybakov} is near 0 so that
classifying most examples is nearly random chance---we see
that the aggregation in the majority vote estimator still allows
convergence (nearly) as quickly as the non-aggregated estimator; the
problem is so noisy that data ``cleaning'' by aggregation is helpful.
Yet for easier problems, where $\beta \gg 0$, the gap between them
grows substantially; this is sensible, as aggregation is likely to
force a dataset to be separable, thus making fitting methods
unstable (and indeed, a minimizer may fail to exist). 



\section{Label aggregation and misspecified model}
\label{sec:mis-spec}

The logistic link provides clean interpretation and results, but we can
move beyond it to more realistic cases where labelers use
distinct links, although, to allow precise statements, we still assume the
same linear term $x \mapsto \<\theta\opt, x\>$ for each labeler's
generalized linear model. We study generalizations of the maximum
likelihood and majority vote estimators~\eqref{eq:maximum-likelihood}
and~\eqref{eq:logistic-majority}, highlighting dependence on
link fidelity.
In this setting, there are $m$ (unknown and possibly distinct)
link functions $\sigma_i\opt$, $i=1,2, \ldots, m$.
We show that the majority-vote estimator $\mve$ enjoys better robustness to model mis-specification than the non-aggregated estimator $\tmle$,
though both use identical losses.
In particular, our main result in this section implies
\begin{equation*}
\sqrt{n}(\what{u}\lr_{n,m} - u\opt) \cd
\normal\left(0, c\lr \Sigma\right)
~~ \mbox{while} ~~
\sqrt{n}(\what{u}\mv_{n,m} - u\opt) \cd
\normal\left(0, m^{-1/2} c\mv \Sigma\right),
\end{equation*}
where $c\lr$ and $c\mv$ are constants that depend only on the links $\sigma$,
$t\opt =
\ltwos{\theta\opt}$, and $\Sigma = I - u\opt {u\opt}^\top$ when $X \sim
\normal(0, I_d)$.  In contrast to the previous section, the majority-vote
estimator enjoys roughly $\sqrt{m}$-faster rates than the non-aggregated
estimator, maintaining its (slow) improvement with $m$, which the
MLE loses to misspecification.

To set the stage for our results, we define the
general link-based loss
\begin{equation*}
  \loss_{\sigma,\theta} (y \mid x) \defeq -\int_0^{y \<\theta, x\>} \sigma(-v) dv.
\end{equation*}
We then consider the general multi-label estimator and the majority-vote
estimator based on the loss $\loss_{\sigma,\theta}$,
\begin{equation}
  \label{eqn:general-estimators}
  \gme(\sigma) \defeq \argmin_\theta \Prb_{n,m} \loss_{\sigma, \theta},
  \qquad
  \mve(\sigma) \defeq \argmin_\theta \wb{\Prb}_{n,m} \loss_{\sigma, \theta}.
\end{equation}
When $\sigma = \sigma\lr$ is the logistic link, we recover the logistic loss
$\loss_{\sigma\lr,\theta} (y \mid x) = \loss_\theta\lr(y \mid x)$, and thus
we recover the results in~\Cref{sec:well-sepcified-mod}.  For both the
estimators, we suppress the dependence on the link $\sigma$
to write $\gme, \mve$ when the context is clear.


\subsection{Master results} \label{sec:master-results}

To characterize the behavior of multiple label estimators versus majority
vote, we provide master results as a foundation
for our convergence rate analyses throughout. By
a bit of notational chicanery, we consider both the cases that $\majY$ is a
majority vote and that we use multiple (non-aggregated) labels
simultaneously. In the case that the estimator uses the majority vote
$\majY$, let
\begin{equation*}
  \varphi_m(t) = \rho_m(t) \ind \{t \geq 0\} + (1 - \rho_m(t)) \ind\{t < 0\},
  ~~ \mbox{where} ~~
  \rho_m(t) \defeq
  \P(\majY = \sgn(\<X, \theta\opt\>) \mid \<X, \theta\opt\> = t),
\end{equation*}
and in the case that the estimator uses each label
from the $m$ labelers, let
\begin{equation*}
	\varphi_m(t) = \frac{1}{m} \sum_{j = 1}^m \sigma_j\opt(t)
	= \frac{1}{m} \sum_{j = 1}^m \P(Y_j = 1 \mid \<X, \theta\opt\> = t).
\end{equation*}
In either case, we then see that the population loss with
the link-based loss $\loss_{\sigma,\theta}$ becomes
\begin{equation}
	\label{eqn:generic-loss}
	L(\theta, \sigma) = \E\left[\loss_{\sigma,\theta}(1 \mid X)
	\varphi_m(\<X, \theta\opt\>)
	+ \loss_{\sigma,\theta}(-1 \mid X)
	(1 - \varphi_m(\<X, \theta\opt\>))\right],
\end{equation}
where we have taken a conditional expectation given $X$. We assume
Assumption~\ref{assumption:x-decomposition} holds, that $X$ decomposes into
the independent sum $X= Zu\opt + W$ with $W \perp u\opt$, and the true link
functions $\sigma_j\opt \in \Flink$. We further impose the following
assumption for the model link.
\begin{assumption}
	\label{assumption:model-link}
	For each sign $s \in \{-1, 1\}$,
	the model link function $\sigma$ satisfies
	$\lim_{t \to \infty} \sigma(s \cdot t) = 1/2 + sc$ for a constant
	$0 < c \le 1/2$ and is a.e.\ differentiable.
\end{assumption}
\paragraph{Minimizer of the population loss.}
We begin by characterizing---at a somewhat abstract level---the (unique)
solutions to the problem of minimizing the population
loss~\eqref{eqn:generic-loss}.  To characterize the minimizer $\theta\opt_L
\defeq \argmin_\theta L(\theta, \sigma)$, we hypothesize that it aligns with $u\opt
= \theta\opt / \ltwo{\theta\opt}$, using the familiar
ansatz that $\theta$ has
the form $\theta = t u\opt$. Using the formulation~\eqref{eqn:generic-loss},
we see that for $t\opt \defeq \ltwo{\theta\opt}$,
\begin{align}
	\nabla L(\theta, \sigma) &
	= -\E[\sigma(-\<\theta, X\>) X
	\varphi_m(\<X, \theta\opt\>)]
	+ \E[\sigma(\<\theta, X\>) X (1 - \varphi_m(\<X, \theta\opt\>))]
	\nonumber \\
	& = -\E[\sigma(- t Z) X \varphi_m(t\opt   Z)]
	+  \E[\sigma(tZ) X (1 - \varphi_m(t\opt Z))] \nonumber \\
	& = \left(-\E[\sigma(- tZ) Z \varphi_m(t\opt Z)] 
	+ \E[\sigma(tZ) Z (1 - \varphi_m(t\opt Z))] \right) u\opt
	= h_m(t) u\opt, \label{eqn:generic-gradient}
\end{align}
where the final line uses the decomposition $X = Z u\opt + W$ for the random
vector $W \perp u\opt$ independent of $Z$, and we recall
expression~\eqref{eqn:h-centering} to define the
\emph{calibration function}
\begin{equation}
  \label{eqn:h-m-func}
  h_{t\opt, m}(t) \defeq \E[\sigma(tZ) Z (1 - \varphi_m(t\opt Z))]
  - \E[\sigma(-t Z) Z \varphi_m(t\opt Z)].
\end{equation}
The function $h$ measures the ``calibration'' gap between the
hypothesized link function $\sigma$ and the label probabilities
$\varphi_m$.
If we presume that a solution to $h_{t\opt,m}(t) = 0$ exists, then
evidently $t u\opt$ is \emph{a} minimizer of $L(\theta, \sigma)$. In fact,
such a solution exists and is unique
(see Appendix~\ref{sec:proof-generic-solutions} for a proof):

\newcommand{\Hessian}{H_L}  
\newcommand{\hessfunc}{\mathsf{he}}  
\newcommand{\linkerr}{\mathsf{le}}  
\newcommand{\varfunc}{C_m}  

\begin{lemma}
  \label{lemma:generic-solutions}
  Let Assumption~\ref{assumption:x-decomposition} hold and $h = h_{t\opt,m}$
  be the gap function~\eqref{eqn:h-m-func}. Then there is a unique solution
  $t_m > 0$ to $h(t) = 0$, and the generic loss~\eqref{eqn:generic-loss} has
  unique minimizer $\theta\opt_L = t_m u\opt$.  Define the matrix
  \begin{equation}
    \begin{split}
      \Hessian(t) &
      \defeq \E[(\sigma'(-tZ) \varphi_m(t\opt Z) + \sigma'(tZ)
	(1 - \varphi_m(t\opt Z))) Z^2 ] u\opt {u\opt}^\top \\
      & \qquad
      ~ + \E[\sigma'(-tZ) \varphi_m(t\opt Z)
	+ \sigma'(tZ) (1 - \varphi_m(t\opt Z))]
      \projperp \Sigma \projperp.
    \end{split}
    \label{eqn:parameterized-hessian}
  \end{equation}
  Then the Hessian is $\nabla^2 L(\theta\opt_L, \sigma) = \Hessian(t_m)$.
\end{lemma}

\paragraph{Asymptotic normality with multiple labels.} With the existence of minimizers assured, we turn to their asymptotics.
For each of these, we require slightly different calculations, as the
resulting covariances are slightly different.
To state the result when we have multiple labels,
we define the average link function $\wb{\sigma}\opt = \frac{1}{m} \sum_{j=1}^m \sigma_j\opt$ and the three functions
\begin{align}
	\linkerr(Z) & \defeq \sigma(t_m Z) (1 - \wb{\sigma}\opt(t\opt Z))
	- \sigma(-t_m Z) \wb{\sigma}\opt(t\opt Z), \nonumber \\
	\hessfunc(Z) & \defeq \sigma'(-t_m Z) \varphi_m(t\opt Z)
	+ \sigma'(t_m Z) (1 - \varphi_m(t\opt Z)),
	\label{eqn:error-functions}
	\\
	v_j(Z) & \defeq
	\sigma_j\opt(t\opt Z)(1 - \sigma_j\opt(t\opt Z))
	(\sigma(t_m Z) + \sigma(-t_m Z))^2.
	\nonumber
\end{align}
The first, the link error $\linkerr$, measures the mis-specification of the
link $\sigma$ relative to the average link $\wb{\sigma}\opt$. The second
function, $\hessfunc$, is a Hessian term, as $\Hessian(t_m) =
\E[\hessfunc(Z) Z^2]u\opt{u\opt}^\top + \E[\hessfunc(Z)] \projperp \Sigma
\projperp$, and the third is a variance term for each labeler $j$.
We have the following theorem, which we prove
in Appendix~\ref{sec:proof-multilabel-asymptotics}.
\begin{theorem}
  \label{theorem:multilabel-asymptotics}
  Let Assumptions~\ref{assumption:x-decomposition}
  and~\ref{assumption:model-link} hold, and let $\gme$ be the multilabel
  estimator~\eqref{eqn:general-estimators}. Define the shorthand $\wb{v} =
  \frac{1}{m} \sum_{j = 1}^m v_j$. Then $\gme \cas \theta_L\opt$,
  and
  \begin{equation*}
    \sqrt{n} (\what{\theta}_n - \theta_L\opt)
    \cd \normal\left(0,
    \frac{\E[\linkerr(Z)^2 Z^2] + m^{-1} \E[\wb{v}(Z) Z^2]}{
      \E[\hessfunc(Z) Z^2]^2} u\opt {u\opt}^\top
    + \frac{\E[\linkerr(Z)^2] + m^{-1}\E[\wb{v}(Z)]}{\E[\hessfunc(Z)]^2}
    \prn{\projperp \Sigma \projperp}^\dagger
    \right).
  \end{equation*}
  Additionally, if $\what{u}_n = \gme / \ltwos{\gme}$
  and $t_m$ is the unique zero of the gap function $h_{t\opt, m}(t) = 0$, then
  \begin{equation*}
    \sqrt{n}(\what{u}_n - u\opt)
    \cd
    \normal\left(0, \frac{1}{t_m^2}
    \frac{\E[\linkerr(Z)^2] + m^{-1}
      \E[\wb{v}(Z)]}{\E[\hessfunc(Z)]^2}
    \left(\projperp \Sigma \projperp\right)^\dagger\right).
  \end{equation*}
\end{theorem}

Theorem~\ref{theorem:multilabel-asymptotics} exhibits two dependencies: the
first on the link error terms $\E[\linkerr(Z)^2]$---essentially, a bias
term---and the second on the rescaled average variance $\frac{1}{m}
\E[\wb{v}(Z)]$. So the multi-label estimator recovers
an optimal $O(1/m)$ covariance if the link errors are negligible, but if
they are not, then it necessarily has $O(1)$ asymptotic covariance.
The next corollary highlights
how things simplify.
In the well-specified
case that $\sigma$ is symmetric and $\sigma = \wb{\sigma}\opt$, 
the zero of the gap
function~\eqref{eqn:h-m-func} is evidently $t_m = t\opt =
\ltwo{\theta\opt}$, the error term $\linkerr(Z) = 0$,
and $v_j(Z) = \sigma\opt(t\opt Z)(1 - \sigma\opt(t\opt Z))$, and by symmetry,
$\sigma'(t) = \sigma'(-t)$ so that
$\hessfunc(Z) = \sigma'(t\opt Z)$:
\begin{corollary}[The well-specified case] \label{cor:well-specified-normalized-cov}
	Let the conditions above hold. Then
	\begin{equation*}
		\sqrt{n}(\what{u}_n - u\opt)
		\cd \normal\left(0, \frac{1}{m} \cdot
		\frac{1}{\ltwo{\theta\opt}^2}
		\frac{\E[\sigma(t\opt Z) (1 - \sigma(t\opt Z))]}{\E[\sigma'(t\opt Z)]^2}
		\projperp \Sigma \projperp \right).
	\end{equation*}
\end{corollary}

\paragraph{Asymptotic normality with majority vote.}
When we use the majority vote estimators, the asymptotics differ: there is
no averaging to reduce variance as $m$ increases, even in a
``well-specified'' case. The asymptotic variance typically decreases
$m$ grows, but at a slower rate, roughly related to
the fraction of data where there is ``signal'', as we discuss
briefly following Corollary~\ref{cor:lowd-majority-vote-m}.
\begin{theorem}
  \label{theorem:majority-vote-asymptotics}
  Let Assumptions~\ref{assumption:x-decomposition}
  and~\ref{assumption:model-link} hold, and let $\what{\theta}_n = \mve$ be
  the general majority vote
  estimator~\eqref{eqn:general-estimators}. Let $t_m$ be
  the zero of the gap function~\eqref{eqn:h-m-func}, solving
  $h_{t\opt,m} (t) = 0$.  Then $\what{\theta}_n \cas
  \theta\opt_L = t_m u\opt$, and $\what{u}_n = \what{\theta}_n /
  \ltwos{\what{\theta}_n}$ satisfies
  \begin{equation*}
    \sqrt{n} \left(\what{u}_n - u\opt\right)
    \cd \normal\left(0, \frac{1}{t_m^2}
    \frac{\E[\sigma(-t_m |Z|)^2 \rho_m(t\opt Z) + \sigma(t_m |Z|)^2
	(1 - \rho_m(t\opt Z))]}{\E[\hessfunc(Z)]^2}
    \prn{\projperp \Sigma \projperp}^\dagger \right).
	\end{equation*}
\end{theorem}
\noindent
We defer the proof to Appendix~\ref{sec:proof-majority-vote-asymptotics}. In most cases, we will take the link function $\sigma$ to be symmetric,
so that $\sigma(t) = 1 - \sigma(-t)$, and thus $\sigma'(t) = \sigma'(-t)$,
so that $\hessfunc(z) = \sigma'(t_m z) \ge 0$. This simplifies the denominator
in Theorem~\ref{theorem:majority-vote-asymptotics} to
$\E[\sigma'(t_m Z)]^2$.
Written differently, we may define a (scalar) variance-characterizing
function $\varfunc$ implicitly as follows: let $t_m = t_m(t)$ be a
zero of $h_{t,m}(s) =
\E[\sigma(sZ) Z (1 - \varphi_m(t Z))] - \E[\sigma(-s Z) Z \varphi_m(t Z)]
= 0$ in $s$, that is, $h_{t,m}(t_m(t)) = 0$ so that $t_m$
is a function of the size $t$ (recall the gap~\eqref{eqn:h-m-func}), and
then define
\begin{equation}
  \label{eqn:generic-majority-covariance}
  \varfunc(t) \defeq \frac{1}{t_m^2} \frac{\E[\sigma(-t_m |Z|)^2
      \rho_m(t Z) + \sigma(t_m |Z|)^2 (1 - \rho_m(t Z))]}{
    \E[\sigma'(t_m Z)]^2}
\end{equation}
where $t_m = t_m(t)$ above is implicitly defined.  Then
\begin{equation*}
  \sqrt{n}\left(\what{u}_n - u\opt\right)
  \cd \normal\left(0,
  \varfunc(t\opt) \prn{\projperp \Sigma \projperp}^\dagger
  \right).
\end{equation*}
Each of our main results, including those on well-specified models
previously, then follows by characterizing the behavior of $\varfunc(t)$ in
the asymptotics as $m \to \infty$ and the scaling
of the solution norm $t_m = \ltwos{\theta_L\opt}$,
which the calibration gap~\eqref{eqn:h-m-func} determines.
The key is that the scaling with $m$ varies depending on
the fidelity of the model, behavior of the links $\sigma$, and
the noise exponent~\eqref{eqn:mammen-tsybakov}, and our coming
consequences of the master theorems~\ref{theorem:multilabel-asymptotics}
and~\ref{theorem:majority-vote-asymptotics} help to reify this scaling.


%

\subsection{Robustness to model mis-specification}
\label{sec:robustness-model-misspecification}

Having established the general convergence results for the multi-label
estimator $\gme$ and the majority vote estimator $\mve(\sigma)$, we further
explicate their performance when we have a mis-specified model---the link
$\sigma$ is incorrect---by leveraging
Theorems~\ref{theorem:multilabel-asymptotics}
and~\ref{theorem:majority-vote-asymptotics} to precisely characterize their
asymptotics and show the majority-vote estimator can be more robust to model
mis-specification.



\paragraph{Multi-label estimator.}
As our focus here is descriptive, to make interpretable statements about
the multi-label estimator $\gme$ in~\eqref{eqn:general-estimators}, we
simplify by assuming that each link $\sigma_j\opt \equiv \sigma\opt \in
\flink$ is identical.
Then an immediate corollary of Theorem~\ref{theorem:multilabel-asymptotics}
follows:
\begin{corollary}
  \label{corollary:lowd-misspcfd-logistic-regression}
  Let Assumptions~\ref{assumption:x-decomposition} and
  \ref{assumption:z-density} hold for some $\beta \in (0, \infty)$, and
  $t\opt = \ltwos{\theta\opt}$. Then the calibration
  gap function~\eqref{eqn:h-m-func} has unique positive zero
  $t_{\sigma\opt}$ such that
  $h_{t\opt,m}(t_{\sigma\opt}) = 0$, and the multilabel
  estimator~\eqref{eqn:general-estimators} satisfies
  \begin{equation*}
    \gme \cp  t_{\sigma\opt} u\opt.
  \end{equation*}
  For the mean $\wb{v} = \frac{1}{m} \sum_{j = 1}^m v_j$ of the
  variances~\eqref{eqn:error-functions}), the estimate
  $\what{u}_{n,m} = \gme / \ltwos{\gme}$ satisfies
  \begin{equation*}
    \sqrt{n} \left(\what{u}_{n,m} - u\opt\right)
    \cd \normal\left(0,
    \frac{\E[\linkerr(Z)^2]  + m^{-1} \E[\wb{v}(Z)]}{t_{\sigma\opt}^2
      \E[\hessfunc(Z)]^2}  
    \prn{\proj_{u\opt} \Sigma \proj_{u\opt}}^\dagger\right)  .
  \end{equation*}
\end{corollary}

So in this simplified case, the asymptotic covariance remains of constant
order in $m$ unless $\E[\linkerr(Z)^2] = 0$.  In contrast, as we now show,
the majority vote estimator exhibits more robustness; this is perhaps
expected, as Corollary~\ref{cor:lowd-majority-vote-m} shows that in the
logistic link case, which is \emph{a fortiori} misspecified for majority
vote labels, has covariance scaling as $1/\sqrt{m}$, though the
generality of the behavior and its distinction from
Corollary~\ref{corollary:lowd-misspcfd-logistic-regression}
is interesting.

\paragraph{Majority vote estimator.}
For the majority-vote estimator, we relax our assumptions and allow
$\sigma_j\opt$ to be different, showing how the broad conclusions
Corollary~\ref{corollary:lowd-misspcfd-logistic-regression} suggests
continue to hold in some generality: majority vote estimators achieve slower
convergence than well-specified (maximum likelihood) estimators using each
label, but exhibit more robustness.
To characterize the large $m$
behavior, we require the following regularity conditions on
the average link $\wb{\sigma}_m\opt = \frac{1}{m} \sum_{j=1}^m \sigma_j\opt$,
which we require has a limiting derivative at $0$.

\begin{assumption} \label{assumption:condition-misspcfd-link-majority-vote}
  For the sequence of link functions $\{\sigma_j \mid j \in \mathbb{N}\}
  \subset \flink$, let $\wb{\sigma}_m\opt = \frac{1}{m} \sum_{j=1}^m
  \sigma_j\opt$. There exists ${\wb{\sigma}\opt}'(0) >0$ such that
  \begin{subequations}
    \begin{align}
      \bm{\mathrm{(i)}} \qquad  & \lim_{m \to \infty} \sqrt{m} \prn{\wb{\sigma}_m\opt\prn{\frac{t}{\sqrt m}} - \frac 1 2} = {\wb{\sigma}\opt}'(0) t   , \quad \textrm{for each } t \in \R  ; \label{ass:misspcfd-link-1} \\
      \bm{\mathrm{(ii)}} \qquad &\liminf_{m \to \infty} \inf_{t \neq 0} \frac{\left|\wb{\sigma}_m\opt(t) - \frac 1 2 \right|}{\min \left\{|t|, 1\right\}} > 0  ; \label{ass:misspcfd-link-2} \\
      \bm{\mathrm{(iii)}} \qquad &\lim_{t \to 0} \sup_{j \in \N}
      \left|\sigma_j\opt(t) - \frac 1 2 \right| = 0  . \label{ass:misspcfd-link-3}
    \end{align}
  \end{subequations}
\end{assumption}
\noindent
These assumptions simplify if the links are identical:
if $\sigma_j\opt \equiv \sigma\opt$,
we only require $\sigma\opt$ is differentiable around
$0$ with ${\sigma\opt}'(0) >0$ and $|\sigma\opt(t) - \half|
\gtrsim \min\{|t|, 1\}$.


We can apply Theorem~\ref{theorem:majority-vote-asymptotics} to obtain
asymptotic normality for the majority vote
estimator~\eqref{eqn:general-estimators}. We recall the probability
\begin{equation}
  \label{eq:link-function-m-majority-vote-general}
  \rho_m(t) \defeq
  \Prb(\majY = \sgn(\<X, \theta^\star\>) \mid \<X,
  \theta\opt\> = t)
\end{equation}
of the majority vote being correct given the margin $\<X, \theta\opt\> = t$
and the calibration gap function~\eqref{eqn:h-m-func}, which by a
calculation resolves to the more convenient form
\begin{align*}
  h(t) & = h_{t\opt, m}(t) = \E[\sigma(t|Z|) |Z| (1 - \rho_m(t\opt Z))]
  - \E[\sigma(-t |Z|) |Z| \rho_m(t\opt Z)].
\end{align*}
The main technical challenge is to characterize the large $m$ behavior for
the asymptotic covariance function $\varfunc(t)$ defined implicitly in the
quantity~\eqref{eqn:generic-majority-covariance}.  We
postpone the details to Appendix~\ref{proof:lowd-misspcfd-majority-vote-m}
and state the result below, which is a consequence of
Theorem~\ref{theorem:majority-vote-asymptotics} and a careful asymptotic
expansion of the covariance
function~\eqref{eqn:generic-majority-covariance}.

\begin{proposition}
  \label{prop:lowd-misspcfd-majority-vote-m}
  Let Assumptions~\ref{assumption:x-decomposition} and
  \ref{assumption:z-density} hold for some $\beta \in (0, \infty)$ with
  $\int_0^\infty z^{\beta - 1} \sigma'(z) dz < \infty$
  and
  $t\opt = \ltwos{\theta\opt}$, and in addition that
  Assumption~\ref{assumption:condition-misspcfd-link-majority-vote} holds
  and $\sigma$ is symmetric.
  Then there are constants $a, b > 0$, depending only on $\beta$, $c_Z$,
  $\sigma$, and ${\wb{\sigma}\opt}'(0)$, such that there is a unique $t_m
  \geq t_1 = t^\star$ solving $h(t_m) = 0$, and for this $t_m$ we have both
  \begin{equation*}
    \what{\theta}\mv_{n,m} \cp t_m u\opt
    ~~ \mbox{and} ~~
    \lim_{m \to \infty} \frac{t_m}{t\opt \sqrt{m}} = a  .
  \end{equation*}
  Moreover, the covariance~\eqref{eqn:generic-majority-covariance} has the
  form $\varfunc(t) = \frac{b}{(t \sqrt m)^{2 - \beta}}(1 + o_m(1))$, and
  \begin{align*}
    \sqrt{n} \left(\what{u}\mv_{n,m} - u^\star \right)
    &\cd \normal\left(0, C_m(t\opt) \prn{\proj_{u\opt} \Sigma \proj_{u\opt}}^\dagger\right).
  \end{align*}
\end{proposition}

Proposition~\ref{prop:lowd-misspcfd-majority-vote-m} highlights the robustness
of the majority vote estimator: even when the link $\sigma$ is (more
or less) arbitrarily incorrect, the asymptotic covariance still
exhibits reasonable scaling. The noise parameter $\beta$ in
Assumption~\ref{assumption:z-density}, roughly equivalent to the
Mammen-Tsybakov noise exponent~\eqref{eqn:mammen-tsybakov}, also plays
an important role. In typical cases with $\beta = 1$ (e.g., when
$X \sim \normal(0, I_d)$), we see $\varfunc(t) \asymp \frac{1}{t \sqrt{m}}$.
In noisier cases, corresponding to $\beta \downarrow 0$, majority vote
provides substantial benefit approaching a well-specified model; conversely,
in ``easy'' cases where $\beta > 2$, majority vote estimators become
\emph{more} unstable, as they make the data (very nearly) separable,
which causes logistic-regression and other margin-type estimators to be
unstable~\cite{CandesSu19pnas}.


\newcommand{\fisher}{\mathsf{I}}
\newcommand{\fishermv}{\fisher\mv}

\section{Fundamental estimation limits of majority vote labels}
\label{sec:lower-bound}

We complement our convergence guarantees via lower bounds for labelers with
identical logistic links $\sigma_1\opt = \dots = \sigma_m\opt = \sigma\lr$.
We begin with a somewhat trivial observation that unless we know the number
of labels $m$ and link ahead of time, consistently estimating $\theta\opt$
from aggregated labels is impossible, which contrasts with algorithms using
non-aggregated data. Section~\ref{sec:lam-majority-vote} presents the more
precise local asymptotic efficiency limits for estimation from majority
labels $\majY$.

\subsection{Inconsistency from aggregate labels}
\label{sec:imp-mv}

%
To make clear the impossibility of estimation without some knowledge of the
link and number $m$ of labelers, we show that there exist distributions with
different linear classifiers $\theta\opt$ and $\wb{\theta}$ generating the
same data distribution.  Consider a generalized linear model for binary
classification, with link  $\sigma\opt \in \fsymlink$, $\theta\opt
\in \R^d$, define the model \eqref{eq:general-calibration-model} with
$\Prb_{\sigma\opt, \theta\opt}(Y = y \mid X = x) = \sigma\opt ( y \<
\theta\opt, x\>)$, and let $X_1, X_2, \ldots, X_n \simiid \normal(0, I_d)$,
where labels follow $Y_{ij}\simiid
\Prb_{\sigma\opt,\theta\opt}(\cdot \mid X_i)$, $j = 1, \ldots, m$;
$\majY_i = \maj(Y_{i1}, \ldots, Y_{im})$ denotes the majority-vote.
Then
there always exists a calibration function $\wb{\sigma}$ and label size
$\wb{m}$ such that $\majY_i$ has identical distribution under both the
model~\eqref{eq:general-calibration-model} and with $\wb{\sigma},
\wb{m}$ replacing $\sigma\opt$ and $m$.  Letting
$\P_{(X,\majY)}^{\sigma\opt, \theta\opt, m}$ denote the induced distribution
on $(X_i, \majY_i)$, we have the following formal statement, which we prove
in Appendix~\ref{proof:impossibility}.

\begin{proposition} \label{prop:impossibility}
  Suppose $\sigma\opt \in \fsymlink$ satisfies $\sigma\opt(t) > 0$ for all $ t > 0$. For any $\wb{\theta} \in \R^d$ such that $\wb{\theta}/\ltwo{\wb{\theta}} = \theta\opt/\ltwo{\theta\opt}$ and any positive integer $\wb{m}$, there is another link function $\wb{\sigma}$ such that 
  \begin{align*}
    \P_{(X,\majY)}^{\sigma\opt, \theta\opt, m} \stackrel{\textup{dist}}{=}
    \P_{(X,\majY)}^{\wb{\sigma}, \wb{\theta}, \wb{m}}.
  \end{align*}
\end{proposition}



\subsection{Fundamental asymptotic limits in the logistic case}
\label{sec:lam-majority-vote}

As most of our results repose on classical asymptotics, we provide lower
bounds in the same setting, using Hajek-Le Cam local asymptotic (minimax)
theory. We recall a typical result, which combines~\citet[Lemma
  6.6.5]{LeCamYa00} and \citet[Theorem 3.11.5]{VanDerVaartWe96}:
\begin{corollary}
  \label{corollary:lam}
  Let $\{\P_\theta\}$ be a family of distributions indexed by $\theta$
  with continuous Fisher Information $\fisher(\theta)$ in a neighborhood
  of $\theta\opt$.
  Then there exist
  probability densities $\pi_{c,n}$ supported on
  $\{\theta \in \R^d \mid \ltwo{\theta\opt - \theta} \le c/\sqrt{n}\}$
  such that for any quasi-convex, symmetric, and bounded
  loss $\varphi$,
  \begin{equation*}
    \liminf_{c \to \infty}
    \liminf_n \inf_{\what{\theta}_n}
    \int \E_\theta[\varphi(\sqrt{n}(\what{\theta}_n - \theta))]
    \pi_{c,n}(\theta) d\theta
    \ge \E[\varphi(W)],
  \end{equation*}
  where $W \sim \normal(0, \fisher(\theta\opt)^{-1})$.  If $T : \R^d \to
  \R^k$ is differentiable at $\theta\opt$ with derivative matrix
  $\dot{T}(\theta\opt) \in \R^{k \times d}$, then additionally
  \begin{equation*}
    \liminf_{c \to \infty}
    \liminf_n \inf_{\what{T}_n}
    \int \E_\theta[\varphi(\sqrt{n}(\what{T}_n - T(\theta)))]
    \pi_{c,n}(\theta) d\theta
    \ge \E[\varphi(\dot{T}(\theta\opt) W)].
  \end{equation*}
  The infima above are taken over all estimators $\what{\theta}_n$ and
  $\what{T}_n$ based on $n$ i.i.d.\ observations from $\P_\theta$.
\end{corollary}
\noindent
This corollary is the sense in which estimators, such as the MLE, achieving
asymptotic variance equal to the inverse Fisher Information are efficient.
It also makes clear that (except perhaps at a Lebesgue measure-zero set of
parameters $\theta$) any estimator $\what{\theta}_n$ with asymptotic
variance $\Sigma_\theta$ necessarily satisfies
$\Sigma_\theta \succeq \fisher(\theta)^{-1}$.

We therefore complete our theoretical analysis of fundamental limits by
giving the Fisher Information matrix for binary classification models with
aggregated majority vote data. We assume the links are identical,
$\sigma_1\opt = \dots = \sigma_m\opt = \sigma\lr$.  Recall (see
Prop.~\ref{prop:lowd-majority-vote-m-noise}) that in the case where $X$ is
normal, the majority vote estimator $\what{u}\mv_{n,m}$ has asymptotic
variance of order $m^{-\half}$. As we will present momentarily in
Theorem~\ref{thm:lower-bound-fisher-information}, $\what{u}\mv_{n,m}$ is
indeed rate optimal for classification as $m \to \infty$.

As usual, letting Assumption~\ref{assumption:x-decomposition} hold, we write
$X = Z u\opt + W$ with $ W \perp u\opt$. Let $t\opt = \ltwos{\theta\opt}$.
In this case, we may explicitly decompose the Fisher information.

\begin{lemma}
  \label{lem:fisher-information-decomposition}
  The Fisher information matrix $\fishermv_m(\theta)$ for the majority vote
  model $\{\P_{(X,\majY)}^{\sigma\lr, \theta, m}\}_{\theta \in
    \R^d}$ satisfies
  \begin{align*}
    \fishermv_m(\theta\opt) = \E \brk{\frac{\rho_m'(t\opt|Z|)^2
        Z^2}{\rho_m(t\opt|Z|) ( 1 - \rho_m(t\opt|Z|))}} \cdot u\opt
             {u\opt}^\top + \E \brk{\frac{\rho_m'(t\opt|Z|)^2
               }{\rho_m(t\opt|Z|) ( 1 - \rho_m(t\opt|Z|))}} \cdot
             \proj_{u\opt}^\perp \Sigma \proj_{u\opt}^\perp .
  \end{align*}
\end{lemma}
\begin{proof}
  For $p_\theta(x, y) = \Prb_\theta(\majY = y \mid X = x) q(x)$, where
  $q$ is the $d$-dimensional centered density of $X$,
  the Fisher information matrix at $\theta$ is
  \begin{align*}
    \fishermv_m(\theta) = \E \brk{\nabla_\theta \log p_\theta(X, \majY) \cdot \nabla_\theta \log p_\theta(X, \majY)^\top }.
  \end{align*}
  Substituting $\rho_m(|\<x, \theta^\star\>|) = \Prb(\majY =
  \sgn(\<x, \theta^\star\>) \mid X = x)$, we then have
	\begin{align*}
		& \E \brk{\nabla_\theta \log p_{\theta\opt}(X, \majY) \cdot  \nabla_\theta \log p_{\theta\opt}(X, \majY)^\top \mid X = x } \\
		& =  \rho_m(|\<x, \theta^\star\>|) \cdot \frac{\nabla_\theta \rho_m(|\<x, \theta^\star\>|) \cdot \nabla_\theta \rho_m(|\<x, \theta^\star\>|)^\top}{\rho_m(|\<x, \theta^\star\>|)^2} \nonumber \\
		& \qquad + (1- \rho_m(|\<x, \theta^\star\>|) \cdot \frac{\prn{\nabla_\theta(1 - \rho_m(|\<x, \theta^\star\>|)} \cdot \prn{\nabla_\theta(1 - \rho_m(|\<x, \theta^\star\>|)}^\top}{(1 - \rho_m(|\<x, \theta^\star\>|)^2} \nonumber \\
		& = \frac{\rho_m'(|\<x, \theta^\star\>|)^2}{\rho_m(|\<x, \theta^\star\>|) ( 1 - \rho_m(|\<x, \theta^\star\>|))} \cdot xx^\top
	\end{align*}
        as desired.
\end{proof}

By characterizing the large $m$ behavior of the Fisher information,
we can then apply Corollary~\ref{corollary:lam} to understand
the optimal efficiency of any estimator given aggregate (majority vote)
data. 

\begin{theorem}
  \label{thm:lower-bound-fisher-information}
  Let Assumption~\ref{assumption:x-decomposition} hold,
  Assumption~\ref{assumption:z-density} hold with
  constants $c_Z$ and $\beta$,
  and $t\opt = \ltwos{\theta\opt}$.
  Then defining
  the constants
  $a = \frac{c_Z}{ \pi} \int_0^\infty \frac{e^{-z^2} z^{\beta+1}}{\Phi(-z) \Phi(z)} dz$ \mbox{and} $b = \frac{c_Z}{4 \pi} \int_0^\infty \frac{e^{-z^2} z^{\beta-1}}{\Phi(-z) \Phi(z)} dz$,
  there are functions
  \begin{align*}
    A_m(t) = \frac{a}{t^2} \left(\frac{1}{t\sqrt{m}}\right)^{\beta}( 1 + o_m(1))
    ~~ \mbox{and} ~~
    B_m(t) = \frac{b}{t^2}\left(\frac{1}{t\sqrt{m}}\right)^{\beta-2} (1 + o_m(1))
  \end{align*}
  for which
  \begin{equation*}
    \fishermv_m(\theta\opt) = A_m (t\opt) \cdot u\opt {u\opt}^\top + B_m(t\opt) \cdot \proj_{u\opt}^\perp \Sigma \proj_{u\opt}^\perp.
  \end{equation*}
\end{theorem}
\noindent
See Appendix~\ref{proof:lower-bound-fisher-information} for the proof.

By combining Theorem~\ref{thm:lower-bound-fisher-information}
with the classical local asymptotic minimax bounds, we can obtain
optimality for estimators of both $\theta\opt$ and the direction
$u\opt$ as $m$ gets large.
Computing the inverse Fisher information for $\theta\opt$ and
for the normalized quantity $u = \phi(\theta) \defeq \theta / \ltwo{\theta}$,
which satisfies $\dot{\phi}(\theta) = \ltwo{\theta}^{-1}
(I_d - \phi(\theta) \phi(\theta)^\top)$, we thus have the
efficient limiting covariances
\begin{align*}
  \Sigma_{\theta\opt} & \defeq
  \fishermv_m(\theta\opt)^{-1}  \asymp {t\opt}^{\beta +2} m^{\frac{\beta}{2}} \cdot u\opt {u\opt}^\top +  {t\opt}^{\beta}m^{\frac{\beta-2}{2}}  \cdot \prn{\proj_{u\opt}^\perp \Sigma \proj_{u\opt}^\perp}^\dagger, \\
  \Sigma_{u\opt} & \defeq \dot{\phi}(\theta\opt)
  \fishermv_m(\theta\opt)^{-1} \dot{\phi}(\theta\opt)
  \asymp {t\opt}^{\beta-2 }m^{\frac{\beta-2}{2}}  \cdot \prn{\proj_{u\opt}^\perp \Sigma \proj_{u\opt}^\perp}^\dagger,
\end{align*}
where we recall the shorthand $t\opt = \ltwo{\theta\opt}$.
(Here we use that if $A, B$ are symmetric matrices with $AB = BA = 0$,
then $(A + B)^{-1} = A^\dagger + B^\dagger$; see Lemma~\ref{lem:Sigma-inverse}
in Appendix~\ref{sec:technical}.)
Then combining Theorem~\ref{thm:lower-bound-fisher-information} with
Corollary~\ref{corollary:lam} yields the following efficiency
lower bound. (Other lower bounds arising from, e.g., the convolution
and variants of the local asymptotic minimax theorem are possible;
we present only this to give the flavor of such results.)
\begin{corollary} \label{cor:LAM-lower-bounds}
  Let $\what{\theta}_n$ and $\what{u}_n$
  be arbitrary estimators of $\theta\opt$ and $u\opt = \phi(\theta\opt)$
  using the aggregated data $\{(X_i,
  \majY_i)\}_{i=1}^n$.
  For any bounded and symmetric quasi-convex loss $\varphi : \R^d \to \R$,
  \begin{equation*}
    \lim_{c \to \infty}
    \liminf_n \sup_{\ltwo{\theta - \theta\opt} \le c/\sqrt{n}}
    \E_\theta[\varphi(\sqrt{n}(\what{u}_n - \phi(\theta)))]
    \ge \E[\varphi(W)],
    ~~ \mbox{where} ~~
    W \sim \normal(0, \Sigma_{u\opt}).
  \end{equation*}
  If there exist random variables
  $T_\theta$ or $U_\theta$
  such that $\sqrt{n}(\what{\theta}_n - \theta\opt) \cd T_{\theta\opt}$
  (respectively, $\sqrt{n}(\what{u}_n - u\opt) \cd U_{\theta\opt}$),
  then
  $\cov(T_\theta) \succeq \Sigma_{\theta\opt}$ and $\cov(U_\theta) \succeq
  \Sigma_{u\opt}$ for Lebesgue almost every $\theta\opt \in \R^d$.
\end{corollary}

This result carries several implications. First, the scaling in both $t\opt
= \ltwo{\theta\opt}$ and $m$ for the asymptotic (achieved) covariance in
Proposition~\ref{prop:lowd-majority-vote-m-noise} is rate-optimal: we have
$C_m(t) \asymp t^{\beta - 2} m^{\frac{\beta - 2}{2}}$ for the majority vote
estimator $\mve$. Notably, this optimal scaling holds even though the model
may be mis-specified: estimators using the aggregated majority vote labels
$\majY$ are rate-optimal in the scale $\ltwo{\theta\opt}$ and the number of
labelers $m$ (and sample size $n$) no matter the margin-based loss.
(Achieving the optimal constant requires a well-specified loss, of course.)
This robustness of estimators using aggregated data may motivate the use of
strong data-cleaning measures in dataset construction when it is difficult
to model the particulars of the data generation process.

A perhaps less intuitive result is that for any noise scale $\beta > 0$ on
the margin variable $Z = \<X, u\opt\>$ arising from
Assumption~\ref{assumption:z-density}, the inverse information
$\Sigma_{\theta\opt} \to \infty$ as $m \to \infty$, meaning calibration
becomes fundamentally harder for estimators using only aggregated labels,
even with a (known) logistic link $\sigma\lr$. For example, in the Gaussian
case when $\beta = 1$, then $\Sigma_{\theta\opt} = \Omega(m^{\half})$ and
$\Sigma_{u\opt} = O(m^{-\half})$, matching the scaling in
Corollary~\ref{cor:lowd-majority-vote-m}.  The same phenomenon holds even in
the ``very low noise'' regime where $\beta > 2$, meaning that the margin
variable $Z$ is typically far from 0 (recall the
exponent~\eqref{eqn:mammen-tsybakov}). In this case, $\Sigma_{u\opt} \to
\infty$ as $m \to \infty$, as a paucity of data align well with the true
margin $u\opt$; estimators (nearly) perfectly predict on training data and
become non-robust.


\providecommand{\lipconst}{\mathsf{L}}
\newcommand{\uinit}{u_n^{\textup{init}}}
\newcommand{\uerror}{\epsilon}

\section{Semi-parametric approaches}
\label{sec:semi-param}

The preceding analysis highlights distinctions between a fuller
likelihood-based approach---which uses all the labels, as
in~\eqref{eq:maximum-likelihood}---and the robust but somewhat slower rates
that majority vote estimators enjoy (as in
Proposition~\ref{prop:lowd-misspcfd-majority-vote-m}).  That full-label
estimators' performance so strongly depends on the fidelity of the link
(recall Corollary~\ref{corollary:lowd-misspcfd-logistic-regression})
suggests that we target estimators achieving the best of both worlds: learn
both a link function (or collection thereof) and refit the model using
\emph{all} the labels. While we mainly focus on efficiency bounds for
procedures we view as closer to practice---fitting a model based on
available data---in this section, we provide a few results targeting more
efficient estimation schemes through semiparametric estimation approaches.

As developing full semi-parametric efficiency bounds would unnecessarily
lengthen and dilute the paper, we develop a few relatively general and
simple convergence results into which we can essentially plug in
semiparametric estimators. In distinction from standard results in
semiparametric theory (e.g.~\cite[Ch.~25]{VanDerVaart98}
or~\cite{BickelKlRiWe98}), our results require little more than consistent
estimation of the links $\sigma_j\opt$ to recover $1/m$ (optimal) scaling in
the asymptotic covariance, as the special structure of our classification
problem allows more nuanced calculations; we assume each labeler (link
function) generates labels for each of the $n$ datapoints $X_1, \ldots,
X_n$, but we could relax the assumption at the expense of extraordinarily
cumbersome notation.  We give two example applications of the general
theory: the first (Sec.~\ref{sec:semiparametric}) analyzing a full pipeline
for a single index model setting, which robustly estimates direction $u\opt$,
the link $\sigma\opt$, and then re-estimates $\theta\opt$; the second
assuming a stylized black-box crowdsourcing mechanism that provides
estimates of labeler reliability, highlighting how even in some
crowdsourcing scenarios, there could be substantial advantages to using full
label information.

\subsection{Master results}
\label{sec:master-results-sem}

For our specializations, we first provide master results that
allow semi-parametric estimation of the link functions.  We consider
Lipschitz symmetric link functions, where for $\lipconst > 0$ we define
\begin{align*}
  \fsymlinkL := \{\sigma \mid \sigma \not \equiv 1/2\text{ is non-decreasing,
    symmetric, and} ~ \lipconst\text{-Lipschitz continuous} \}
  \subset \fsymlink.
\end{align*}
We consider the general case where there are $m$ distinct labeler link
functions $\sigma_1\opt, \ldots, \sigma_m\opt$. To eliminate ambiguity
in the links, we
assume the model is normalized, $\ltwo{\theta\opt} = 1$ so $\theta\opt
= u\opt$. To distinguish from the typical case,
we write $\vec{\sigma} = (\sigma_1, \ldots, \sigma_m)$, and for
$(x, y) \in \R^d \times \{\pm 1\}^m$ define
\begin{equation*}
  \loss_{\vec{\sigma}, \theta}(y \mid x)
  \defeq \frac{1}{m} \sum_{j = 1}^m \loss_{\sigma_j,\theta}(y_j \mid x),
\end{equation*}
which allows us to consider both the standard margin-based loss and the case
in which we learn separate measures of quality per labeler.
With this notation, we can then naturally define the population
loss
\begin{equation*}
	L(\theta, \vec{\sigma})
	\defeq \frac{1}{m} \sum_{j = 1}^m
	\E[\loss_{\sigma_j, \theta}(Y_j \mid X)].
\end{equation*}
For any sequence $\{\vec{\sigma}_n\} \subset (\fsymlinkL)^m$ of (estimated)
links
and data
$(X_i, Y_i)$ for $Y_i = (Y_{i1}, \ldots, Y_{im})$, we define the semi-parametric
estimator
\begin{align*}
  \spe \defeq \argmin_{\theta \in \R^d}
  \frac{1}{n} \sum_{i = 1}^n
  \loss_{\vec{\sigma}_n, \theta}(Y_i \mid X_i).
\end{align*}

We will demonstrate both consistency and
asymptotic normality under appropriate convergence and regularity
assumptions for the link functions. We assume there is a (semiparametric)
collection $\fsymlinksp \subset (\fsymlinkL)^m$ of link functions of
interest, which may coincide with $(\fsymlinkL)^m$ but may be smaller,
making estimation easier.  Define the distance $d_{\fsymlinksp}$ on $\R^d
\times \fsymlinksp$ by
\begin{align*}
  d_{\fsymlinksp}\left((\theta_1, \vec{\sigma}_1), (\theta_2, \vec{\sigma}_2)
  \right)
  \defeq
  \ltwo{\theta_1 - \theta_2}
  + \LtwoPbold{\vec{\sigma}_1(-Y\<X, u\opt\>)
    - \vec{\sigma}_2(-Y\<X, u\opt\>)}.
\end{align*}
We make the following assumption.
\begin{assumption}
  \label{assumption:sp-general}
  The links $\vec{\sigma}\opt \in \fsymlinksp$ are normalized so
  that $\P(Y_j = y \mid X = x) = \sigma\opt_j(y \<x, u\opt\>)$, and
  the sequence $\{\vec{\sigma}_n\} \subset \fsymlinksp$ is consistent:
  \begin{equation*}
    \LtwoPbold{\vec{\sigma}_n(-Y\<X, u\opt\>)
      - \vec{\sigma}\opt(-Y\<X, u\opt\>)}
    \cp 0.
  \end{equation*}
  Additionally, the mapping $(\theta, \vec{\sigma}) \mapsto
  \nabla_\theta^2 \poploss(\theta,
  \vec{\sigma})$ is continuous for $d_{\fsymlinksp}$ at $(u\opt,
  \vec{\sigma}\opt)$.
\end{assumption}
The continuity of $\nabla_\theta^2 \poploss(\theta, \vec{\sigma})$ at
$(u\opt, \vec{\sigma}\opt)$ allows us to develop local asymptotic normality.
To see that we may expect the assumption to hold, we give reasonably simple
conditions sufficient for it, including that the collection of links
$\fsymlinksp$ is sufficiently smooth or the data distribution is continuous
enough. (See Appendix~\ref{proof:d-continuity-guarantees} for a proof.)

\begin{lemma} \label{lem:d-continuity-guarantees}
  Let $d_{\fsymlinksp}$ be the distance in
  Assumption~\ref{assumption:sp-general}. Let
  Assumption~\ref{assumption:x-decomposition} hold, where $|Z| > 0$ with
  probability one, has nonzero and continuously differentiable density
  $p(z)$ on $(0, \infty)$ satisfying $\lim_{z \to s} z^2 p(z) = 0$ for $s
  \in \{0, \infty\}$. The mapping $(\theta, \vec{\sigma}) \mapsto
  \nabla_\theta^2 \poploss(\theta, \vec{\sigma})$ is continuous for
  $d_{\fsymlinksp}$ at $(u\opt, \vec{\sigma}\opt)$ whenever $\Ep [\ltwo{X}^4] <
  \infty$ and either of the following conditions holds:
  \begin{enumerate}[leftmargin=2em,label=(\arabic*)]
  \item For any  $\vec{\sigma} = (\sigma_1, \ldots, \sigma_m) \in \fsymlinksp$, $\sigma_j'$ are Lipschitz continuous.
  \item $X$ has continuous density on $\R^d$.
  \end{enumerate}
\end{lemma}
We can now present the master result for semi-parametric
approaches, which characterizes the asymptotic behavior of the
semi-parametric estimator with the variance function
\begin{align*}
  C_{m, \vec{\sigma}\opt} := \frac{1}{m} \cdot
  \frac{\frac{1}{m} \sum_{j = 1}^m \E[\sigma_j\opt(Z) (1 - \sigma_j\opt(Z))]}{
    (\frac{1}{m} \sum_{j = 1}^m
    \E[{\sigma_j\opt}'(Z)])^2}.
\end{align*} 
\begin{theorem}
  \label{theorem:semiparametric-master}
  Let Assumption~\ref{assumption:x-decomposition} hold and assume $|Z| >
  0$ with probability one and has nonzero and continuous density $p(z)$ on
  $(0, \infty)$.  Let Assumption~\ref{assumption:sp-general} hold and assume
  that $\Ep [\ltwo{X}^4] < \infty$.  Then $\sqrt{n}(\spe - u\opt)$ is
  asymptotically normal, and the normalized estimator $\what{u}\sem_{n,m} =
  \spe / \ltwos{\spe}$ satisfies
	\begin{equation*}
		\sqrt{n} \prn{\what{u}\sem_{n,m} - u\opt}
		\cd \normal\prn{0, 	C_{m, \vec{\sigma}\opt}
			\prn{\proj_{u\opt}^\perp \Sigma \proj_{u\opt}^\perp}^\dagger}.
	\end{equation*}
\end{theorem}
\noindent
See Appendix~\ref{proof:semiparametric-master} for proof details.
Notably, Theorem~\ref{theorem:semiparametric-master} exhibits optimal
$1/m$ scaling in the covariance whenever $\E[{\sigma\opt_j}'(Z)] \gtrsim 1$.

\subsection{A single index model}
\label{sec:semiparametric}

Our first example application of Theorem~\ref{theorem:semiparametric-master}
is to a single index model. We present a multi-phase estimator that first
estimates the direction $u\opt = \theta\opt / \ltwos{\theta\opt}$, then uses
this estimate to find a (consistent) estimate of the link $\sigma\opt$,
which we can then substitute directly into
Theorem~\ref{theorem:semiparametric-master}.
We defer all proofs of this section to Appendix~\ref{sec:proof-nonparametric},
which also includes a few auxiliary results that we use to prove the
results proper.

We present the abstract convergence result for link functions first,
considering a scenario where we have an initial guess $\uinit$ of the
direction $u\opt$, independent of $(X_i, Y_{ij})_{i \le n,j\le m}$, for
example constructed via a small held-out subset of the data. We set
\begin{equation*}
  \vec{\sigma}_n = \argmin_{\vec{\sigma} \in \fsymlinksp}
  \sum_{i = 1}^n\sum_{j = 1}^m
  \left(\sigma_j(\<\uinit, X_i\>) - Y_{ij}\right)^2,
\end{equation*}
where $\fsymlinksp \subset (\fsymlinkL)^m$ and so it consists of
nondecreasing $\lipconst$-Lipschitz link functions with $\sigma(0) = \half$.
We assume that for all $n$, there exists a (potentially random) $\epsilon_n$
such that
\begin{equation*}
  \ltwos{\uinit - u\opt} \le \uerror_n.
\end{equation*}
\begin{proposition}
  \label{proposition:consistency-of-sigma}
  Let $X_i$ be vectors with $\E[\ltwo{X}^k] < \infty$, where
  $k \ge 2$. Then with probability $1$,
  there is a finite (random)
  $C < \infty$ such that for all large enough $n$,
  \begin{equation*}
    \LtwoPbold{\vec{\sigma}_n(Y\<u\opt, X\>)
      - \vec{\sigma}\opt(Y\<u\opt, X\>)}^2
    \le C \left[
      n^{\frac{2}{3k} - \frac{2}{3}}
      + \uerror_n^2
      + \lipconst n^{-\frac{k}{2(k + 1)}}
      \right].
  \end{equation*}
\end{proposition}
\noindent
The proof is more or less a consequence of standard convergence results
for nonparametric function estimation, though we include
it for completeness in Appendix~\ref{sec:proof-consistency-of-sigma}
as it includes a few additional technicalities because
of the initial estimate of $u\opt$.

Summarizing, we see that a natural procedure is available: if we have models
powerful enough to accurately estimate the conditional label probabilities
$Y \mid X$, then Proposition~\ref{proposition:consistency-of-sigma} coupled
with Theorem~\ref{theorem:semiparametric-master} shows that we can achieve
estimation with near-optimal asymptotic covariance. In particular, if
$\uinit$ is consistent (so $\uerror_n \cp 0$), then $\spe$ induces a
normalized estimator $\what{u}\sem_n = \spe / \ltwos{\spe}$ satisfying
$\sqrt{n}(\what{u}\sem_n - u\opt) \cd \normal(0, C_{m,\vec{\sigma}}
\prns{\proj_{u\opt}^\perp \Sigma \proj_{u\opt}^\perp}^\dagger)$.

\subsection{Crowdsourcing model}
\label{sec:crowdsourcing}


Crowdsourcing typically targets estimating rater reliability, then using
these reliability estimates to recover ground truth labels as accurately as
possible, with versions of this approach central since at least
\citeauthor{DawidSk79}'s Expectation-Maximization-based
approaches~\cite{DawidSk79, WhitehillWuBeMoRu09, RaykarYuZhVaFlBoMo10}.  We
focus here on a simple model of rater reliability,
highlighting how---at least in our stylized model of classifier
learning---by combining a crowdsourcing reliability model and still using
all labels in estimating a classifier, we can achieve asymptotically
efficient estimates of $\theta\opt$, rather than the robust but slower
estimates $\what{\theta}_{n,m}\mv$ arising from ``cleaned'' labels.


We adopt \citeauthor{WhitehillWuBeMoRu09}'s roughly ``low-rank'' model for
label generation~\citep{WhitehillWuBeMoRu09}: for binary classification with
$m$ labelers and distinct link functions $\sigma_i\opt$, model the
difficulty of $X_i$ by $\beta_i \in (-\infty, \infty)$, where
$\sgn(\beta_i)$ denotes the true class $X_i$ belongs to.  A parameter
$\alpha_j$ models the expertise of annotator $j$, and the probability
labeler $j$ correctly classifies $X_i$ is
\begin{align*}
  \Prb(Y_{ij} = 1) = \frac{1}{1 +\exp(-\alpha_i \beta_j)}.
\end{align*}
(See also \citet{RaykarYuZhVaFlBoMo10}.) 
The focus in these papers was to construct gold-standard labels
and datasets $(X_i, Y_i)$; here, we take the alternative perspective
we have so far advocated to show how using all labels can yield strong
performance.


We thus adopt a semiparametric approach: we model the labelers, assuming a
black-box crowdsourcing model that can infer each labeler's ability, then
fit the classifier.  We represent labeler $j$'s expertise by a scalar
$\alpha_j\opt \in (0, \infty)$. Given data $X_i = X$ and the normalized
$\theta\opt = u\opt$, we assume a modified logistic link
\begin{align*}
  \Prb(Y_{ij} = 1 \mid X_i = x) = \frac{1}{1 + \exp (-\alpha\opt_j  \< \theta\opt, x \>)} = \sigma\lr (\alpha_j\opt  \< u \opt, x \>),
\end{align*}
so $\alpha_j\opt = \infty$ represents an omniscient labeler while
$\alpha_j\opt = 0$ means the labeler chooses random labels regardless of the
data. Let $\alpha\opt = (\alpha_1\opt, \ldots, \alpha_m\opt) \in
\R^m_{+}$. Then, in keeping with the plug-in approach of the master
Theorem~\ref{theorem:semiparametric-master}, we assume the blackbox
crowdsourcing model generates an estimate $\alpha_n = (\alpha_{n,1},
\ldots, \alpha_{n,m}) \in \R_+^m$ of $\alpha^\star$ from the data $\{(X_i,
(Y_{i1}, \ldots, Y_{im}))\}_{i=1}^n$.

We consider the algorithm using the blackbox crowdsourcing model and
empirical risk minimization with the margin-based loss
$\loss_{\vec{\sigma}_n, \theta}$ as in Section~\ref{sec:master-results-sem},
with
$\vec{\sigma}_n(t)
\defeq (\sigma\lr(\alpha_{n,1}t),\dots, \sigma\lr(\alpha_{n,m}t))$,
or equivalently using the rescaled logistic loss
\begin{align*}
  \loss_{\vec{\sigma}_n, \theta}(y \mid x)
  = \frac{1}{m} \sum_{j=1}^m \loss\lr_\theta (y_j \mid \alpha_{n,j} x) .
\end{align*}
This allows us to apply our general semiparametric
Theorem~\ref{theorem:semiparametric-master} as long as the crowdsourcing
model produces a consistent estimate $\alpha_n \cp \alpha^\star$ (see
Appendix~\ref{proof:crowd-sourcing-estimator} for a proof):
\begin{proposition}
  \label{prop:crowd-sourcing-estimator}
  Let Assumption~\ref{assumption:x-decomposition} hold, $|Z| > 0$ have
  nonzero and continuous density $p(z)$ on $(0, \infty)$, and $\Ep
  [\ltwo{X}^4] < \infty$. If $\alpha_n \cp \alpha\opt \in \R^m$, then
  $\sqrt{n}(\spe - u\opt)$ is asymptotically normal, and the normalized
  estimator $\what{u}\sem_{n,m} = \spe / \ltwos{\spe}$ satisfies
  \begin{equation*}
    \sqrt{n} \prn{\what{u}\sem_{n,m} - u\opt}
    \cd \normal\prn{0,
      \frac{1}{\sum_{j = 1}^m \E[\sigma\lr(\alpha_j\opt Z)
          (1 - \sigma\lr(\alpha_j\opt Z))]}
      \prn{\proj_{u\opt}^\perp \Sigma \proj_{u\opt}^\perp}^\dagger}.
  \end{equation*}
\end{proposition}

By Proposition~\ref{prop:crowd-sourcing-estimator}, the semiparametric
estimator $\what{u}\sem_{n,m}$ is efficient when the rater reliability
estimates $\alpha_n \in \R^m$ are consistent.  It is also immediate that if
$\alpha_j\opt \le \alpha_{\max} < \infty$ are bounded, then the asymptotic
covariance multiplier $C_{m,\alpha\opt} = (\sum_{j = 1}^m
\E[\sigma\lr(\alpha_j\opt Z) (1 - \sigma\lr(\alpha_j\opt Z))])^{-1} =
O(1/m)$, so we recover the $1/m$ scaling of the MLE, as opposed to the
slower rates of the majority vote estimators in
Section~\ref{sec:robustness-model-misspecification}.  At this point, the
refrain is perhaps unsurprising: using all the label information can yield
much stronger convergence guarantees.


\newcommand{\logit}{\mathop{\rm logit}}

\section{Experiments}
\label{sec:experiments}



We conclude the paper with several experiments to 
test whether the (sometimes implicit)
methodological suggestions hold merit.  Before delving into our experimental
results, we detail a few of the expected behaviors our theory suggests; if
we fail to see them, then the model we have proposed is too unrealistic to
inform practice.  First, based on the results
in~\Cref{sec:well-sepcified-mod}, we expect the classification error to be
better for the non-aggregated algorithm $\tmle$, and the gap between the two
algorithms to become larger for less noisy problems. Moreover, we only
expect $\tmle$ to be calibrated, and our theory predicts that the
majority-vote estimator's calibration worsens as the number of labels $m$
increases. Generally,
so long as we model uncertainty with enough fidelity,
Corollary~\ref{corollary:lowd-misspcfd-logistic-regression} suggests
that multilabel estimators should exhibit better performance than
those using majority vote labels $\majY$.

To that end, we provide experiments on two real datasets and one
semi-synthetic dataset: the BlueBirds dataset (\Cref{sec:BB}), CIFAR-10H
(\Cref{sec:cifar}), and a modified version of the CIFAR-10 dataset
(\Cref{sec:semi-synthetic}). We consider our two main algorithmic models:
\begin{enumerate}[label=(\roman*),leftmargin=*]
  \setlength{\itemsep}{0pt}
\item The maximum likelihood estimator $\tmle$ based on non-aggregated data in
  Eq.~\eqref{eq:maximum-likelihood}.
\item The majority-vote based estimator $\mve$ of
  Eq.~\eqref{eq:logistic-majority}, our proxy for modern data
  pipelines.
\end{enumerate}
Unfortunately, a paucity of large-scale multi-label datasets that we know of
precludes experiments on fully real datasets; the ImageNet
creators~\cite{DengDoSoLiLiFe09, RussakovskyDeSuKrSaMaHuKaKhBeBeFe15} no
longer have any intermediate label information from their construction. To
simulate the larger scale data collection process, we create a large scale
semisynthetic dataset (Section~\ref{sec:semi-synthetic}) using the raw
CIFAR-10 data and using trained neural networks as labelers, giving us a
semi-synthetic dataset with multiple labels. As the ground truth model is
well-specified and we know the link function, beyond the aforementioned
estimators with full information and majority vote, it is also of practical
interest to see how the aggregated labels from other crowdsourcing methods
compare to majority vote, and to knowing all individual labels.

\subsection{BlueBirds}
\label{sec:BB}

We begin with the
\href{https://github.com/eaplatanios/noisy-labels}{BlueBirds}
dataset~\cite{WelinderBrPeBe10}, which is a relatively small dataset
consisting of 108 images with ResNet features.  The classification problem
is challenging, and the task is to classify each image as one of
\textit{Indigo Bunting} or \textit{Blue Grosbeak} (two similar-looking blue
bird species). For each image, we have $39$ labels, obtained through Amazon
Mechanical Turk workers.  We use a pretrained (on ImageNet)
ResNet50 model to generate image
features, then apply PCA to reduce the dimensionality from
$d_{\textup{init}} = 2048$ to $d=25$.

We repeat the following experiment $T = 100$ times.
For each number $m = 1, \ldots, 35$ of labelers, we fit
the multilabel logistic model~\eqref{eq:maximum-likelihood}
and the majority vote estimator~\eqref{eq:logistic-majority},
finding calibration and classification errors using 10-fold cross validation.
We measure calibration error on a held-out example $x$
by $|\logit(\wt{p}(x)) -
\logit(\what{p}(x))|$, where $\wt{p}(x)$ is the predicted probability
and $\what{p}(x)$ is the empirical probability (over the labelers),
where $\logit(p) = \log\frac{p}{1 - p}$; we measure classification
error on example $x$ with
labels $(y_1, \ldots, y_m)$ by $\frac{1}{m} \sum_{j = 1}^m \ind\{y_j \neq
\sgn(\wt{p}(x) - \half)\}$, giving an inherent noise floor because
of labeler uncertainty. We
report the results in Figure~\ref{fig:bluebird}.
These plots corroborate
our theoretical predictions: as the number of labelers $m$ increases, both
the majority vote method and the full label method exhibit improved
classification error, but considering all labels gives a (significant)
improvement in accuracy and in calibration error.

%
%

\begin{figure}[ht]
  \begin{center}
    \begin{tabular}{cc}
      \begin{overpic}[width=.48\columnwidth]{
	  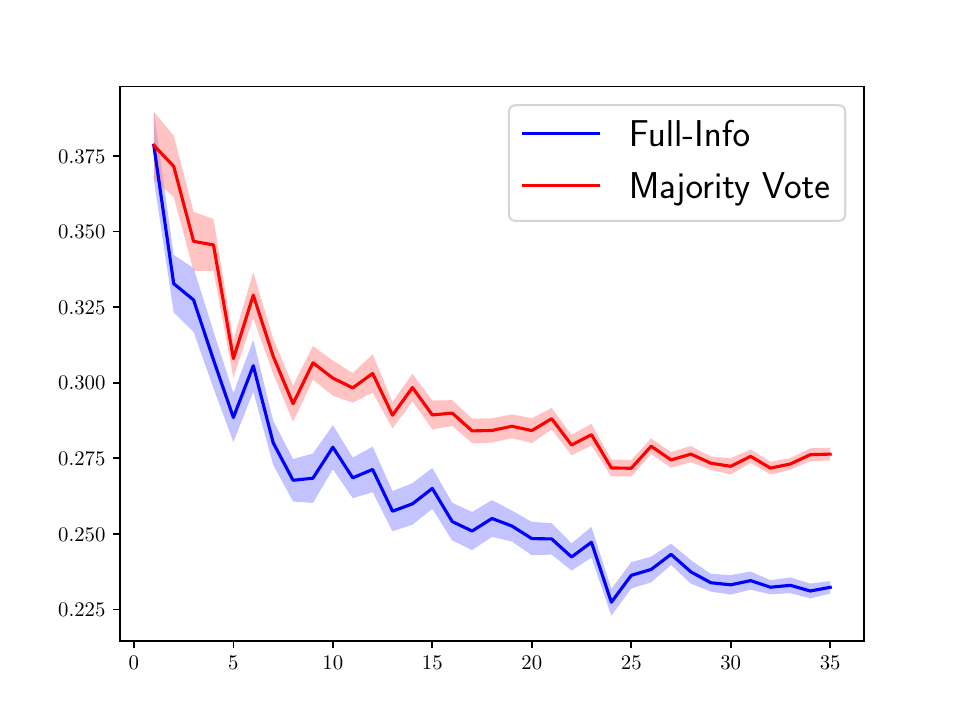}
	\put(30,-1){{\small Number of labelers $m$}}
	\put(-3,20){
	  \rotatebox{90}{{\small Classification error}}}
      \end{overpic} &
      \hspace{-.3cm}
      \begin{overpic}[width=.48\columnwidth]{
	  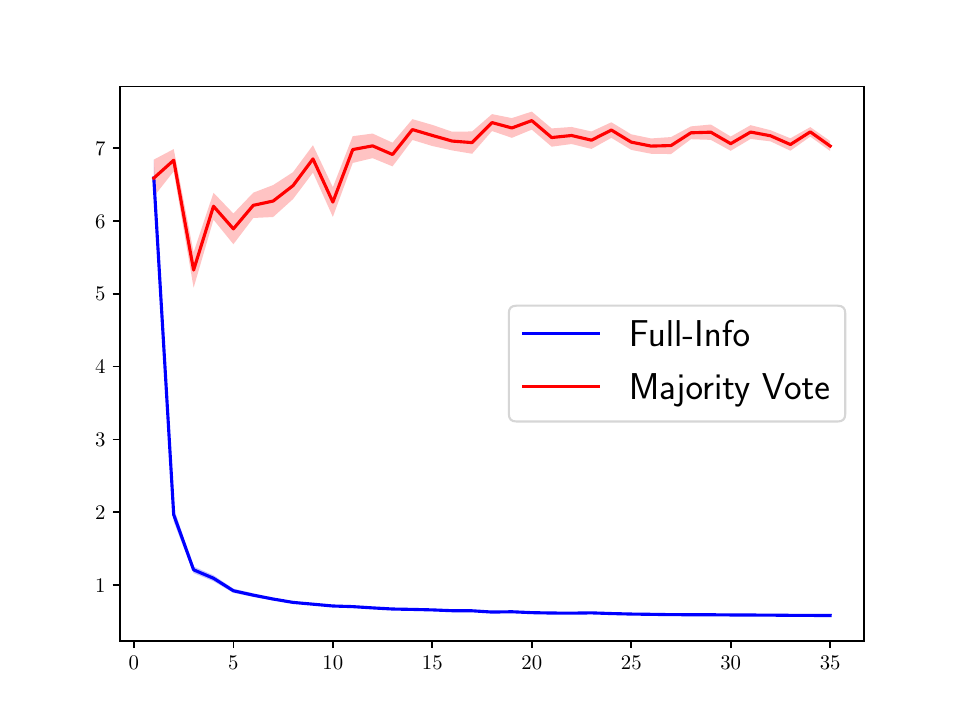}
	\put(30,-1){{\small Number of labelers $m$}}
	\put(-3,20){
	  \rotatebox{90}{{\small Calibration error}}}
      \end{overpic} \\
      (a) & (b)
    \end{tabular}
    \caption{\label{fig:bluebird} Experiments on BlueBirds dataset. (a)
      Classification error.  (b) Calibration error $|\logit(\wt{p}) -
      \logit(p)|$ with ResNet features reduced via PCA to dimension
      $d=25$. Error bars show 2 standard error confidence
      bands over $T = 100$ trials.}
  \end{center}
\end{figure}

\subsection{CIFAR-10H}
\label{sec:cifar}

\begin{figure}[ht!]
  \begin{center}
    \begin{tabular}{cc}
      \begin{overpic}[width=.48\columnwidth]{
	  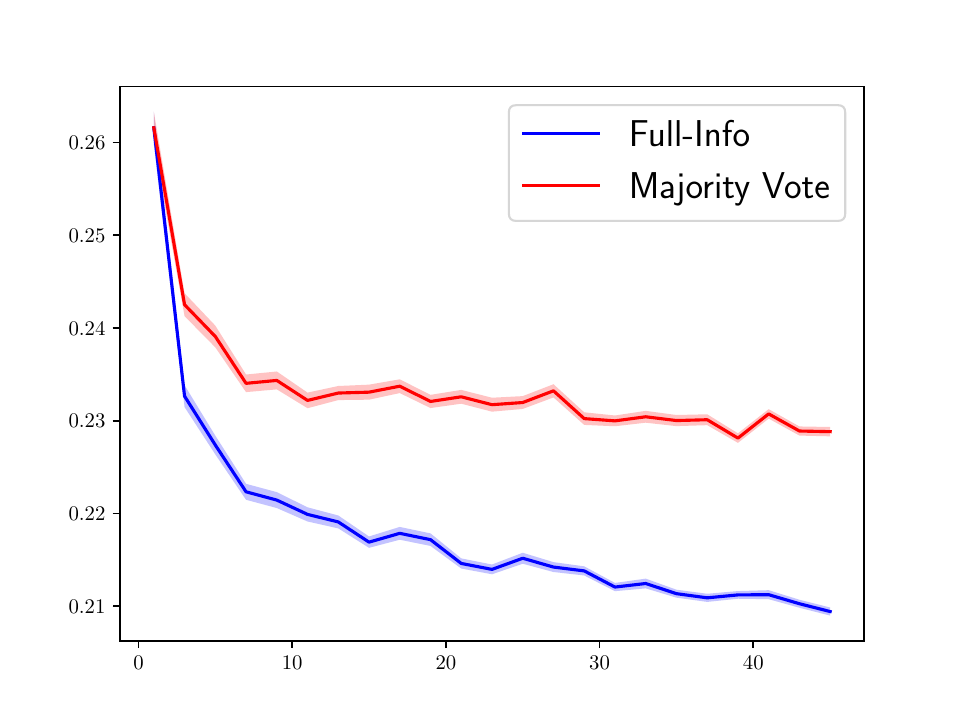}
	\put(30,-1){{\small Number of labelers $m$}}
	\put(-3,20){
	  \rotatebox{90}{{\small Classification error}}}
      \end{overpic} &
      \hspace{-.3cm}
      \begin{overpic}[width=.48\columnwidth]{
	  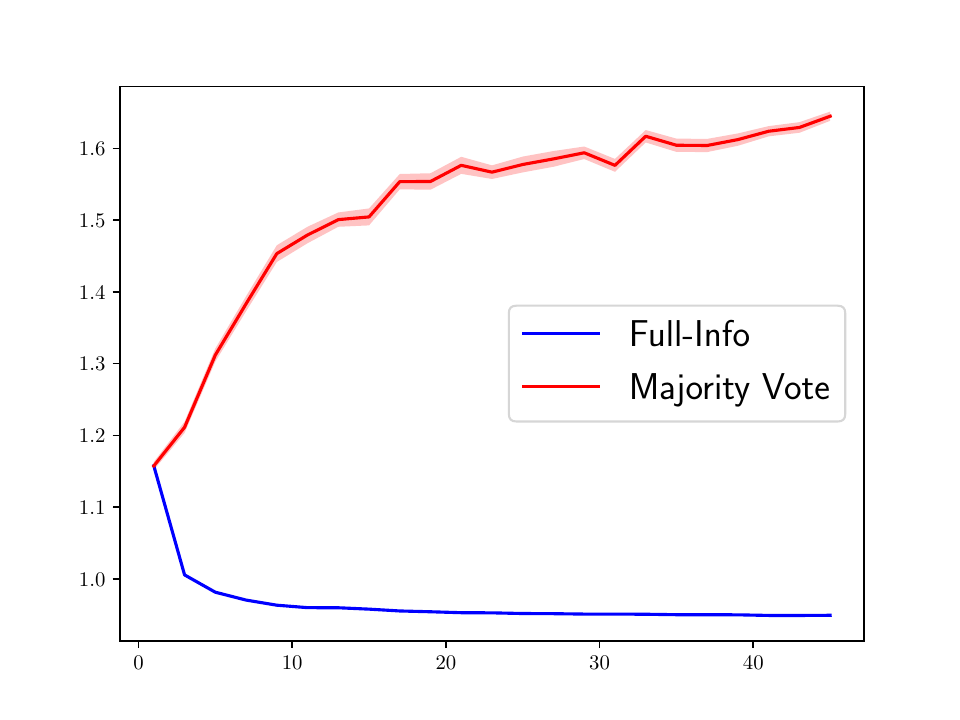}
	\put(30,-1){{\small Number of labelers $m$}}
	\put(-3,20){
	  \rotatebox{90}{{\small Calibration error}}}
      \end{overpic} \\
      (a) & (b)
    \end{tabular}
    \caption{\label{fig:cifar} Experiments on CIFAR-10H dataset. (a)
      Classification error. (b) Calibration error $|\logit(\wt{p}) -
      \logit(p)|$ with ResNet features reduced via PCA to dimension
      $d=40$.
      Error bars show 2 standard error confidence bands over
      $T = 100$ trials.}
  \end{center}
\end{figure}

For our second experiment, we consider \citeauthor{PetersonBaGrRu19}'s
\href{https://github.com/jcpeterson/cifar-10h}{CIFAR-10H}
dataset~\citep{PetersonBaGrRu19}, which consists of $10,\!000$ images from
CIFAR-10 test set with soft labeling in that for each image, we have
approximately $50$ labels from different annotators.  Each $32 \times 32$
image in the dataset belongs to one of the ten classes \texttt{airplane},
\texttt{automobile}, \texttt{bird}, \texttt{cat}, \texttt{dog},
\texttt{frog}, \texttt{horse}, \texttt{ship}, or \texttt{truck}; labelers
assign each image to one of the classes. To maintain some fidelity to the
binary classification setting we analyze throughout the paper, we transform
the problem into a set of 10 binary classification problems. For each class
$c$, we take each initial image/label pair $(x, y) \in \R^{32 \times 32}
\times \{1, \ldots, 10\}$, assigning binary label $1$ if $y = c$ and $0$
otherwise (so the annotator labels it as an alternative class $y \neq c$).
Most of the images in the dataset are very easy to classify: more than 80\%
have a unanimous label from each of the $m = 50$ labelers, meaning that the
MLE and majority vote estimators~\eqref{eq:maximum-likelihood}
and~\eqref{eq:logistic-majority} coincide for these.  (In experiments with
this full dataset, we saw little difference between the two estimators.)

As our theoretical results highlight the importance of classifier
difficulty, we therefore balance the dataset by considering subsets
of harder images as follows. For each fixed target $c$ (e.g., \texttt{cat})
and for image $i$, let $\what{p}_i$ be the empirical
probability of the target among the 50 annotator labels.
Then for $p \in [\frac{1}{2}, 1]$, define the subsets
\begin{align*}
  \mathcal{S}_p = \left\{i \in [n]
  : \max\left\{\what{p}_i, 1 - \what{p}_i\right\} \leq p\right\},
\end{align*}
so that $p = \frac{1}{2}$ corresponds to images with substantial confusion,
and $p = 1$ to all images (most of which are easy).  We test on
$\mathcal{S}_{0.9}$ (labelers have at most 90\% agreement), which consists
with $441$ images. For image $i$, we again generate features $x_i \in
\R^d$ by taking the last layer of a pretrained ResNet50 neural network
$\tilde{x}_i \in \R^{d_{\textup{init}}}$, using PCA to reduce
to a $d = 40$-dimensional feature.
We follow the same procedure as in Sec.~\ref{sec:BB}, subsampling
$m = 1, 2, \ldots, 45$ labelers and using 10-fold cross validation
to evaluate classification and calibration error.
We report the results
in~\Cref{fig:cifar}. Again we see that---as the number of labelers
increases---both aggregated and non-aggregated methods evidence improved
classification error, but the majority vote procedure (cleaned data) yields
less improvement than one with access to all (uncertain) labels.  These
results are again consistent with our theoretical predictions.


\subsection{Crowdsourcing methods on semisynthetic CIFAR-10 labels} \label{sec:semi-synthetic}

In our final set of experiments, we adapt the original
CIFAR-10~\cite{KrizhevskyHi09} dataset, which consists of 6000 $32 \times
32$ images from each of $k = 10$ classes (60,000 total images).  To mimic
collecting and cleaning data with noisy labelers---rather than the
single-label ``gold standard'' in the base CIFAR-10 data---we construct
pseudo-labelers using a ResNet18 network whose accuracy we can directly
adjust. This allows us to compare the maximum-likelihood estimator
$\what{\theta}^{\textup{mle}}_{n,m}$, the majority vote estimator $\mve$,
and, to test if the predictions of our stylized models still hold
for more advanced aggregation strategies, we include
Dawid-Skene~\citep{DawidSk79} and GLAD~\citep{WhitehillWuBeMoRu09}
crowdsourcing estimators for the labels.


\newcommand{\softmax}{\textup{softmax}}

The original CIFAR-10 dataset~\cite{KrizhevskyHi09} consists of 6000
$32 \times 32$ images from each of $k = 10$ classes (60,000 images total).
To generate the semisynthetic labelers and labels, we use a pretrained
ResNet18 model~\cite{HeZhReSu16}, an eighteen layer residual network, $f:
\mc{X} \to \R^k$. For $s \in \R^k$, define the softmax mapping $\softmax(s)
= [e^{s_y} / \sum_{l = 1}^k e^{s_l}]_{y = 1}^k$.  Then for an input $x \in
\mc{X}$, the model $f$ outputs a score $f(x) \in \R^k$, where
$\softmax(f(x))$ indicates the probabilities $\P(Y = y \mid X = x)$.
The model $f$ has the form
\begin{equation*}
  f(x) = {\Theta\opt}^\top \phi(x),
\end{equation*}
where $\Theta\opt = [\theta\opt_1 ~ \cdots ~ \theta\opt_k] \in \R^{d \times
  k}$ is a matrix of per-class weights, and $\phi : \mc{X} \to \R^d$
represents the second-to-last layer outputs of the neural network.  As we
fit linear models to the data, we therefore take these outputs as our data,
so that for the $i$th image $x_i$, we have $X_i = \phi(x_i)$, and
$\Theta\opt$ is the ground-truth parameter.  We construct semi-synthetic
labels $Y_{ij} \in \{1, \ldots, k\}$, $j = 1, \ldots, m$ of varying
accuracy/labeler expertise as follows. Letting $\alpha > 0$
be a chosen constant we vary to model labeler expertise, we draw
\begin{equation*}
  \P(Y_{ij} = y \mid X_i = x)
  = \softmax(\alpha f(x))_y
  = \frac{\exp(\alpha \<\theta\opt_y, \phi(x)\>)}{
    \sum_{l = 1}^k \exp(\alpha \<\theta\opt_l, \phi(x)\>)}.
\end{equation*}
We vary $\alpha$ in experiments to adjust that the median value
of $p_i \defeq \max_{y \in [k]} \softmax(\alpha f(x_i))_y$, where
$p_i \in (1/k, 1)$. In the ``expert labelers'' case, we take
$\alpha \to \infty$ so that $\mbox{median}(\{p_i\}) \to 1$,
while the ``crowd labelers'' case corresponds to
$\alpha \downarrow 0$ and $\mbox{median}(\{p_i\}) \to 1/k$.

Given a semisynthetic dataset $\{(X_i, (Y_{i1}, \dots, Y_{im}))\}_{i=1}^n$,
$\what{\theta} \in \R^{d \times k}$ estimates $\Theta\opt$.
We fit $\mve$ and
$\what{\theta}^{\textup{mle}}_{n,m}$ using the multiclass logistic loss.
To test our predictions, we also investigate the standard
Dawid-Skene (DS)~\cite{DawidSk79} and GLAD~\cite{WhitehillWuBeMoRu09}
crowdsourcing methods.
The DS method models the probability that labeler $j$ assigns label $y'$
given correct label $y$ through a matrix $\alpha_j \in \R^{k \times k}$ with
$p_j(y' \mid y) = \frac{1}{1 + e^{-\alpha_j(y', y)}}$, while GLAD estimates
the probability that labeler $j$ correctly identifies the label of an
example $i$ via $p_{ij} = \frac{1}{1 + e^{-\beta_i \alpha_j}}$, where
$\alpha_j \in [-\infty,\infty]$ indicates labeler $j$'s ability and
$\beta_i \in [0, \infty]$ ambiguity in example $i$.
Each method uses EM to approximately maximize the likelihood $\prod_{i =
  1}^n \prod_{j = 1}^m \Prb(Y_{ij} \mid \alpha_j, \beta_i)$ (letting
$\beta_i = 1$ for DS) of the observed labels $Y_{ij}$,
initializing $\alpha$ to maximize the likelihood $\prod_{i = 1}^n \prod_{j
  = 1}^m \Prb(Y_{ij} \mid \majY_i, \alpha_j, \beta_i)$ assuming the majority
votes $\majY_i$ are correct and $\beta_i = 1$.
Given estimates $\what{\alpha}$ and $\what{\beta}$, one may impute soft
labels $\what{p}_i \in \R_k^+$ for each example $i$, where $\what{p}_{iy}$
indicates the method's estimated probability that example $i$ is of class
$y$, and hard labels $\what{Y}_i = \argmax_y \what{p}_{iy}$.
The labels yield  estimators from, respectively, hard and soft labels via
\begin{equation*}
  \what{\theta}_{\textup{h}}
  = \argmin_{\theta_1, \ldots, \theta_k}
  \sum_{i = 1}^n \log\bigg(\sum_{l = 1}^k
  e^{\<X_i, \theta_l - \theta_{\what{Y}_i}\>}\bigg)
  ~~ \mbox{and} ~~
  \what{\theta}_{\textup{s}}
  = \argmin_{\theta_1, \ldots, \theta_k}
  \sum_{i = 1}^n
  \sum_{y = 1}^k
  \what{p}_{iy} \log\bigg(\sum_{l = 1}^k
  e^{\<X_i, \theta_l - \theta_y\>} \bigg).
\end{equation*}
We let $\dse$
and $\glade$ denote the corresponding estimators using hard labels,
and $\dspe$ and $\gladpe$ those trained with the estimated
probabilities $\what{p}_i$.
Our experiments use
Crowd-Kit~\citep{UstalovPaTs24}.

\begin{figure}[ht!]
  \begin{center}
    \begin{tabular}{cc}
      \hspace{-2cm}
      \begin{overpic}[width=.65\columnwidth]{
	  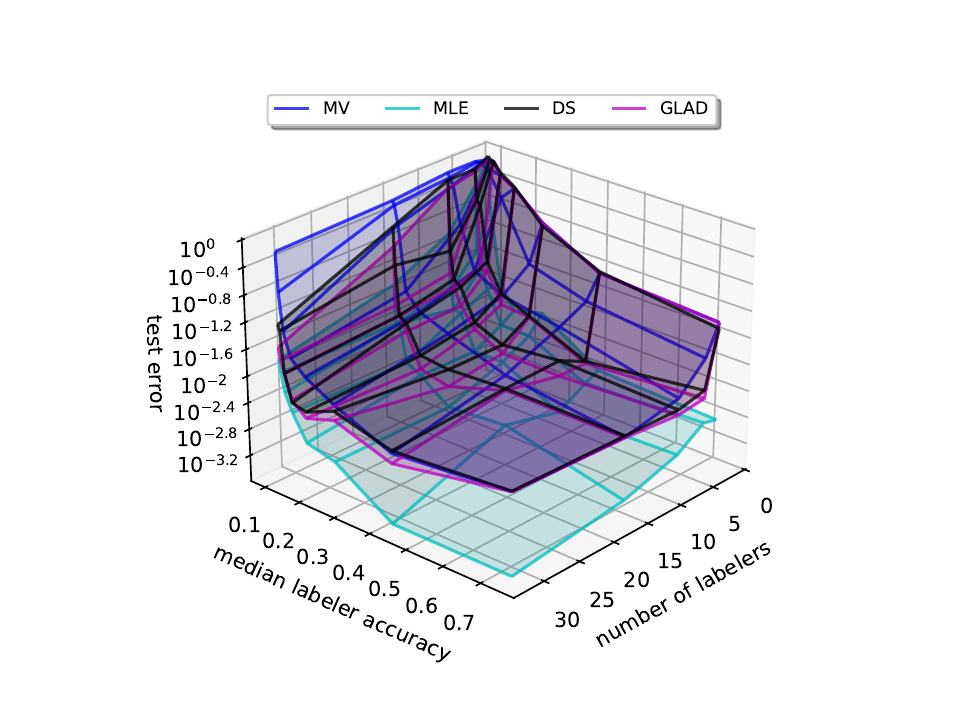}
      \end{overpic} &
      \hspace{-2.5cm}
      \begin{overpic}[width=.65\columnwidth]{
	  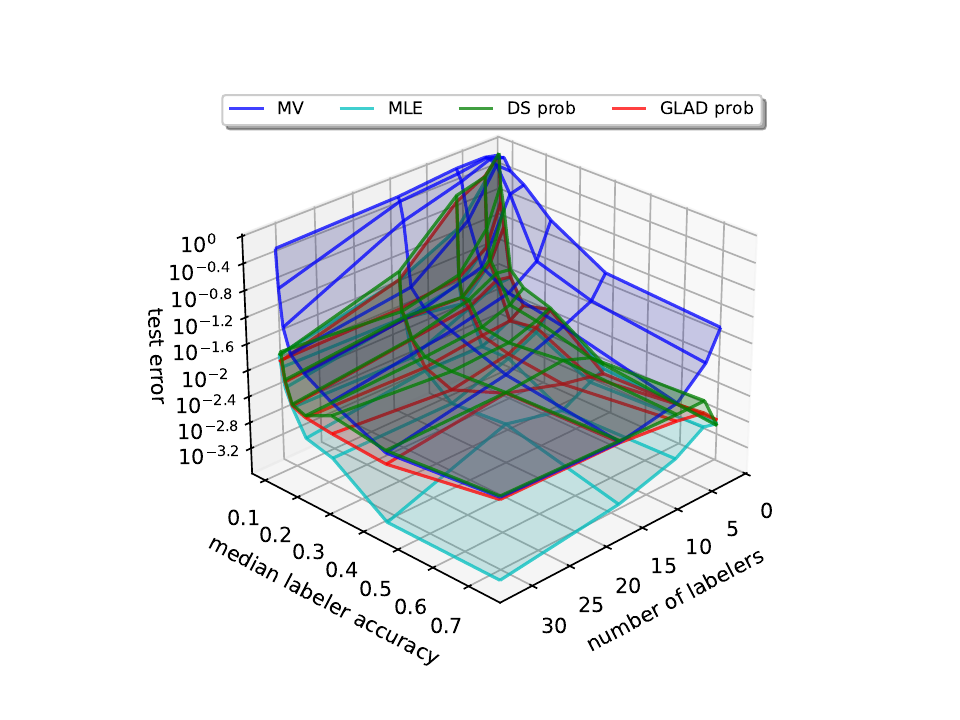}
      \end{overpic} \\
      \hspace{-2cm} (a) &
      \hspace{-2.5cm} (b) \\
      \hspace{-2cm}
      \begin{overpic}[width=.5\columnwidth]{
	  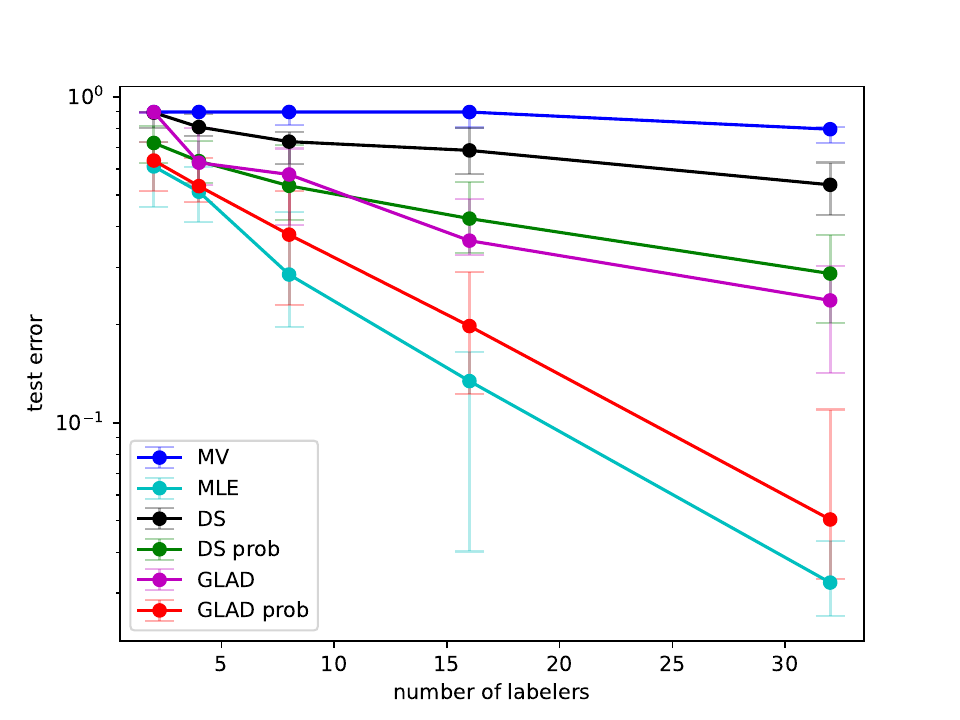}
      \end{overpic} &
      \hspace{-2.5cm}
      \begin{overpic}[width=.5\columnwidth]{
	  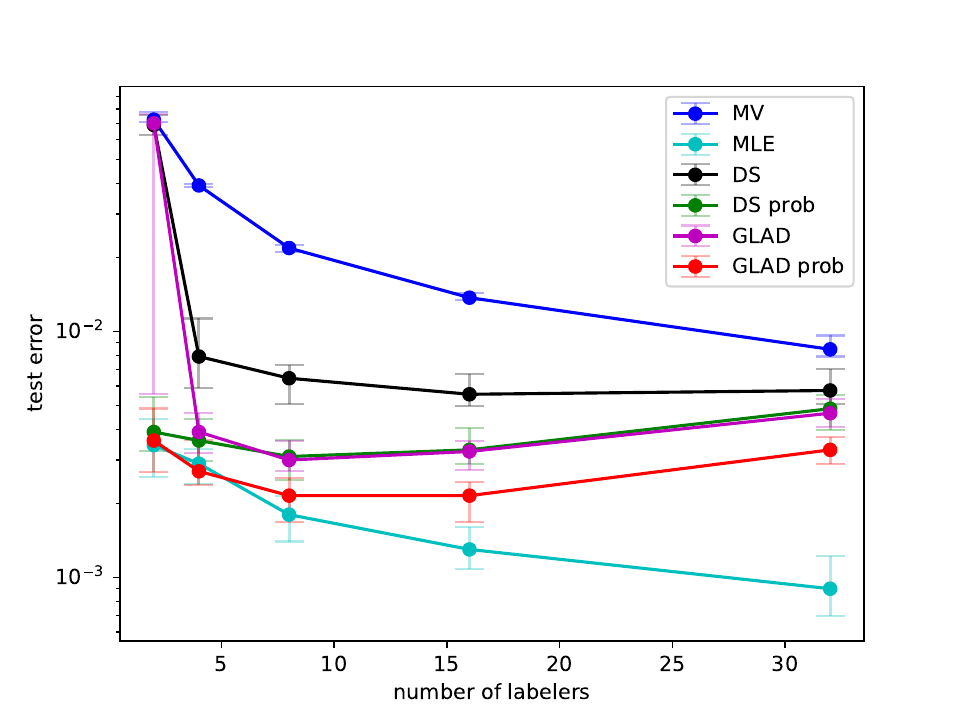}
      \end{overpic} \\
      \hspace{-2cm} (c) & \hspace{-2.5cm} (d)
    \end{tabular}
    \caption{\label{fig:semisynthetic} Experiments on the semisynthetic
      CIFAR-10 dataset. The median labeler accuracy represents the
      median $p$ of $p_i = \max_y \softmax(\alpha f(x_i))_y$
      on the training data. Results are
      averaged over $20$ trials, and vertical axes
      give mis-classification rate from (synthetic) ground-truth labels
      on held-out test set. Legend keys correspond to
      maximum likelihood $\what{\theta}^{\textup{mle}}_{n,m}$ (MLE),
      majority vote $\mve$ (MV), hard-labeled Dawid-Skene (DS) and
      GLAD crowdsourced estimators $\dse$ and $\glade$,
      and soft-labeled Dawid-Skene (DS prob) and GLAD (GLAD prob)
      estimators $\dspe$ and $\gladpe$.
      (a) Comparison of methods using hard labels.
      (b) Comparison with crowdsourced estimated soft-labels.
      (c) Number of labelers $m$ versus test error for
      fixed median accuracy $p = .105$ (noisy labelers). (d)
      Number of labelers $m$ versus test error for
      fixed median accuracy $p = .42$. Both (c) and (d) report
      95\% error bars over the trials.}
  \end{center}
\end{figure}

We report the results in Fig.~\ref{fig:semisynthetic}. As we see, having the
ground truth model well-specified, using the full information of soft
labels, the MLE $\what{\theta}^{\textup{mle}}_{n,m}$ outperforms not only
all aggregation methods, but also black-box crowdsourcing models that
generate soft labels. The more advanced aggregation methods $\dse$ and
$\glade$ yield smaller test error than $\mve$, and using the soft labels,
$\dspe$ and $\gladpe$ further reduces the error, suggesting the benefits
of using soft labels from crowdsourcing methods for training even if the
underlying model is unknown.

\FloatBarrier
\section{Discussion}

In spite of the technical detail we require to prove our results,
we view this work as almost preliminary and hope that it inspires further
work on the full pipeline of statistical machine learning,
from dataset creation to model release. Many questions remain
both on the theoretical and applied sides of the work.

On the theoretical side, our main focus has been on a stylized model of
label aggregation, with majority vote mostly---with the exception of the
crowdsourcing model in Sec.~\ref{sec:semi-param}--functioning as the
stand-in for more sophisticated aggregation strategies. It seems challenging
to show that \emph{no} aggregation strategy can work as well as multi-label
strategies; it would be interesting to more precisely delineate the benefits
and drawbacks of more sophisticated denoising and whether it is useful.  We
focus throughout on low-dimensional asymptotics, using asymptotic normality
to compare estimators. While these insights are valuable, and make
predictions consistent with the experimental work we provide, investigating
how things may change with high-dimensional scaling or via non-asymptotic
results might yield new insights both for theory and methodology. As one
example, with high-dimensional scaling, classification datasets often become
separable~\cite{CandesSu19pnas, CandesSu19}, which is consistent with modern
applied machine learning~\cite{ZhangBeHaReVi21} but makes the asymptotics we
derive impossible. One option
is to investigate the asymptotics of maximum-margin estimators,
which coincide with the limits of ridge-regularized logistic
regression~\cite{RossetZhHa04, SoudryHoNaGuSr18}.

On the methodological and applied side, given the extent to which the
test/challenge dataset methodology drives progress in machine
learning~\cite{Donoho17}, it seems that developing newer datasets to
incorporate labeler uncertainty could yield substantial benefits.  A
particular refrain is that modern deep learning methods are overconfident in
their predictions~\cite{GoodfellowShSz15, ZhangBeHaReVi21}; perhaps by
calibrating them to \emph{labeler} uncertainty we could substantially
improve their robustness and performance. Additionally, recent large-scale
models now frequently use only distantly supervised
labels~\cite{RadfordKiHaRaGoAgSaAsMiClKrSu21,
  SchuhmannBeVeGoWiChCoKaMuWoScKuCrScKaJi22, GadreEtAl23}, making
the study of appropriate data cleaning and collection more timely.
We look forward to deeper
investigations of the intricacies and intellectual foundations of the full
practice of statistical machine learning.

\subsection*{Acknowledgement}

This research was partially supported by the Office of Naval Research under
awards N00014-19-2288 and N00014-22-1-2669, the Stanford DAWN Consortium,
and the National Science Foundation grants IIS-2006777 and HDR-1934578.
The authors report no competing interests.

\bibliography{bib}

\begin{thebibliography}{50}
\providecommand{\natexlab}[1]{#1}
\providecommand{\url}[1]{\texttt{#1}}
\expandafter\ifx\csname urlstyle\endcsname\relax
  \providecommand{\doi}[1]{doi: #1}\else
  \providecommand{\doi}{doi: \begingroup \urlstyle{rm}\Url}\fi

\bibitem[Asuncion and Newman(2007)]{AsuncionNe07}
A.~Asuncion and D.~J. Newman.
\newblock {UCI} machine learning repository, 2007.
\newblock URL \url{http://www.ics.uci.edu/~mlearn/MLRepository.html}.

\bibitem[Bartlett et~al.(2006)Bartlett, Jordan, and McAuliffe]{BartlettJoMc06}
P.~L. Bartlett, M.~I. Jordan, and J.~McAuliffe.
\newblock Convexity, classification, and risk bounds.
\newblock \emph{Journal of the American Statistical Association}, 101:\penalty0
  138--156, 2006.

\bibitem[Beygelzimer et~al.(2021)Beygelzimer, Liang, {Wortman Vaughan}, and
  Dauphin]{BeygelzimerLiVaDa21}
A.~Beygelzimer, P.~Liang, J.~{Wortman Vaughan}, and Y.~Dauphin.
\newblock {NeurIPS} 2021 datasets and benchmarks track, 2021.

\bibitem[Bickel et~al.(1998)Bickel, Klaassen, Ritov, and
  Wellner]{BickelKlRiWe98}
P.~Bickel, C.~A.~J. Klaassen, Y.~Ritov, and J.~Wellner.
\newblock \emph{Efficient and Adaptive Estimation for Semiparametric Models}.
\newblock Springer Verlag, 1998.

\bibitem[Boucheron et~al.(2005)Boucheron, Bousquet, and
  Lugosi]{BoucheronBoLu05}
S.~Boucheron, O.~Bousquet, and G.~Lugosi.
\newblock Theory of classification: a survey of some recent advances.
\newblock \emph{ESAIM: Probability and Statistics}, 9:\penalty0 323--375, 2005.

\bibitem[Cand\`{e}s and Sur(2019)]{CandesSu19pnas}
E.~Cand\`{e}s and P.~Sur.
\newblock A modern maximum-likelihood theory for high-dimensional logistic
  regression.
\newblock \emph{Proceedings of the National Academy of Sciences}, 116\penalty0
  (29):\penalty0 14516--14525, 2019.

\bibitem[Cand\`{e}s and Sur(2020)]{CandesSu19}
E.~Cand\`{e}s and P.~Sur.
\newblock The phase transition for the existence of the maximum likelihood
  estimate in high-dimensional logistic regression.
\newblock \emph{Annals of Statistics}, 48\penalty0 (1):\penalty0 27--42, 2020.

\bibitem[Charikar et~al.(2017)Charikar, Steinhardt, and
  Valiant]{CharikarStVa17}
M.~Charikar, J.~Steinhardt, and G.~Valiant.
\newblock Learning from untrusted data.
\newblock In \emph{Proceedings of the Forty-Ninth Annual ACM Symposium on the
  Theory of Computing}, 2017.

\bibitem[Chen et~al.(2010)Chen, Goldstein, and Shao]{ChenGoSh10}
L.~H. Chen, L.~Goldstein, and Q.-M. Shao.
\newblock \emph{Normal Approximation by Stein’s method}.
\newblock Springer, 2010.

\bibitem[Codella et~al.(2019)Codella, Rotemberg, Tschandl, Celebi, Dusza,
  Gutman, Helba, Kalloo, Liopyris, Marchetti, Kittler, and
  Halpern]{CodellaRoTsCeDuGuHeKaLiMa19}
N.~Codella, V.~Rotemberg, P.~Tschandl, M.~E. Celebi, S.~Dusza, D.~Gutman,
  B.~Helba, A.~Kalloo, K.~Liopyris, M.~Marchetti, H.~Kittler, and A.~Halpern.
\newblock Skin lesion analysis toward melanoma detection 2018: A challenge
  hosted by the {I}nternational {S}kin {I}maging {C}ollaboration ({ISIC}).
\newblock \emph{arXiv:1902.03368 [cs.CV]}, 2019.

\bibitem[Dawid and Skene(1979)]{DawidSk79}
A.~Dawid and A.~Skene.
\newblock Maximum likelihood estimation of observer error-rates using the {EM}
  algorithm.
\newblock \emph{Journal of the Royal Statistical Society}, 28:\penalty0 20--28,
  1979.

\bibitem[Deng et~al.(2009)Deng, Dong, Socher, Li, Li, and
  Fei-Fei]{DengDoSoLiLiFe09}
J.~Deng, W.~Dong, R.~Socher, L.~Li, K.~Li, and L.~Fei-Fei.
\newblock Image{N}et: a large-scale hierarchical image database.
\newblock In \emph{Proceedings of the IEEE Conference on Computer Vision and
  Pattern Recognition}, pages 248--255, 2009.

\bibitem[Donoho(2017)]{Donoho17}
D.~L. Donoho.
\newblock 50 years of data science.
\newblock \emph{Journal of Computational and Graphical Statistics}, 26\penalty0
  (4):\penalty0 745--766, 2017.

\bibitem[Gadre et~al.(2023)Gadre, Ilharco, Fang, Hayase, Smyrnis, Nguyen,
  Marten, Wortsman, Ghosh, Zhang, Orgad, Entezari, Daras, Pratt, Ramanujan,
  Bitton, Marathe, Mussmann, Vencu, Cherti, Krishna, Koh, Saukh, Ratner, Song,
  Hajishirzi, Farhadi, Beaumont, Oh, Dimakis, Jitsev, Carmon, Shankar, and
  Schmidt]{GadreEtAl23}
S.~Y. Gadre, G.~Ilharco, A.~Fang, J.~Hayase, G.~Smyrnis, T.~Nguyen, R.~Marten,
  M.~Wortsman, D.~Ghosh, J.~Zhang, E.~Orgad, R.~Entezari, G.~Daras, S.~Pratt,
  V.~Ramanujan, Y.~Bitton, K.~Marathe, S.~Mussmann, R.~Vencu, M.~Cherti,
  R.~Krishna, P.~W. Koh, O.~Saukh, A.~Ratner, S.~Song, H.~Hajishirzi,
  A.~Farhadi, R.~Beaumont, S.~Oh, A.~Dimakis, J.~Jitsev, Y.~Carmon, V.~Shankar,
  and L.~Schmidt.
\newblock {DataComp}: In search of the next generation of multimodal datasets.
\newblock In \emph{Advances in Neural Information Processing Systems 36}, 2023.

\bibitem[Gin{\'e} and Nickl(2021)]{GineNi21}
E.~Gin{\'e} and R.~Nickl.
\newblock \emph{Mathematical Foundations of Infinite-Dimensional Statistical
  Models}.
\newblock Cambridge University Press, 2021.

\bibitem[Goodfellow et~al.(2015)Goodfellow, Shlens, and
  Szegedy]{GoodfellowShSz15}
I.~Goodfellow, J.~Shlens, and C.~Szegedy.
\newblock Explaining and harnessing adversarial examples.
\newblock In \emph{Proceedings of the Third International Conference on
  Learning Representations}, 2015.

\bibitem[Hardt and Recht(2022)]{HardtRe22}
M.~Hardt and B.~Recht.
\newblock \emph{Patterns, Predictions, and Actions: A story about machine
  learning}.
\newblock Princeton University Press, 2022.
\newblock Available at \url{https://mlstory.org/}.

\bibitem[Hastie et~al.(2009)Hastie, Tibshirani, and Friedman]{HastieTiFr09}
T.~Hastie, R.~Tibshirani, and J.~Friedman.
\newblock \emph{The Elements of Statistical Learning}.
\newblock Springer, second edition, 2009.

\bibitem[He et~al.(2016)He, Zhang, Ren, and Sun]{HeZhReSu16}
K.~He, X.~Zhang, S.~Ren, and J.~Sun.
\newblock Deep residual learning for image recognition.
\newblock In \emph{Proceedings of the IEEE Conference on Computer Vision and
  Pattern Recognition}, pages 770--778, 2016.

\bibitem[Horn and Johnson(1985)]{HornJo85}
R.~A. Horn and C.~R. Johnson.
\newblock \emph{Matrix Analysis}.
\newblock Cambridge University Press, 1985.

\bibitem[Karger et~al.(2014)Karger, Oh, and Shah]{KargerOhSh14}
D.~Karger, S.~Oh, and D.~Shah.
\newblock Budget-optimal task allocation for reliable crowdsourcing systems.
\newblock \emph{Operations Research}, 62\penalty0 (1):\penalty0 1--24, 2014.

\bibitem[Krizhevsky and Hinton(2009)]{KrizhevskyHi09}
A.~Krizhevsky and G.~Hinton.
\newblock Learning multiple layers of features from tiny images.
\newblock Technical report, University of Toronto, 2009.

\bibitem[{Le Cam} and Yang(2000)]{LeCamYa00}
L.~{Le Cam} and G.~L. Yang.
\newblock \emph{Asymptotics in Statistics: Some Basic Concepts}.
\newblock Springer, 2000.

\bibitem[Ledoux and Talagrand(1991)]{LedouxTa91}
M.~Ledoux and M.~Talagrand.
\newblock \emph{Probability in Banach Spaces}.
\newblock Springer, 1991.

\bibitem[Lewis et~al.(2004)Lewis, Yang, Rose, and Li]{LewisYaRoLi04}
D.~Lewis, Y.~Yang, T.~Rose, and F.~Li.
\newblock {RCV}1: A new benchmark collection for text categorization research.
\newblock \emph{Journal of Machine Learning Research}, 5:\penalty0 361--397,
  2004.

\bibitem[Mammen and Tsybakov(1999)]{MammenTs99}
E.~Mammen and A.~B. Tsybakov.
\newblock Smooth discrimination analysis.
\newblock \emph{Annals of Statistics}, 27:\penalty0 1808--1829, 1999.

\bibitem[Marcus et~al.(1994)Marcus, Santorini, and Marcinkiewicz]{MarcusSaMa94}
M.~P. Marcus, B.~Santorini, and M.~A. Marcinkiewicz.
\newblock Building a large annotated corpus of {E}nglish: the {P}enn
  {T}reebank.
\newblock \emph{Computational Linguistics}, 19:\penalty0 313--330, 1994.

\bibitem[Owen(1990)]{Owen90}
A.~Owen.
\newblock Empirical likelihood ratio confidence regions.
\newblock \emph{The Annals of Statistics}, 18\penalty0 (1):\penalty0 90--120,
  1990.

\bibitem[Peterson et~al.(2019)Peterson, Battleday, Griffiths, and
  Russakovsky]{PetersonBaGrRu19}
J.~C. Peterson, R.~M. Battleday, T.~L. Griffiths, and O.~Russakovsky.
\newblock Human uncertainty makes classification more robust.
\newblock In \emph{Proceedings of the IEEE/CVF International Conference on
  Computer Vision}, pages 9617--9626, 2019.

\bibitem[Plan and Vershynin(2013)]{PlanVe13}
Y.~Plan and R.~Vershynin.
\newblock Robust 1-bit compressed sensing and sparse logistic regression: A
  convex programming approach.
\newblock \emph{IEEE Transactions on Information Theory}, 59\penalty0
  (1):\penalty0 482--494, 2013.

\bibitem[Plan and Vershynin(2016)]{PlanVe16}
Y.~Plan and R.~Vershynin.
\newblock The generalized lasso with non-linear observations.
\newblock \emph{IEEE Transactions on Information Theory}, 62\penalty0
  (3):\penalty0 1528--1537, 2016.

\bibitem[Platanios et~al.(2020)Platanios, Al-Shedivat, Xing, and
  Mitchell]{PlataniosAlXiMi20}
E.~A. Platanios, M.~Al-Shedivat, E.~Xing, and T.~Mitchell.
\newblock Learning from imperfect annotations.
\newblock \emph{arXiv:2004.03473 [cs.LG]}, 2020.

\bibitem[Radford et~al.(2021)Radford, Kim, Hallacy, Ramesh, Goh, Agarwal,
  Sastry, Askell, Mishkin, Clark, Krueger, and
  Sutskever]{RadfordKiHaRaGoAgSaAsMiClKrSu21}
A.~Radford, J.~W. Kim, C.~Hallacy, A.~Ramesh, G.~Goh, S.~Agarwal, G.~Sastry,
  A.~Askell, P.~Mishkin, J.~Clark, G.~Krueger, and I.~Sutskever.
\newblock Learning transferable visual models from natural language
  supervision.
\newblock In \emph{Proceedings of the 38th International Conference on Machine
  Learning}, 2021.

\bibitem[Ratner et~al.(2017)Ratner, Bach, Ehrenberg, Fries, Wu, and
  R\'{e}]{RatnerBaEhFrWuRe17}
A.~Ratner, S.~H. Bach, H.~Ehrenberg, J.~Fries, S.~Wu, and C.~R\'{e}.
\newblock Snorkel: rapid training data creation with weak supervision.
\newblock \emph{Proceedings of the VLDB Endowment}, 11\penalty0 (3):\penalty0
  269--282, 2017.

\bibitem[Raykar et~al.(2010)Raykar, Yu, Zhao, Valadez, Florin, Bogoni, and
  Moy]{RaykarYuZhVaFlBoMo10}
V.~C. Raykar, S.~Yu, L.~H. Zhao, G.~H. Valadez, C.~Florin, L.~Bogoni, and
  L.~Moy.
\newblock Learning from crowds.
\newblock \emph{Journal of Machine Learning Research}, 11\penalty0 (4), 2010.

\bibitem[Rosset et~al.(2004)Rosset, Zhu, and Hastie]{RossetZhHa04}
S.~Rosset, J.~Zhu, and T.~Hastie.
\newblock Boosting as a regularized path to a maximum margin classifier.
\newblock \emph{Journal of Machine Learning Research}, 5:\penalty0 941--973,
  2004.

\bibitem[Russakovsky et~al.(2015)Russakovsky, Deng, Su, Krause, Satheesh, Ma,
  Huang, Karpathy, Khosla, Bernstein, Berg, and
  Fei-Fei]{RussakovskyDeSuKrSaMaHuKaKhBeBeFe15}
O.~Russakovsky, J.~Deng, H.~Su, J.~Krause, S.~Satheesh, S.~Ma, Z.~Huang,
  A.~Karpathy, A.~Khosla, M.~Bernstein, A.~C. Berg, and L.~Fei-Fei.
\newblock Image{N}et large scale visual recognition challenge.
\newblock \emph{International Journal of Computer Vision}, 115\penalty0
  (3):\penalty0 211--252, 2015.

\bibitem[Schuhmann et~al.(2022)Schuhmann, Beaumont, Vencu, Gordon, Wightman,
  Cherti, Coombes, Katta, Mullis, Wortsman, Schramowski, Kundurthy, Crowson,
  Schmidt, Kaczmarczyk, and Jitsev]{SchuhmannBeVeGoWiChCoKaMuWoScKuCrScKaJi22}
C.~Schuhmann, R.~Beaumont, R.~Vencu, C.~Gordon, R.~Wightman, M.~Cherti,
  T.~Coombes, A.~Katta, C.~Mullis, M.~Wortsman, P.~Schramowski, S.~Kundurthy,
  K.~Crowson, L.~Schmidt, R.~Kaczmarczyk, and J.~Jitsev.
\newblock {LAION-5B}: An open large-scale dataset for training next generation
  image-text models.
\newblock In \emph{Advances in Neural Information Processing Systems 35}, 2022.

\bibitem[Shapiro et~al.(2009)Shapiro, Dentcheva, and
  Ruszczy\'nski]{ShapiroDeRu09}
A.~Shapiro, D.~Dentcheva, and A.~Ruszczy\'nski.
\newblock \emph{Lectures on Stochastic Programming: Modeling and Theory}.
\newblock SIAM and Mathematical Optimization Society, 2009.

\bibitem[Shevtsova(2014)]{Shevtsova14}
I.~Shevtsova.
\newblock On the absolute constants in the {B}erry-{E}sseen-type inequalities.
\newblock \emph{Doklady Mathematics}, 89\penalty0 (3):\penalty0 378--381, 2014.

\bibitem[Soudry et~al.(2018)Soudry, Hoffer, Nacson, Gunasekar, and
  Srebro]{SoudryHoNaGuSr18}
D.~Soudry, E.~Hoffer, M.~S. Nacson, S.~Gunasekar, and N.~Srebro.
\newblock The implicit bias of gradient descent on separable data.
\newblock \emph{Journal of Machine Learning Research}, 19\penalty0
  (18):\penalty0 1--57, 2018.

\bibitem[Tian and Zhu(2015)]{TianZh15}
T.~Tian and J.~Zhu.
\newblock Max-margin majority voting for learning from crowds.
\newblock In \emph{Advances in Neural Information Processing Systems 28}, pages
  1621--1629, 2015.

\bibitem[Ustalov et~al.(2024)Ustalov, Pavlichenko, and Tseitlin]{UstalovPaTs24}
D.~Ustalov, N.~Pavlichenko, and B.~Tseitlin.
\newblock Learning from crowds with crowd-kit.
\newblock \emph{Journal of Open Source Software}, 9\penalty0 (96):\penalty0
  6227, 2024.

\bibitem[van~der Vaart(1998)]{VanDerVaart98}
A.~W. van~der Vaart.
\newblock \emph{Asymptotic Statistics}.
\newblock Cambridge University Press, 1998.

\bibitem[van~der Vaart and Wellner(1996)]{VanDerVaartWe96}
A.~W. van~der Vaart and J.~A. Wellner.
\newblock \emph{Weak Convergence and Empirical Processes: With Applications to
  Statistics}.
\newblock Springer, New York, 1996.

\bibitem[Vapnik(1995)]{Vapnik95}
V.~Vapnik.
\newblock \emph{The Nature of Statistical Learning Theory}.
\newblock Springer, 1995.

\bibitem[Wainwright(2019)]{Wainwright19}
M.~J. Wainwright.
\newblock \emph{High-Dimensional Statistics: A Non-Asymptotic Viewpoint}.
\newblock Cambridge University Press, 2019.

\bibitem[Welinder et~al.(2010)Welinder, Branson, Perona, and
  Belongie]{WelinderBrPeBe10}
P.~Welinder, S.~Branson, P.~Perona, and S.~Belongie.
\newblock The multidimensional wisdom of crowds.
\newblock In \emph{Advances in Neural Information Processing Systems 23}, pages
  2424--2432, 2010.

\bibitem[Whitehill et~al.(2009)Whitehill, Wu, Bergsma, Movellan, and
  Ruvolo]{WhitehillWuBeMoRu09}
J.~Whitehill, T.~Wu, J.~Bergsma, J.~Movellan, and P.~Ruvolo.
\newblock Whose vote should count more: Optimal integration of labels from
  labelers of unknown expertise.
\newblock \emph{Advances in Neural Information Processing Systems 22},
  22:\penalty0 2035--2043, 2009.

\bibitem[Zhang et~al.(2021)Zhang, Bengio, Hardt, Recht, and
  Vinyals]{ZhangBeHaReVi21}
C.~Zhang, S.~Bengio, M.~Hardt, B.~Recht, and O.~Vinyals.
\newblock Understanding deep learning (still) requires rethinking
  generalization.
\newblock \emph{Communications of the ACM}, 64\penalty0 (3):\penalty0 107--115,
  2021.

\end{thebibliography}
\bibliographystyle{abbrvnat}
\newpage

\appendix
\section{Technical lemmas} \label{sec:technical}

We collect several technical lemmas and their proofs in this section, which
will be helpful in the main proofs. See
Appendices~\ref{proof:Sigma-inverse}, \ref{proof:t-Z-integral-noise},
\ref{proof:m-rho-integral-noise-general} and
\ref{proof:rademacher-symmetrization} for their proofs.

\begin{lemma} \label{lem:Sigma-inverse}
  Suppose $A, B \in \mathbb{R}^{d \times d}$ are symmetric, $AB = BA = 0$
  and the matrix $A+B$ is invertible. Then
  \begin{align*}
    \left(A +B\right)^{-1} = A^\dagger + B^\dagger.
  \end{align*}
\end{lemma}

The next two lemmas characterize asymptotic behaviors of expectations
involving a fixed function $f$ and some random variable $Z$ that satisfies
Assumption~\ref{assumption:z-density} for given $\beta>0$ and $c_Z <
\infty$. To facilitate stating the theorems, we recall
that such $Z$ are $(\beta,
c_Z)$-regular.
\begin{lemma} \label{lem:t-Z-integral-noise}
  Let $\beta > 0$ and $f$ be a function on $\R_+$ such that
  $z^{\beta-1}f(z)$ is integrable. If $Z$ is $(\beta,
  c_Z)$-regular (Assumption~\ref{assumption:z-density}), then
  \begin{align*}
    \lim_{t \to \infty} t^\beta \Ep [f(t|Z|)] =  c_Z \int_{0}^{\infty} z^{\beta - 1} f(z) dz.
  \end{align*}
\end{lemma}

\begin{lemma} \label{lem:m-rho-integral-noise-general}
  Let $\beta > 0$ and $c_Z < \infty$, and $Z$ be $(\beta, c_Z)$-regular, let
  $f : \R_+ \to \R$ satisfy $|f(z)| \leq a_0 + a_p z^p$ for some $a_0, a_p <
  \infty$ and all $z \in \R$, and assume $|Z|$ has finite $p$th
  moment. Additionally let $\rho_m(t)$ be the majority vote prediction
  function~\eqref{eq:link-function-m-majority-vote-general} and
  Assumption~\ref{assumption:condition-misspcfd-link-majority-vote} hold for
  each $\sigma_j\opt$ with limiting average derivative
  ${\wb{\sigma}\opt}'(0)$ at zero. Then for any $c > 0$
  \begin{align*}
    \lim_{m \to \infty} m^{\frac{\beta}{2}}\Ep \left[f(\sqrt{m}|Z|) (1 - \rho_m(cZ))\right] = c_Z \int_0^\infty z^{\beta - 1} f(z) \Phi \left(-2 {\wb{\sigma}\opt}'(0) cz\right) dz,
  \end{align*}
  where $\Phi(z) = \int_{-\infty}^z \frac{1}{\sqrt{2 \pi}}
  e^{-\frac{t^2}{2}}dt$ is the standard normal cumulative distribution function.
\end{lemma}

The fourth lemma is a uniform convergence result for the empirical risk in
the case that we have potentially distinct link functions
(cf.~Sec.~\ref{sec:master-results-sem}).
\begin{lemma}
  \label{lemma:rademacher-symmetrization}
  Assume $\E[\ltwo{X}^{\gamma}] \le M^\gamma$ for some $M \ge 1$ and $\gamma
  \ge 2$ and let the radius $1 \le r < \infty$. Let $\Flink \subset \{\sigma
  : \R \to [0, 1], \lipnorm{\sigma} \le \lipconst\}$.  Then for a constant
  $C \lesssim \sqrt{d \lipconst}$, we have
  \begin{equation*}
    \E\left[\sup_{\ltwo{\theta} \le r}
      \sup_{\vec{\sigma} \in \Flink^m} |P_n \loss_{\vec{\sigma},\theta} - L(\theta,
      \vec{\sigma})|\right]
    \le C \cdot
    (M r)^\frac{4\gamma}{3\gamma + 1} 
    \left(\frac{m}{n}\right)^\frac{\gamma}{3\gamma + 1}
    \sqrt{\log(rn)}.
  \end{equation*}
\end{lemma}

\subsection{Proof of Lemma~\ref{lem:Sigma-inverse}}
\label{proof:Sigma-inverse}

As $AB=BA=0$, the symmetric matrices $A$ and $B$ commute and so are
simultaneously orthogonally diagonalizable~\cite[Thm.~4.5.15]{HornJo85}. As
$AB=BA=0$, we can thus write
\begin{align*}
	A = U \begin{bmatrix} \Lambda_1 & 0 \\ 0 & 0 \end{bmatrix} U^\top  , \qquad B = U \begin{bmatrix} 0 & 0 \\ 0 & \Lambda_2 \end{bmatrix} U^\top  ,
\end{align*}
for some orthogonal $U \in \R^{d \times d}$, and as $A+B$ is invertible,  $\Lambda_1, \Lambda_2$ are invertible diagonal matrices. We conclude the proof by writing
\begin{align*}
	(A+ B)^{-1} = U \begin{bmatrix} \Lambda_1^{-1} & 0 \\ 0 & \Lambda_2^{-1} \end{bmatrix} U^\top = U \begin{bmatrix} \Lambda_1^{-1} & 0 \\ 0 & 0 \end{bmatrix} U^\top + U \begin{bmatrix} 0 & 0 \\ 0 & \Lambda_2^{-1} \end{bmatrix} U^\top = A^\dagger + B^\dagger .
\end{align*}

\subsection{Proof of Lemma~\ref{lem:t-Z-integral-noise}}  \label{proof:t-Z-integral-noise}
By the change of variables $w =t z$, we have
\begin{align*}
	t^\beta \Ep [f(t|Z|)] & = \int_{0}^{\infty} f(t z) \cdot t^{\beta - 1} p(z)  \cdot td z =  \int_{0}^{\infty} f(w) \cdot t^{\beta-1} p (w/t) d w \\
	& =  \int_{0}^{\infty} w^{\beta - 1} f(w) \cdot (w/t)^{1-\beta} p (w/t) d w.
\end{align*}
As $|  w^{\beta - 1} f(w) \cdot (w/t)^{1-\beta} p (w/t)| \leq \sup_{z \in (0, \infty)} z^{1-\beta} p(z) \cdot |w^{\beta - 1} f(w)|$, where $w^{\beta - 1} f(w)$ is integrable by assumption, we can invoke dominated convergence to see that
\begin{align*}
	\lim_{t \to \infty} t^\beta \Ep [f(t|Z|)] = \int_{0}^{\infty} w^{\beta - 1} f(w) \cdot \lim_{t \to \infty} (w/t)^{1-\beta} p (w/t) d w = c_Z \int_{0}^{\infty} w^{\beta - 1} f(w) dw.
\end{align*}

\subsection{Proof of Lemma~\ref{lem:m-rho-integral-noise-general}} \label{proof:m-rho-integral-noise-general}
By rescaling arguments, it suffices to prove the theorem for any function
$f$ satisfying
$|f(z)| \leq 1 + z^p$ on $\R_+$ for any $p \in \N$ such that $|Z|$
has finite $p$th moment.  For such $f$, we wish to
show
\begin{align*}
  \lim_{m \to \infty} m^{\frac{\beta}{2}} \Ep \left[f(\sqrt{m}|Z|) (1 - \rho_m(c Z))\right] = c_Z \int_0^\infty z^{\beta-1} f(z) \Phi \left(-2 {\wb{\sigma}\opt}'(0) c z\right) dz,
\end{align*}
where $\Phi$ is the standard normal cdf. The key insight is that we can
approximate $1 - \rho_m(t)$ by a suitable Gaussian cumulative distribution
function (recognizing that $\rho_m(t) > \half$ for $t \neq 0$ by
definition~\eqref{eq:link-function-m-majority-vote-general} as the
probability the majority vote is correct given margin
$\<\theta\opt, X\> = t$).

We first assume $Z \geq 0$ with probability $1$, as the general result
follows by writing $Z = (Z)_+ - (-Z)_+$.  We decompose into two
expectations, depending on $Z$ being large or small:
\begin{align}
  \lefteqn{m^{\frac{\beta}{2}} \Ep \left[f(\sqrt m |Z|) (1 - \rho_m(cZ)) \right]}
  \nonumber \\
  & = \underbrace{m^{\frac{\beta}{2}} \Ep \left[f(\sqrt m Z) (1 - \rho_m(cZ)) \ind\left\{0 \leq Z \leq \frac{M}{\sqrt m}\right\} \right]}_{\mathrm{(I)}} + \underbrace{m^{\frac{\beta}{2}} \Ep \left[f(\sqrt m Z) (1 - \rho_m(cZ)) \ind\left\{Z > \frac{M}{\sqrt m}\right\} \right]}_{\mathrm{(II)}}   .  \label{eq:normal-mid-main}
\end{align}
The proof consists of three main parts. 
\begin{enumerate}[leftmargin=2em,label=\arabic*.]
\item We approximate $1-\rho_m(t)$ by a Gaussian cdf.
\item We can approximate term (I) by replacing $1-\rho_m(t)$ with
  the Gaussian cdf, showing that
  \begin{align}
    \left|\lim_{m \to \infty} \mathrm{(I)} - c_Z \int_0^\infty z^{\beta - 1} f(z) \Phi \left( - 2{\wb{\sigma}\opt}'(0) cz \right) dz \right| = o_M(1)   . \label{eq:normal-mid-3}
  \end{align} 
\item For term (II), we show $1- \rho_m(cz)$ is small when $Z > M/\sqrt{m}$, which allows us to show that
  \begin{equation*}
    \limsup_{m \to \infty}
    \left|\mathrm{(II)} \right| = o_M(1).
  \end{equation*}
\end{enumerate} 
Thus by adding the two preceding displays and taking $M \to \infty$, we
obtain the lemma.

\newcommand{\linkstd}{\sigma_m^{\textup{std}}}

Before we dive into further details, we use the shorthand functions
\begin{equation}
  \linkstd(z)
  \defeq, \frac{\sum_{j=1}^m \prns{\sigma_j\opt(cz) - \half}}{\sqrt{\sum_{j=1}^m \sigma_j\opt(cz)\prns{1 - \sigma_j\opt(cz)}}},
  ~~~
  \Delta_m(z) \defeq 1 - \rho_m(cz) - \Phi(- \linkstd(z))
  \label{eqn:std-link},
\end{equation}
and we also write $p_\infty(\beta) := \sup_{z \in (0, \infty)} z^{1-\beta}
p(z) < \infty$.

\paragraph{Part 1. Normal approximation for $1- \rho_m(t)$ when $t = O(1/\sqrt{m})$.} Let $p_j = \sigma_j\opt(t)$ for shorthand and $Y_j \sim \mathrm{Bernoulli}(p_j)$ be independent random variables. For $t > 0$, then
\begin{align*}
  1 - \rho_m(t) & = \Prb \prn{Y_1 + \cdots + Y_m < \half} = \Prb
  \prn{\frac{\sum_{j=1}^m (Y_j - p_j)}{\sqrt{\sum_{j=1}^m p_j
        (1-p_j)}} < - \frac{\sum_{j=1}^m (p_j - \half)}{
      \sqrt{\sum_{j=1}^m p_j (1-p_j)}}} .
\end{align*}
Consider the centered and standardized random variables $\xi_j = (Y_j - p_j)/\sqrt{\sum_{j=1}^m p_j (1-p_j)} $ so that $\xi_1, \dots, \xi_m$ are zero mean, mutually independent, and satisfy
\begin{align*}
  \sum_{j=1}^m \var \prn{\xi_j} & = 1  , \\
  \sum_{j=1}^m \Ep \brk{\left|\xi_j\right|^3} & =
  \frac{\sum_{j=1}^m p_j (1-p_j)(p_j^2 + (1-p_j)^2)}{\prns{\sum_{j=1}^m p_j (1-p_j)}^{3/2}}
  \leq \frac{\max_{1 \leq j \leq m} (p_j^2 + (1-p_j)^2)}{\sqrt{\sum_{j=1}^m p_j(1-p_j)}}  .
\end{align*} 
By the Berry-Esseen theorem (cf.~\citet{ChenGoSh10, Shevtsova14}), for all
$t > 0$
\begin{align*}
	\left|1 - \rho_m(t) - \Phi \left(-\frac{ \sum_{j=1}^m \left( p_j - \half\right)}{\sqrt{\sum_{j=1}^m p_j (1-p_j)}}\right)\right| \leq \frac{3}{4} \cdot \frac{\max_{1 \leq j \leq m} (p_j^2 + (1-p_j)^2)}{\sqrt{\sum_{j=1}^m p_j (1-p_j)}}  .
\end{align*}
Fix any $M <\infty$. Then for $0 \leq t \leq cM /\sqrt m$ and large enough
$m$, the right hand side of the preceding display has the upper bound
$2/\sqrt{m}$, as in numerator we have $p_j^2 + (1-p_j)^2 \leq 1$, and in
denominator $\min_{1 \leq j \leq m}p_j(1-p_j) \to 1/4$ for all $j$ as $m \to
\infty$, which follows from
Assumption~\ref{assumption:condition-misspcfd-link-majority-vote}
that
\begin{equation*}
  \half \leq \limsup_{m \to \infty} \max_{1 \leq j \leq m} p_j \leq \limsup_{m \to \infty} \sup_{1 \leq j < \infty} {\sigma_j\opt} \prn{\frac{cM}{\sqrt m} } =
  \half.
\end{equation*}
By repeating the same argument for $-cM /\sqrt m \leq t < 0$, we obtain that for large $m$ and $|t| \leq cM / \sqrt m$,
\begin{align}
  \left|1 - \rho_m(t) - \Phi \left(-\left|\frac{ \sum_{j=1}^m \left( \sigma_j\opt(t) - \half\right)}{\sqrt{\sum_{j=1}^m \sigma_j\opt(t) \prn{1 - \sigma_j\opt(t)}}}\right|\right)\right| \leq \frac{2}{\sqrt m}   . \label{eq:BE-approximation}
\end{align}

\paragraph{Part 2. Approximating (I) by Gaussian cdf.}
For the first term (I) in~\eqref{eq:normal-mid-main}, we further decompose into a normal approximation term and an error term,
\begin{align*}
  \mathrm{(I)} = \underbrace{m^{\frac{\beta}{2}}
    \Ep \left[f(\sqrt m Z)\Phi(-\linkstd(z))
      \ind\left\{0 \leq Z \leq \frac{M}{\sqrt m}\right\} \right]}_{\mathrm{(III)}} + \underbrace{m^{\frac{\beta}{2}} \Ep \left[f(\sqrt m Z)\Delta_m(z) \ind\left\{0 \leq Z \leq \frac{M}{\sqrt m}\right\} \right]}_{\mathrm{(IV)}}, 
\end{align*}
where $\Delta_m(z) = 1 - \rho_m(cz) - \Phi(-\linkstd(z))$ as in
def.~\eqref{eqn:std-link}. We will show $\mathrm{(IV)} \to 0$ and so
$\mathrm{(III)}$ dominates. By the change of variables $w = \sqrt{m}z$, we can
further write (III) as
\begin{align*}
  \mathrm{(III)} 
  & = m^{\frac{\beta}{2}} \int_0^{\frac{M}{\sqrt m}} f(\sqrt m z) \Phi
  ( -\linkstd(z)) \cdot p(z) d z \nonumber \\
  & = \int_0^{M} w^{\beta - 1} f(w) \Phi  \left( - \linkstd\prn{\frac{w}{\sqrt m}}\right) \cdot  (w/\sqrt{m})^{1-\beta} p(w/\sqrt{m}) d w   .
\end{align*}
We want to take the limit $m \to \infty$ and apply dominated convergence
theorem. Because $(w/\sqrt{m})^{1-\beta} p(w/\sqrt{m}) \leq p_\infty(\beta)
< \infty$ and $\linkstd(w/\sqrt m) \geq 0$, we have
\begin{align*}
	&  w^{\beta - 1} f(w) \Phi  \left( - \linkstd\prn{\frac{w}{\sqrt m}}\right) \cdot  \left( - \linkstd\prn{\frac{w}{\sqrt m}}\right) \cdot  (w/\sqrt{m})^{1-\beta} p(w/\sqrt{m}) \leq w^{\beta - 1} f(w) \cdot \Phi(0) p_\infty(\beta)   .
\end{align*} 
As $\beta > 0$ and $|f(w)| \leq 1 + w^p$, $w^{\beta - 1} f(w)$ is integrable
on $[0, M]$, and by~\eqref{ass:misspcfd-link-1} and \eqref{ass:misspcfd-link-3}
  in Assumption~\ref{assumption:condition-misspcfd-link-majority-vote},
\begin{align*}
	\lim_{m \to \infty} \linkstd\prn{\frac{w}{\sqrt m}} & = \lim_{m \to \infty}\frac{\sqrt{m}\left(\wb{\sigma}_m\opt\prn{\frac{cw}{\sqrt m}}-\frac{1}{2}\right)}{\sqrt{ \frac{1}{m} \sum_{j=1}^m \sigma_j\opt\prn{\frac{cw}{\sqrt m}}\prn{1 - \sigma_j\opt\prn{\frac{cw}{\sqrt m}}}}} =  2 {\wb{\sigma}\opt}'(0) cw   .
\end{align*}
Using the above display and that $\lim_{m \to \infty} (w/\sqrt{m})^{1-\beta}
p(w/\sqrt{m}) = c_Z$, we can thus apply dominated convergence theorem
to conclude that
\begin{align*}
	\lim_{m \to \infty}\mathrm{(III)} 
	&  = c_Z \int_0^M w^{\beta - 1} f(w) \cdot  \Phi  \left( - 2 {\wb{\sigma}\opt}'(0) cw\right) dw = c_Z \int_0^\infty w^{\beta - 1} f(w) \Phi \left( - 2{\wb{\sigma}\opt}'(0) cw \right) dw + o_M(1). 
\end{align*} 

Next we turn to the error term (IV). By the
bound~\eqref{eq:BE-approximation}, $|\Delta_m(z)| \leq 2/\sqrt{m}$ when $|z|
\leq M/\sqrt{m}$ for large enough $m$, and substituting $w = \sqrt{m}z$,
\begin{align*}
	|\mathrm{(IV)}| & \leq m^{\frac{\beta}{2}} \Ep \left[|f(\sqrt m Z)| \cdot \frac{2}{\sqrt m} \cdot \ind\left\{0 \leq Z \leq \frac{M}{\sqrt m}\right\}   \right] \nonumber \\
	& = \frac{2}{\sqrt m} \int_0^{\frac{M}{\sqrt m}} \sqrt m \cdot |f(\sqrt m z)| \cdot m^{\frac{\beta-1}{2}} p(z) dz  = \frac{2}{\sqrt m} \int_0^M w^{\beta - 1}|f(w)| \cdot(w/\sqrt{m})^{1-\beta} p(w/\sqrt{m}) dw.
\end{align*}
By using Assumption~\ref{assumption:z-density} again that $(w/\sqrt{m})^{1-\beta} p(w/\sqrt{m}) \leq p_\infty(\beta) < \infty$, we further have
\begin{equation*}
  |\mathrm{(IV)}| 
  \leq \frac{2p_\infty(\beta) }{\sqrt m} \cdot \int_0^M w^{\beta - 1}|f(w)| dw
  \leq  \frac{2p_\infty(\beta) }{\sqrt m} \cdot
  \prn{\frac{M^\beta}{\beta} +\frac{M^{p+\beta}}{p+\beta}}
  \to 0,
\end{equation*}
where we use $|f(w)| \leq 1+w^p$. We have thus shown the
limit~\eqref{eq:normal-mid-3}.

\paragraph{Part 3. Upper bounding (II).}
In term~(II), when $Z$ is large, the key is that the quantity
$1-\rho_m(t)$ is
small when $|t| \geq cM /\sqrt{m}$: Hoeffding's
inequality implies the tail bound
\begin{align*}
  0 \leq 1 - \rho_m(t) \leq e^{-2 \left(\wb{\sigma}_m\opt(t) - \half\right)^2 m}  .
\end{align*}
Thus
\begin{align}
	|\mathrm{(II)}| &\leq  \int_{\frac{M}{\sqrt m}}^1 \sqrt m |f(\sqrt m z)|  e^{-2\left(\wb{\sigma}_m\opt(cz) - \half\right)^2m}\cdot  m^{\frac{\beta-1}{2}} p(z) dz + \int_{1}^\infty m^{\frac \beta 2} |f(\sqrt m z)|  e^{-2\left(\wb{\sigma}_m\opt(cz) - \half\right)^2m}\cdot  p(z) dz   . \nonumber  
\end{align}
Using the assumption that $\gamma \defeq\liminf_m \inf_{t \geq c}
\prn{\wb{\sigma}_m\opt(t) - \half} > 0$ from~\eqref{ass:misspcfd-link-2},
that $|f(z)| \leq 1 +z^p$ and $|Z|$ has finite $p$th moment, we observe
that
\begin{align*}
	\int_{1}^\infty m^{\frac \beta 2} |f(\sqrt m z)|  e^{-2\left(\wb{\sigma}_m\opt(cz) - \half\right)^2m}\cdot  p(z) dz\leq  m^{\frac{1}{2}(\beta + p)}  e^{-2\gamma^2m} \int_1^\infty (1 + z^p) p(z) dz \to 0  ,
\end{align*}
and consequently
\begin{align*} 
	\limsup_{m \to \infty} |\mathrm{(II)}|  & \leq  \limsup_{m \to \infty} \int_{\frac{M}{\sqrt m}}^1 \sqrt m |f(\sqrt m z)|  e^{-2\left(\wb{\sigma}_m\opt(cz) - \half\right)^2m}\cdot  m^{\frac{\beta-1}{2}} p(z) dz \nonumber \\
	& = \limsup_{m \to \infty}   \int_{M}^{\sqrt{m}}  w^{\beta - 1} |f(w)| e^{-2\left(\wb{\sigma}_m\opt\prn{\frac{cw}{\sqrt m}} - \half\right)^2m} \cdot (w/\sqrt{m})^{1-\beta} p(w/\sqrt{m}) d w   .
\end{align*}
For $w \in [M, \sqrt m]$, we have
\begin{align*}
	\left(\wb{\sigma}_m\opt\prn{\frac{cw}{\sqrt m}} - \half\right)^2m & = \left(\frac{\wb{\sigma}_m\opt\prn{\frac{cw}{\sqrt m}} - \half}{\frac{cw}{\sqrt m}}\right)^2 c^2 w^2 \geq \left(\inf_{0 < t \leq c} \frac{\wb{\sigma}_m\opt(t) - \half}{t}\right)^2 c^2 w^2   ,
\end{align*}
while Assumption~\eqref{ass:misspcfd-link-2} gives $\delta:=\inf_{0 < t \leq c} \frac{\wb{\sigma}_m\opt(t) - \half}{t} > 0$,  so
\begin{align*}
\limsup_{m \to \infty} |\mathrm{(II)}| & \leq \int_{M}^{\sqrt{m}}  w^{\beta - 1} |f(w)| e^{-\delta w^2} \cdot (w/\sqrt{m})^{1-\beta} p(w/\sqrt{m}) d w. 
\end{align*}
Using the inequality
$(w/\sqrt{m})^{1-\beta} p(w/\sqrt{m}) \leq p_\infty(\beta) < \infty$, we
apply dominated convergence:
\begin{align*}
  \lim_{m \to \infty} \int_{M}^{\sqrt{m}}  w^{\beta - 1} |f(w)| e^{-\delta w^2} \cdot (w/\sqrt{m})^{1-\beta} p(w/\sqrt{m}) d w  = c_Z \int_M^\infty  w^{\beta - 1}  |f(w)| e^{-\delta w^2} dw = o_M(1).
\end{align*}

\subsection{Proof of Lemma~\ref{lemma:rademacher-symmetrization}} \label{proof:rademacher-symmetrization}
We follow a typical symmetrization approach, then construct a covering
that we use to prove the lemma.
Let $P_n^0 = n^{-1} \sum_{i = 1}^n \varepsilon_i 1_{X_i, Y_i}$ be the
(random) symmetrized measure with point masses at $(X_i, Y_i)$ for
$Y_i = (Y_{i1}, \ldots, Y_{im})$. Then by a standard symmetrization
argument, we have
\begin{equation}
	\label{eqn:symmetrization-step}
	\E\left[\sup_{\ltwo{\theta} \le r}
	\sup_{\vec{\sigma} \in \Flink^m} |P_n \loss_{\vec{\sigma},\theta} - L(\theta,
	\vec{\sigma})|\right]
	\le 2 \E\left[\sup_{\ltwo{\theta} \le r}
	\sup_{\vec{\sigma} \in \Flink^m} \left|P_n^0 \loss_{\vec{\sigma}, \theta}
	\right|\right].
\end{equation}

We use a covering argument to bound the symmetrized
expectation~\eqref{eqn:symmetrization-step}.  Let $R < \infty$ to be
chosen, and for an (again, to be determined) $\epsilon > 0$ let $\mc{G}
\subset \Flink$
denote an $\epsilon$-cover of $\Flink$ in the supremum norm on $[-R, R]$,
that is, $\norm{g - \sigma} = \sup_{t \in [-R,R]} |g(t) - \sigma(t)|$, and
so for each $\sigma \in \Flink$ there exists $g \in \mc{G}$ such that
$\norm{g - \sigma} \le \epsilon$. Then~\cite[Ch.~2.7]{VanDerVaartWe96} we
have $\log \card(\mc{G}) \le O(1) \frac{R \lipconst}{\epsilon}$. Let
$\Theta_\epsilon$ be a minimal $\epsilon$-cover of $\{\theta \mid
\ltwo{\theta} \le r\}$ in $\ltwo{\cdot}$, so that $\log
\card(\Theta_\epsilon) \le d \log(1 + \frac{2r}{\epsilon})$
and $\max_{\theta \in \Theta_\epsilon} \ltwo{\theta} \le r$.
We claim that for each $\ltwo{\theta} \le r$ and $\sigma
\in \Flink$, there exists
$v \in \Theta_\epsilon$ and $g \in \mc{G}$ such that
\begin{equation}
	\left|\frac{1}{n} \sum_{i = 1}^n \varepsilon_i
	\left(\loss_{\sigma, \theta}(Y_{ij} \mid X_i)
	-
	\loss_{g, v}(Y_{ij} \mid X_i)\right)
	\right|
	\le \epsilon +
	\frac{1}{n}\sum_{i = 1}^n \ltwo{X_i} \epsilon
	+ \frac{1}{n} \sum_{i = 1}^n\indic{\ltwo{X_i} \ge R / r}.
	\label{eqn:covering-diff}
\end{equation}
Indeed, for any $g \in \Flink$ and $\theta, v \in \R^d$, we have
\begin{align*}
	\left|\loss_{\sigma, \theta}(y \mid x)
	- \loss_{g, v}(y \mid x)\right|
	& = \left|\int_0^{y \<x, \theta\>} \sigma(-t) dt
	- \int_0^{y \<x, v\>} g(-t) dt\right| \\
	& \le \sup_{|t| \le r \ltwo{x}} |\sigma(t) - g(t)|
	+ \left|\int_{y\<x, v\>}^{y\<x, \theta\>}
	|\sigma(-t) - g(-t)| dt \right| \\
	& \le \norm{\sigma - g} + \indic{\ltwo{x} \ge R/r}
	+ |\<x, \theta - v\>| \\
	& \le
	\norm{\sigma - g} + \ltwo{x} \ltwo{\theta - v}
	+ \indic{\ltwo{x} \ge R/r},
\end{align*}
where we have used that $\sigma, g \in [0, 1]$.
Taking the elements $g, v$ in the respective coverings to minimize
the above bound gives the guarantee~\eqref{eqn:covering-diff}.

We now leverage inequality~\eqref{eqn:covering-diff}
in the symmetrization step~\eqref{eqn:symmetrization-step}.
We have
\begin{align*}
	\lefteqn{\E\left[\sup_{\ltwo{\theta} \le r}
		\sup_{\vec{\sigma} \in \Flink^m} |P_n \loss_{\vec{\sigma},\theta} - L(\theta,
		\vec{\sigma})|\right]} \\
	& \lesssim
	\E\left[\max_{\theta \in \Theta_\epsilon}
	\max_{\vec{g} \in \mc{G}^m}
	\left|P_n^0 \loss_{\vec{g}, \theta}\right|
	\right]
	+ \epsilon + \E[\ltwo{X_1}] \epsilon
	+ \frac{1}{n} \sum_{i = 1}^n \P(\ltwo{X_i} \ge R / r) \\
	& \stackrel{(i)}{\lesssim}
	\sqrt{\frac{d m R \lipconst}{\epsilon}
		\log\left(1 + \frac{2r}{\epsilon}\right)}
	\E\left[\frac{r^2}{n}\sum_{i = 1}^n \ltwo{X_i}^2\right]^{1/2}
	+ \epsilon + \E[\ltwo{X_1}] \epsilon
	+ \frac{\E[\ltwo{X_i}^{\gamma}]}{(R/r)^{\gamma}} \\
	& \le
	M r \sqrt{\frac{dm R \lipconst}{n \epsilon}
		\log \left(1 + \frac{2r}{\epsilon}\right)}
	+ (M + 1) \epsilon
	+ \frac{M^{\gamma}}{(R/r)^{\gamma}},
\end{align*}
where inequality $(i)$ uses that if $Z_i$ are $\tau^2$-sub-Gaussian, then
$\E[\max_{i \le N} |Z_i|] \le \sqrt{2 \tau^2 \log N}$, and that
conditional on $\{X_i, Y_i\}_{i = 1}^n$, the symmetrized sum $\sum_{i =
	1}^n \varepsilon_i \frac{1}{m} \sum_{j = 1}^m \loss_{g_j, \theta}(Y_{ij}
\mid X_i)$ is $r^2 \sum_{i = 1}^n \ltwo{X_i}^2$-sub-Gaussian, as
$|\loss_{g_j, \theta}(Y_{ij} \mid X_i)| \le |\<X_i, \theta\>| \le r
\ltwo{X_i}$.  We optimize this bound to get the final guarantee of the
lemma: set $\epsilon = (Rm / n)^{1/3}$ (and note that we will
choose $R \ge 1$) to obtain
\begin{equation*}
	\E\left[\sup_{\ltwo{\theta} \le r}
	\sup_{\vec{\sigma} \in \Flink^m} |P_n \loss_{\vec{\sigma},\theta} - L(\theta,
	\vec{\sigma})|\right]
	\lesssim \sqrt{d \lipconst \log(rn)}
	M r \left(\frac{R m}{n}\right)^{1/3}
	+ \frac{(Mr)^\gamma}{R^\gamma}.
\end{equation*}
Choose $R
= ((Mr)^{3(\gamma - 1)} n / m)^\frac{1}{3\gamma + 1}$.

\section{Proof of Lemma~\ref{lemma:generic-solutions}}
\label{sec:proof-generic-solutions}

Recall that $h = h_{t\opt, m}$ is the calibration
gap function~\eqref{eqn:h-m-func}.
We see that because $\E[Z] = 0$, we have
$h(0) = -2 \E[Z \varphi_m(t\opt Z)] < 0$,
while using Assumption~\ref{assumption:model-link} and that
$\E[|Z|] < \infty$, we apply dominated convergence and that
$\lim_{t \to \infty} (\sigma(t Z) - \half) Z = c|Z|$ with probability 1
to obtain
\begin{align*}
	\lim_{t \to \infty} h(t) = \E[c|Z| (1 - \varphi_m(t\opt Z))]
	+ \E[c |Z| \varphi_m(t\opt Z)] = c \E[|Z|] > 0.
\end{align*}
Because $h'(t) = \E[\sigma'(tZ) Z^2 (1 - \varphi_m(t\opt Z))]
+ \E[\sigma'(-tZ) Z^2 \varphi_m(t\opt Z)] > 0$, we see that there
is a unique $t_m$ solving $h(t_m) = 0$, and evidently
$\theta\opt_L = t_m u\opt$ is a minimizer of $L$.

We compute the Hessian of $L$. For this,
we again let $\theta = t u\opt$ to write
\begin{align*}
	\nabla^2 L(\theta, \sigma) & = \E[\sigma'(-\<\theta, X\>) XX^\top \varphi_m(\<X, \theta\opt\>)]
	+ \E[\sigma'(\<\theta, X\>) XX^\top (1 - \varphi_m(\<X, \theta\opt\>))] \\
	& = \E\left[\sigma'(-t Z) \varphi_m(t\opt Z)
	\left(Z^2 u\opt {u\opt}^\top + WW^\top\right)\right]
	+ \E\left[\sigma'(t Z) (1 - \varphi_m(t\opt Z))
	\left(Z^2 u\opt {u\opt}^\top + WW^\top \right)\right] \\
	& = \E\left[\left(\sigma'(-tZ) \varphi_m(t\opt Z)
	+ \sigma'(tZ) (1 - \varphi_m(t\opt Z)) \right) Z^2\right]
	u\opt {u\opt}^\top \\
	& \quad + \E\left[\sigma'(-tZ) \varphi_m(t\opt Z)
	+ \sigma'(tZ) (1 - \varphi_m(t\opt Z))\right] \E[WW^\top],
\end{align*}
which gives $\nabla^2 L(\theta, \sigma) \succ 0$ and so $\theta\opt_L$ is
unique.  Finally, the desired form of the Hessian follows as $\E[WW^\top] =
\projperp \Sigma \projperp$.

\section{Proof of Theorem~\ref{theorem:multilabel-asymptotics}}
\label{sec:proof-multilabel-asymptotics}

Let $t_m$ be the solution to $h_{t\opt,m}(t) = 0$ as in
Lemma~\ref{lemma:generic-solutions}. The consistency argument is
immediate: the losses $\loss$ are convex, continuous, and locally Lipschitz,
so~\citet[Thm.~5.4]{ShapiroDeRu09} gives $\what{\theta}_n
\cas \theta\opt_L$. By an appeal to standard
M-estimator theory~\cite[e.g.][Thm.~5.23]{VanDerVaart98},
we thus obtain
\begin{equation}
	\label{eqn:multilabel-basic-asymptote}
	\sqrt{n} \big(\what{\theta}_n - \theta\opt_L\big)
	\cd \normal\left(0, \nabla^2 L(\theta\opt_L, \sigma)^{-1}
	\cov\bigg(\frac{1}{m} \sum_{j = 1}^m
	\nabla \loss_{\sigma,\theta\opt_L}(Y_j \mid X)\bigg)
	\nabla^2 L(\theta\opt_L, \sigma)^{-1} \right).
\end{equation}
We expand the covariance term to obtain the first main result of the
theorem. For shorthand, let $G_j = \nabla \loss_{\sigma,\theta\opt_L}(Y_j
\mid X)$, so that $\sum_{j = 1}^m \E[G_j] = 0$ and the $G_j$ are
conditionally independent given $X$. Applying the law of total covariance,
we have
\begin{align*}
	\cov\bigg(\frac{1}{m} \sum_{j = 1}^m G_j\bigg)
	& = \cov\bigg(\frac{1}{m} \sum_{j = 1}^m
	\E[G_j \mid X]\bigg)
	+ \E\left[\cov\bigg(
	\frac{1}{m} \sum_{j = 1}^m G_j \mid X\bigg)\right] \\
	& =
	\underbrace{\cov\bigg(\frac{1}{m} \sum_{j = 1}^m \E[G_j \mid X] \bigg)}_{
		\mathrm{(I)}}
	+ \frac{1}{m^2} \sum_{j = 1}^m
	\underbrace{\E\left[\cov(G_j \mid X)\right]}_{\mathrm{(II)}},
\end{align*}
where we have used the conditional independence of the $Y_j$ conditional on
$X$.  We control each of terms (I) and (II) in turn.

For the first, we have by the independent decomposition
$X = Z u\opt + \projperp W$ that
\begin{equation*}
	\E[G_j \mid X]
	= 
	\left(\sigma(t_m Z)
	(1 - \sigma_j\opt(t\opt Z))
	- \sigma(- t_m Z) \sigma_j\opt(t\opt Z)\right)
	X,
\end{equation*}
and so
\begin{align*}
	\mathrm{(I)} & =
	\E\left[\left(\sigma(t_m Z) (1 - \wb{\sigma}\opt(t\opt Z))
	- \sigma(-t_m Z) \wb{\sigma}\opt(t\opt Z)\right)^2 XX^\top \right] \\
	& =
	\E\left[\left(\sigma(t_m Z) (1 - \wb{\sigma}\opt(t\opt Z))
	- \sigma(-t_m Z) \wb{\sigma}\opt(t\opt Z)\right)^2 Z^2 \right]
	u\opt {u\opt}^\top
	\\ & \qquad ~ + 
	\E\left[\left(\sigma(t_m Z) (1 - \wb{\sigma}\opt(t\opt Z))
	- \sigma(-t_m Z) \wb{\sigma}\opt(t\opt Z)\right)^2\right]
	\projperp \Sigma \projperp \\
	& = \E[\linkerr(Z)^2 Z^2] u\opt {u\opt}^\top
	+ \E[\linkerr(Z)^2]   \projperp \Sigma \projperp,
\end{align*}
where we used the independence of $W$ and $Z$.
For term (II) above, we see that conditional on $X$,
$\nabla \loss_{\sigma,\theta}(Y_j \mid X)$ is a binary
random variable taking values in
$\{-\sigma(-t_m Z) X, \sigma(t_m Z) X\}$ with probabilities
$\sigma_j\opt(t\opt Z)$ and
$1 - \sigma_j\opt(t\opt Z)$, respectively, so that
a calculation leveraging the variance of Bernoulli random variables
yields
\begin{align*}
	\cov(G_j \mid X)
	& = \sigma_j\opt(t\opt Z) (1 - \sigma_j\opt(t\opt Z))
	\left(\sigma(t_m Z) + \sigma(-t_m Z)\right)^2 XX^\top
	= v_j(Z) XX^\top,
\end{align*}
where we used the definition~\eqref{eqn:error-functions} of the
variance terms.
A similar calculation to that we used for term (I) then gives that
\begin{align*}
	\mathrm{(II)} & = \frac{1}{m^2}
	\sum_{j = 1}^m \E[v_j(Z) Z^2] u\opt {u\opt}^\top
	+ \frac{1}{m^2}
	\sum_{j = 1}^m \E[v_j(Z)] \projperp \Sigma \projperp.
\end{align*}
Applying Lemma~\ref{lem:Sigma-inverse} then allows us to decompose the
covariance in expression~\eqref{eqn:multilabel-basic-asymptote} into
terms in the span of $u\opt {u\opt}^\top$ and those perpendicular to
it, so that the asymptotic covariance is
\begin{align*}
	\frac{\E[\linkerr(Z)^2 Z^2] + \frac{1}{m^2}
		\sum_{j = 1}^m \E[v_j(Z) Z^2]}{
		\E[\hessfunc(Z) Z^2]^2} u\opt {u\opt}^\top
	+ \frac{\E[\linkerr(Z)^2] + \frac{1}{m^2}
		\sum_{j = 1}^m \E[v_j(Z)]}{\E[\hessfunc(Z)]^2}
	\prn{\projperp \Sigma \projperp}^\dagger,
\end{align*}
by applying Lemma~\ref{lemma:generic-solutions} for the form of the Hessian
$\nabla^2 L(\theta_L\opt, \sigma)$ and Lemma~\ref{lem:Sigma-inverse} for the
inverse, giving the first result of the theorem.

To obtain the second result, we apply the delta method with
the mapping $\phi(x) = x / \ltwo{x}$, which satisfies
$\nabla \phi(x) = (I - \phi(x) \phi(x)^\top) / \ltwo{x}$,
so that for $\what{u}_n = \what{\theta}_n / \ltwos{\what{\theta}_n}$ we have
\begin{align*}
  \sqrt{n}(\what{u}_n - u\opt)
  & \cd
  \normal\left(0, \frac{1}{t_m^2}
  \frac{\E[\linkerr(Z)^2] + \frac{1}{m^2}
    \sum_{j = 1}^m \E[v_j(Z)]}{\E[\hessfunc(Z)]^2}
  \prn{\projperp \Sigma \projperp}^\dagger
  \right)
\end{align*}
as desired.

\section{Proof of Theorem~\ref{theorem:majority-vote-asymptotics}}
\label{sec:proof-majority-vote-asymptotics}

As in the proof of Theorem~\ref{theorem:multilabel-asymptotics}, we begin
with a consistency result. Let $t_m$ be the solution to $h_{t\opt, m}(t) =
0$ as in Lemma~\ref{lemma:generic-solutions}.
Once again, \citet[Thm.~5.4]{ShapiroDeRu09}
shows that $\what{\theta}_n \cas \theta\opt_L$.
As previously, appealing to standard
M-estimator theory~\cite[e.g.][Thm.~5.23]{VanDerVaart98},
we obtain
\begin{equation}
	\label{eqn:vote-basic-asymptote}
	\sqrt{n} \big(\what{\theta}_n - \theta\opt_L\big)
	\cd \normal\left(0, \nabla^2 L(\theta\opt_L, \sigma)^{-1}
	\cov \left(\nabla \loss_{\sigma,\theta\opt_L}(\majY \mid X)\right)
	\nabla^2 L(\theta\opt_L, \sigma)^{-1} \right),
\end{equation}
the difference from the asymptotic~\eqref{eqn:multilabel-basic-asymptote}
appearing in the covariance term.
Here, we recognize that conditional on $X = Z u\opt + W$, the vector
$\nabla \loss_{\sigma,\theta\opt_L}(\majY \mid X)$
takes on the values
$\{-\sigma(-t_m Z) X, \sigma(t_m Z) X\}$ each with probabilities
$\varphi_m(t\opt Z)$ and $1 - \varphi_m(t\opt Z)$, respectively,
while $\nabla \loss_{\sigma,\theta\opt_L}(\majY \mid X)$ is (unconditionally)
mean zero.
Thus
we have
\begin{align*}
	\cov \left(\nabla \loss_{\sigma,\theta\opt_L}(\majY \mid X)\right)
	& = \E\left[\sigma(-t_m Z)^2 \varphi_m(t\opt Z) XX^\top
	+ \sigma(t_m Z)^2 (1 - \varphi_m(t\opt Z)) XX^\top\right] \\
	& = \E\left[\left(\sigma(- t_m Z)^2 \varphi_m(t\opt Z)
	+ \sigma(t_m Z)^2 (1 - \varphi_m(t\opt Z))\right) Z^2\right]
	u\opt{u\opt}^\top \\
	& \qquad ~ +
	\E\left[\sigma(- t_m Z)^2 \varphi_m(t\opt Z)
	+ \sigma(t_m Z)^2 (1 - \varphi_m(t\opt Z))\right]
	\projperp \Sigma \projperp.
\end{align*}

Applying Lemma~\ref{lem:Sigma-inverse}
as in the proof of Theorem~\ref{theorem:multilabel-asymptotics}
to  decompose the covariance
terms in the asymptotic~\eqref{eqn:vote-basic-asymptote}, and substituting in $\rho_m(t) = \sigma(t) \ind \{t \geq 0\} + (1- \sigma(t)) \ind \{t < 0\}$,
the limiting covariance in expression~\eqref{eqn:vote-basic-asymptote}
becomes
\begin{align*}
  & \frac{\E[(\sigma(-t_m |Z|)^2 \rho_m(t\opt Z) + \sigma(t_m |Z|)^2
      (1 - \rho_m(t\opt Z))) Z^2]}{
    \E[\hessfunc(Z) Z^2]^2}
  u\opt {u\opt}^\top \\
  & \qquad ~ + 
  \frac{\E[\sigma(-t_m |Z|)^2 \rho_m(t\opt Z) + \sigma(t_m |Z|)^2
      (1 - \rho_m(t\opt Z))]}{
    \E[\hessfunc(Z)]^2}
  \prn{\projperp \Sigma \projperp}^\dagger.
\end{align*}
Lastly, we apply the delta method to $\phi(x) = x / \ltwo{x}$,
exactly as in the proof of Theorem~\ref{theorem:multilabel-asymptotics},
which gives the theorem.


\section{Proofs of asymptotic normality}

In this appendix, we include proofs of the convergence results in
Propositions~\ref{prop:lowd-maximum-likelihood},
\ref{prop:lowd-majority-vote-m-noise}, and
\ref{prop:lowd-misspcfd-majority-vote-m}. In each, we divide the proof into
three steps: we characterize the loss minimizer, apply one of the master
Theorems~\ref{theorem:multilabel-asymptotics}
or~\ref{theorem:majority-vote-asymptotics} to obtain asymptotic normality,
and then characterize the behavior of the asymptotic covariance as $m \to
\infty$.

\subsection{Proof of Proposition~\ref{prop:lowd-maximum-likelihood}}
\label{proof:lowd-maximum-likelihood}

\paragraph{Asymptotic normality of the MLE.}
The asymptotic
normality result is an
immediate consequence of the classical asymptotics for maximum likelihood
estimators~\cite[Thm.~5.29]{VanDerVaart98}.

\paragraph{Normalized estimator.}
For the normalized estimator, we appeal to the master results developed in Section~\ref{sec:master-results}.
In particular, since we are in the well-specified logistic model, we can invoke Corollary~\ref{cor:well-specified-normalized-cov} and write directly that
\begin{equation*}
	\sqrt{n}(\what{u}_{n,m}\lr- u\opt)
	\cd \normal\left(0, \frac{1}{m} \cdot
	\frac{1}{{t\opt}^2}
	\frac{\E[\sigma\lr(t\opt Z) (1 - \sigma\lr(t\opt Z))]}{\E[{\sigma\lr}'(t\opt Z)]^2}
	\projperp \Sigma \projperp \right),
\end{equation*}
which immediately implies
\begin{align*}
	C(t) & = \frac{1}{t^2}
	\frac{\E[\sigma\lr(t Z) (1 - \sigma\lr(t Z))]}{\E[{\sigma\lr}'(t Z)]^2} = \frac{1}{t^2\Ep\left[\frac{e^{t Z} }{(1 + e^{t Z})^2}\right]},
\end{align*}
and that further $ C(t) {t}^{2 - \beta} = {t}^{-\beta} \Ep[\frac{e^{t|Z|}
  }{(1 + e^{t |Z|})^2}]^{-1}$. To compute the limit when $t \to
\infty$, we invoke Lemma~\ref{lem:t-Z-integral-noise} and
we conclude that
\begin{align*}
  \lim_{t \to \infty} C(t) t^{2 - \beta} = \lim_{t\to \infty}
  \frac{1}{t^\beta \Ep[\frac{e^{t |Z|} }{(1 + e^{t |Z|})^2}]}
  = \frac{1}{ c_Z \int_{0}^{\infty} \frac{z^{\beta - 1} e^{z} }{(1 + e^{z})^2}
    dz}.
\end{align*}

\subsection{Proof of Proposition~\ref{prop:lowd-majority-vote-m-noise}}
\label{proof:lowd-majority-vote-m-noise}

\paragraph{Minimizer of the population loss.}
We can see identity~\eqref{eq:majority-vote-root} still holds with the
calibration gap $h(t) = h_m(t) = \E[|Z|(1 - \rho_m(t\opt |Z|))] -
\E[\frac{|Z|}{1 + e^{t|Z|}}]$ in Eq.~\eqref{eqn:h-centering} as $X-u\opt Z$
and $u\opt Z$ are independent. The function $h(t)$ is monotonically
increasing in $t$ with $h(\infty) =\Ep \brk{|Z|(1-\rho_m(t^\star |Z|))} >
0$, while $1 - \rho_m(t|Z|) \leq 1 - \rho_1(t|Z|) = \frac{1}{1 + e^{t|Z|}}$,
we must have $h(t^\star) \leq 0$. Therefore there must be a unique zero
point $t_m \ge t\opt$ of $h(t)$, and so $t_m u\opt$ is the unique minimizer
of the population loss $\poploss\mv_m(\theta)$.

\paragraph{Asymptotic variance.}
As $t_m$ solves $h_m(t_m) = 0$, 
Eq.~\eqref{eq:majority-vote-root} guarantees that $t_m u\opt$ is the global
minimizer of the population loss $\poploss\mv_m$. Appealing to
Theorem~\ref{theorem:majority-vote-asymptotics}, it follows that
$\what{\theta}\mv_{n,m} \cp t_m u\opt$, and
\begin{align*}
  \sqrt{n} (\what{u}\mv_{n,m} -  u^\star)&\cd \normal\left(0, \varfunc(t\opt)
  \prn{\proj_{u\opt} \Sigma \proj_{u\opt}}^\dagger\right)
\end{align*}
for the variance function~\eqref{eqn:generic-majority-covariance},
which in this case simplifies to
\begin{align*}
  \varfunc(t\opt)
  & = \frac{\Ep[\frac{1}{(1 + e^{t_m|Z|})^2} \rho_m(t\opt |Z|)
      + \frac{1}{(1 + e^{-t_m|Z|})^2} (1-\rho_m(t\opt |Z|)) ]}{
    t_m^2 \Ep[\frac{e^{t_m Z} }{ (1 + e^{t_m Z})^2}]^2}
\end{align*}
via the symmetry $\rho_m(t) = \rho_m(-t)$.

\paragraph{Large $m$ behavior.} The remainder of the
proof is to characterize the behavior of $\varfunc(t\opt)$ as $m\to\infty$.
We first derive asymptotics for $t_m$. To simplify notation, we let
$\ltwo{\theta} = t = t\opt$. Because $t_m$ solves $h_m(t_m) = 0$ we have
\begin{align}
  \Ep \brk{\frac{|Z|}{1 + e^{t_m|Z|}}} = \Ep \brk{|Z|(1-\rho_m(t |Z|))},
  \label{eqn:solving-for-t-m}
\end{align}
we must have $t_m \to \infty$ as $m \to \infty$ because $\rho_m(t) \to 1$
for any $t > 0$ as $m \to \infty$, so the right side of
equality~\eqref{eqn:solving-for-t-m} converges to 0 by the dominated
convergence theorem, and hence so must the left hand side.  Invoking
Lemma~\ref{lem:t-Z-integral-noise} for the left hand side, it follows that
\begin{align*}
  \lim_{m \to \infty} t_m^{\beta + 1} \Ep \brk{\frac{|Z|}{1 + e^{t_m|Z|}}} & = \lim_{m \to \infty} t_m^{\beta} \Ep \brk{\frac{t_m|Z|}{1 + e^{t_m|Z|}}}  =  c_Z \int_{0}^{\infty} \frac{z^\beta}{1 + e^{z}} dz, 
\end{align*}
while invoking Lemma~\ref{lem:m-rho-integral-noise-general} for the right
hand side of~\eqref{eqn:solving-for-t-m}, it follows that
\begin{align*}
  \lim_{m \to \infty} m^{\frac{\beta+1}{2}}\Ep \brk{|Z|(1-\rho_m(t  |Z|))} & = \lim_{m \to \infty} m^{\frac{\beta}{2}} \Ep \brk{\sqrt m|Z|(1-\rho_m(t  |Z|))} \nonumber \\
  &  = c_Z \int_0^\infty z^\beta \Phi \left(-\frac{tz}{2}\right) dz = c_Z {t}^{-\beta-1}\int_0^\infty z^\beta \Phi \left(-\frac{z}{2}\right) dz, 
\end{align*}
where the last line follows from change of variables $tz \mapsto z$. The
identity~\eqref{eqn:solving-for-t-m}
implies that the ratio $\Ep [|Z|/(1 + e^{t_m|Z|})] / \Ep [|Z|(1-\rho_m(t
  |Z|))] = 1$ and so we have that as $m \to \infty$,
\begin{align*}
		\frac{t_m}{\sqrt{m}} & = \left(\frac{	t_m^{\beta + 1}}{ m^{\frac{\beta+1}{2}}}\right)^{\frac{1}{\beta +1}} =  \left(\frac{	t_m^{\beta + 1}\Ep \brk{\frac{|Z|}{1 + e^{t_m|Z|}}}}{ m^{\frac{\beta+1}{2}} \Ep \brk{|Z|(1-\rho_m(t |Z|))} }\right)^{\frac{1}{\beta +1}}  \to  \left(\frac{\int_{0}^{\infty} \frac{z^\beta}{1 + e^{z}} dz}{\int_0^\infty z^\beta \Phi \left(-\frac{z}{2}\right) dz}\right)^{\frac{1}{\beta +1}} \cdot t= : a t .  \label{eq:Cm-asymp-noise-tm}
\end{align*}
In particular, $t_m / t\opt \sqrt{m} = a(1 + o_m(1))$.

We finally proceed to compute asymptotic behavior of $\varfunc(t\opt)$, the
variance~\eqref{eqn:generic-majority-covariance}. By
Lemma~\ref{lem:t-Z-integral-noise} the limit of its denominator as $t_m \to
\infty$ satisfies
\begin{align*}
  \lim_{m \to \infty}
  \underbrace{t_m^{2 \beta}
    \Ep\left[\frac{e^{t_m Z} }{ (1 + e^{t_m Z})^2}\right]^2}_{
    :=\mathsf{den}(C_m(t))}
  & = \lim_{t \to \infty} \left( t^\beta \Ep\left[\frac{e^{t Z} }{ (1 + e^{t Z})^2}\right]\right)^2 =
  \left(c_Z\int_{0}^\infty \frac{z^{\beta - 1} e^{z} }{ (1 + e^{z})^2} dz\right)^2.
\end{align*}
We decompose the numerator into the two parts
\begin{align*}
  \lefteqn{m^{\frac{\beta}{2}} \Ep \left[\frac{1}{\left(1 + e^{t_m|Z|}\right)^2} \rho_m(t |Z|) + \frac{1}{\left(1 + e^{-t_m|Z|}\right)^2} (1-\rho_m(t|Z|)) \right]} \nonumber \\
  & = \underbrace{m^{\frac{\beta}{2}} \Ep \left[\left(\frac{1}{\left(1 + e^{-t_m|Z|}\right)^2} - \frac{1}{\left(1 + e^{t_m|Z|}\right)^2} \right)(1-\rho_m(t|Z|)) \right]}_{\mathrm{(I)}} + \underbrace{m^{\frac{\beta}{2}} \Ep \left[\frac{1}{\left(1 + e^{t_m|Z|}\right)^2}\right]}_{\mathrm{(II)}}  .
\end{align*}
As we have already shown that $m^{-\frac{1}{2}} t_m \to at$, we know for any
$\epsilon > 0$ that for large enough $m$, $(1-\epsilon) a t \sqrt m
\leq t_m \leq (1 +\epsilon) at \sqrt{m}$. We can thus invoke
Lemma~\ref{lem:t-Z-integral-noise} to get
\begin{align*}
  \lim_{m \to \infty}  \mathrm{(II)} = \lim_{m \to \infty} m^{\frac \beta 2} \Ep \left[\frac{1}{\left(1 + e^{\sqrt m at|Z|}\right)^2}\right] &  = c_Z \int_0^\infty  \frac{z^{\beta - 1}}{\left(1 + e^{a t z}\right)^2} dz  = c_Z {t}^{-\beta} \int_0^\infty \frac{z^{\beta - 1}}{\left(1 + e^{a z}\right)^2} dz  . 
\end{align*}
With the same argument, we apply Lemma~\ref{lem:m-rho-integral-noise-general}
to establish the convergence
\begin{align*}
  \lim_{m \to \infty} \mathrm{(I)} 
  & = c_Z \int_0^\infty z^{\beta - 1}\left(\frac{1}{\left(1 + e^{-a tz}\right)^2} - \frac{1}{\left(1 + e^{a t z}\right)^2} \right) \Phi \left(-\frac{tz}{2}\right) dz \\
  & = c_Z \int_0^\infty z^{\beta - 1}\frac{e^{a tz} - 1}{e^{a t z} + 1} \Phi \left(-\frac{tz}{2}\right) dz
  = c_Z {t}^{-\beta} \int_0^\infty z^{\beta - 1}\frac{e^{a z} - 1}{e^{a z} + 1} \Phi \left(-\frac{z}{2}\right) dz    , 
\end{align*}
where we use the change of variables $tz \mapsto z$. Taking limits, we have
\begin{align*}
  \lim_{m\to \infty} m^{1 - \frac{1}{2} \beta} C_m(t)
  & = \lim_{m\to \infty}  \frac{m^{\frac \beta 2}
    \Ep[\frac{1}{(1 + e^{t_m|Z|})^2} \rho_m(t |Z|)
      + \frac{1}{(1 + e^{-t_m|Z|})^2}(1-\rho_m(t|Z|))]}{
    m^{\beta - 1} t_m^{2 - 2 \beta} \cdot t_m^{2 \beta} \Ep[\frac{e^{t_m Z} }{ (1 + e^{t_m Z})^2}]^2} \\
  & = \lim_{m \to \infty}
  \left(\frac{t_m}{\sqrt{m}}\right)^{2 \beta - 2}
  \cdot \frac{\lim_{m\to \infty}\mathrm{(I)} + \lim_{m\to \infty}\mathrm{(II)}}{\mathsf{den}(C_m(t))} \\
  & =
  (at)^{2 \beta - 2}
  \cdot \frac{c_Z {t}^{-\beta} \int_0^\infty z^{\beta - 1}\left(\frac{1}{\left(1 + e^{a z}\right)^2} + \frac{e^{a z} - 1}{e^{a z} + 1} \Phi \left(-\frac{z}{2}\right)\right) dz}{  (c_Z\int_{0}^\infty \frac{z^{\beta - 1} e^{z} }{ (1 + e^{z})^2} dz)^2},
\end{align*}
where we used that $t_m/\sqrt{m} \to at$ as above. 


\subsection{Proof of Proposition~\ref{prop:lowd-misspcfd-majority-vote-m}}  \label{proof:lowd-misspcfd-majority-vote-m}

\paragraph{Minimizer of the population loss.} 
By Lemma~\ref{lemma:generic-solutions}, we know the gap $h(t)=0$ has a
unique solution $t_m$, with $t_m u\opt$ minimizing the population loss
$L(\theta, \sigma)$.

\paragraph{Asymptotic variance.} 
Directly invoking Theorem~\ref{theorem:majority-vote-asymptotics} yields
asymptotic normality:
\begin{equation*}
  \sqrt{n}\left(\what{u}\mv_{n,m} - u\opt\right)
  \cd \normal\prn{0, \varfunc(t\opt) \prn{\projperp \Sigma \projperp}^\dagger},
\end{equation*}
where the covariance function~\eqref{eqn:generic-majority-covariance} has the
form
\begin{equation}
  \varfunc(t\opt) = \frac{1}{t_m^2} \frac{\E[\sigma(-t_m |Z|)^2
      \rho_m(t\opt Z) + \sigma(t_m |Z|)^2 (1 - \rho_m(t\opt Z))]}{
    \E[\sigma'(t_m Z)]^2} \label{eqn:mis-majority-covariance}
\end{equation}
and again $t_m$ is the implicitly defined zero of $h(t) = 0$.

\paragraph{Large $m$ behavior.}
We derive the large $m$ asymptotics of $t_m$ and $\varfunc$ under
Assumption~\ref{assumption:condition-misspcfd-link-majority-vote} and using
the shorthand $\ltwo{\theta\opt} = t$. The proof is essentially identical to
that of Proposition~\ref{prop:lowd-majority-vote-m-noise} in
Appendix~\ref{proof:lowd-majority-vote-m-noise}.
First, recalling
the probability~\eqref{eq:link-function-m-majority-vote-general},
$\rho_m(t) = \Prb(\majY = \sgn(\<X, \theta\opt\>) \mid \<X,\theta\opt\> = t)$,
we see that $1 - \rho_m(t z) \to 0$ for
any $z \neq 0$ as $m \to \infty$, and thus by dominated convergence,
$\Ep \brk{|Z|\prn{1 - \rho_m(t Z)}} \to 0$.
The analogue of the identity~\eqref{eqn:solving-for-t-m}
in the proof of Proposition~\ref{prop:lowd-majority-vote-m-noise}, that
$t_m$ is the zero of $h(t) = \E[\sigma(t|Z|)|Z|(1 - \rho_m(t\opt Z))]
- \E[\sigma(-t|Z|) |Z| \rho_m(t\opt Z)]$, implies
\begin{equation}
  \Ep[\sigma(t_m|Z|) |Z|(1 - \rho_m(t\opt Z))]
  = \Ep \brk{\sigma(-t_m |Z|) |Z| \rho_m(t\opt Z)}.
  \label{eqn:mis-spec-solving-for-t-m}
\end{equation}
As $\sigma$ is bounded and $\E[\sigma(-t|Z|) |Z| \rho_m(t\opt |Z|)] \to
\E[\sigma(-t|Z|) |Z|]$ for any $t$, the convergence of the left hand side of
equality~\eqref{eqn:mis-spec-solving-for-t-m} to 0 as $m \to \infty$ means
we must have $t_m \to \infty$.
Invoking
Lemma~\ref{lem:t-Z-integral-noise} yields
\begin{align*}
  \lim_{m \to \infty} t_m^{\beta + 1} \Ep \brk{|Z| \sigma(-t_m |Z|)}
  =
  \lim_{m \to \infty} t_m^{\beta} \Ep \brk{t_m|Z| \sigma(-t_m |Z|)}
  = c_Z \int_0^\infty z^\beta \sigma(-z) dz.
\end{align*}
Applying Lemma~\ref{lem:m-rho-integral-noise-general} gives
\begin{align*}
  m^{\frac{\beta+1}{2}}\Ep \brk{|Z|(1-\rho_m(t Z))}
  =  m^{\frac{\beta}{2}} \Ep \brk{\sqrt m|Z|(1-\rho_m(t  Z))}
  \to c_Z {t}^{-\beta-1}\int_0^\infty z^\beta \Phi \left(-2{\wb{\sigma}\opt}'(0)z\right) dz.
\end{align*}
Rewriting the identity~\eqref{eqn:mis-spec-solving-for-t-m}
using the symmetry of $\sigma$, so that $\sigma(t) + \sigma(-t) = 1$,
we have the equivalent statement that
$\E[|Z|(1 - \rho_m(t\opt Z))] = \E[\sigma(-t_m |Z|) |Z|]$,
or $\E[\sigma(-t_m |Z|)|Z|] / \E[|Z|(1 - \rho_m(t\opt Z))] = 1$.
Using this identity ratio, we find that
\begin{equation*}
  \frac{t_m}{\sqrt{m}}
  =
  \left(\frac{	t_m^{\beta + 1}\Ep \brk{|Z| \sigma(-t_m|Z|)}}{
    m^{\frac{\beta+1}{2}} \Ep \brk{|Z|(1-\rho_m(t  Z))} }\right)^\frac{1}{\beta +1}
  \to
  \left(\frac{\int_{0}^{\infty} z^\beta \sigma(-z) dz}{
    \int_0^\infty z^\beta \Phi(-2{\wb{\sigma}\opt}'(0)z) dz}\right)^{
    \frac{1}{\beta +1}} \cdot t\opt
  = : a t\opt.
\end{equation*}
This concludes the asymptotic characterization that $t_m = \sqrt{m} a t\opt
\cdot(1 + o_m(1))$.

Finally, we turn to the asymptotics for $\varfunc(t\opt)$
in~\eqref{eqn:mis-majority-covariance}.
By Lemma~\ref{lem:t-Z-integral-noise}
its denominator has limit
\begin{align*}
  \lim_{m \to \infty}
  \underbrace{t_m^{2 \beta} \Ep\left[\sigma'(t_m Z)\right]^2}_{:=
    \mathsf{den}{\varfunc(t\opt)}}
  = \lim_{t \to \infty}
  \left(t^\beta \Ep\left[\sigma'(t Z)\right]\right)^2
  = \left(c_Z\int_{0}^\infty z^{\beta - 1} \sigma'(z)dz\right)^2.  
\end{align*}
We decompose the (rescaled) numerator of the
variance~\eqref{eqn:mis-majority-covariance} into the two parts
\begin{align*}
  \underbrace{m^{\frac{\beta}{2}} \Ep \left[\left(
      \sigma(t_m|Z|)^2 - \sigma(-t_m |Z|)^2\right)
      (1-\rho_m(t|Z|)) \right]}_{\mathrm{(I)}}
  + \underbrace{m^{\frac{\beta}{2}} \Ep \left[\sigma(-t_m|Z|)^2\right]}_{
    \mathrm{(II)}}.
\end{align*}
Lemmas~\ref{lem:t-Z-integral-noise} and that $t_m = a \sqrt{m} t\opt(1 +
o_m(1)) \to \infty$, coupled with the dominated
convergence theorem, establishes the convergence
\begin{align*}
  \lim_{m \to \infty}  \mathrm{(II)} = \lim_{m \to \infty} m^{\frac \beta 2}
  \Ep \left[\sigma(-\sqrt{m} a t\opt |Z|)^2\right]
  = c_Z {t}^{-\beta} \int_0^\infty \frac{z^{\beta - 1}}{\left(1 + e^{a z}\right)^2}
  dz.
\end{align*}
Similarly,
Lemma~\ref{lem:m-rho-integral-noise-general} and that
$t_m = a \sqrt{m} t\opt(1 + o_m(1))$
gives that
\begin{align*}
  \lim_{m \to \infty} \mathrm{(I)} 
  & = \lim_{m \to \infty} m^{\frac \beta 2} \Ep \left[
    \left(\sigma(at\opt \sqrt{m}|Z|)^2 - \sigma(-a t\opt\sqrt{m}|Z|)^2
    \right) (1 - \rho_m(t Z)) \right] \\
  & = c_Z \int_0^\infty z^{\beta - 1}
  \left(\sigma(a t\opt z)^2 - \sigma(-at\opt z)^2\right)
  \Phi \left(-2{\wb{\sigma}\opt}'(0) t z\right) dz \\
  & = c_Z {t\opt}^{-\beta} \int_0^\infty z^{\beta - 1}
  \left(\sigma(az)^2 - \sigma(-az)^2\right)
  \Phi \left(-2{\wb{\sigma}\opt}'(0) z\right) dz,
\end{align*}
where in the last line we use change of variables $tz \mapsto z$. Hence we have
\begin{align*}
  \lim_{m\to \infty} m^{1 - \frac{1}{2} \beta} \varfunc(t\opt)
  & = \lim_{m \to \infty} \frac{1}{m^{\beta - 1} t_m^{2 - 2 \beta}}
  \cdot \frac{\lim_{m\to \infty}\mathrm{(I)} + \lim_{m\to \infty}\mathrm{(II)}}{
    \mathsf{den}(C_m(t\opt))} \\
  & =  \lim_{m\to \infty} \prn{\frac{t_m}{\sqrt m}}^{2\beta - 2}
  \cdot \frac{c_Z {t\opt}^{-\beta} \int_0^\infty z^{\beta - 1}
    \prns{\sigma(az)^2
      + (\sigma(az)^2 - \sigma(-az)^2) \Phi(-2{\wb{\sigma}\opt}'(0) z)}
    dz}{
    (c_Z\int_{0}^\infty \frac{z^{\beta - 1} e^{z} }{ (1 + e^{z})^2} dz)^2}.
\end{align*}
Finally, as $t_m / \sqrt{m} = a t\opt (1 + o_m(1))$ we obtain that
$m^{1 - \beta/2} \varfunc(t\opt) = {t\opt}^{\beta - 2} b$ for
some constant $b$ depending only on $\beta, c_Z, {\wb{\sigma}\opt}'(0)$,
and $\sigma$.

\section{Proofs of fundamental limits of estimation with aggregate labels} 

\subsection{Proof of Proposition~\ref{prop:impossibility}} \label{proof:impossibility}
We assume without loss of generality $m$ is odd. When $m$ is even the proof
is identical, except that we randomize $\majY$ when we have equal votes for
both classes. As the marginal distribution of $X$ is $\normal(0, I_d)$ for
both $\P_{(X,\majY)}^{\sigma\opt, \theta\opt, m}$ and
$\P_{(X,\majY)}^{\wb{\sigma}, \wb{\theta}, \wb{m}}$, we only have to show
the existence of a link $\wb{\sigma} \in \fsymlink$ such that the
conditional distribution on any $X=x$ is the same, i.e.,
\begin{align*}
  \Prb_{\sigma\opt, \theta\opt, m} (\majY = 1 \mid X=x) =
  \Prb_{\wb{\sigma}, \wb{\theta}, \wb{m}}(\majY = 1 \mid X=x) .
\end{align*}
For $m \in \mathbb{N}$, define the probability that a $\bindist(m, t)$ is at
least $\lceil m / 2 \rceil$ through the one-to-one transformation
\begin{align*}
  P_m(t) := \sum_{i=\lceil m/2 \rceil}^m \binom{m}{i} t^m(1-t)^{m-i}   ,
\end{align*}
 For any $m$, $P_m(t)$ monotonically increases in $t$, and $P_m(0)= 0, P_m(\frac{1}{2}) = \frac{1}{2}, P_m(1) = 1$, and by symmetry $P_m(t) + P_m(1-t) = 1$. By the definition of majority vote
\begin{align*}
	\Prb_{\sigma\opt, \theta\opt, m} (\majY = 1 \mid X=x) & = \sum_{i=\lceil m/2 \rceil}^m \binom{m}{i} \sigma\opt(\< \theta\opt, x \>)^m(1-\sigma\opt(\< \theta\opt, x \>))^{m-i} = P_m \circ \sigma\opt(\< \theta\opt, x \>).
\end{align*}
Therefore, the link
\begin{align*}
	\wb{\sigma}(t) := P_{\wb{m}}^{-1} \circ P_{m} \circ \sigma\opt \left(\frac{\ltwo{\theta}}{\ltwo{\wb{\theta}}} t \right)
\end{align*}
satisfies $\wb{\sigma}(0) = \frac{1}{2}$ and $\wb{\sigma}(t) - \frac{1}{2} > 0$ for all $t > 0$, since $P_m$ maps $(\frac 1 2, 1]$ to $(\frac 1 2 , 1]$. We also have
\begin{align*}
	\wb{\sigma}(t) + \wb{\sigma}(-t) & = P_{\wb{m}}^{-1} \circ P_{m} \circ \sigma\opt \left(\frac{\ltwo{\theta}}{\ltwo{\wb{\theta}}} t \right) + P_{\wb{m}}^{-1} \circ P_{m} \circ \sigma\opt \left(-\frac{\ltwo{\theta}}{\ltwo{\wb{\theta}}} t \right) \nonumber \\
	& \stackrel{\mathrm{(i)}}{=} P_{\wb{m}}^{-1} \circ P_{m} \circ \sigma\opt \left(\frac{\ltwo{\theta}}{\ltwo{\wb{\theta}}} t \right) + P_{\wb{m}}^{-1} \circ P_{m} \circ \left(1 - \sigma\opt \left(\frac{\ltwo{\theta}}{\ltwo{\wb{\theta}}} t \right) \right) \nonumber \\
	& \stackrel{\mathrm{(ii)}}{=} P_{\wb{m}}^{-1} \circ P_{m} \circ \sigma\opt \left(\frac{\ltwo{\theta}}{\ltwo{\wb{\theta}}} t \right) + P_{\wb{m}}^{-1} \circ \left(1 -  P_m \circ \sigma\opt \left(\frac{\ltwo{\theta}}{\ltwo{\wb{\theta}}} t \right)\right) \nonumber \\
	& = 1   ,
\end{align*}
where (i) and (ii) follow from symmetry of $\sigma\opt$ and $P_m$, respectively. Thus $\wb{\sigma}$ belongs to $\fsymlink$ and is a valid link function. This $\bar{\sigma}$ yields the desired equality:
\begin{align*}
	\Prb_{\wb{\sigma}, \wb{\theta}, \wb{m}}(\majY = 1 \mid X=x) &   = P_{\wb{m}} \circ P_{\wb{m}}^{-1} \circ P_{m} \circ \sigma\opt \left(\frac{\ltwo{\theta}}{\ltwo{\wb{\theta}}} \cdot \< \wb{\theta}, x \> \right)  = \Prb_{\sigma\opt, \theta\opt, m} (\majY = 1 \mid X=x).
\end{align*}

\subsection{Proof of Theorem~\ref{thm:lower-bound-fisher-information}} \label{proof:lower-bound-fisher-information}

Throughout this proof, we use $a_m \sim b_m$ to mean that
$a_m/b_m \to 1$ as $m \to \infty$, or equivalently that
$a_m = b_m \cdot(1 + o_m(1))$.
By Lemma~\ref{lem:fisher-information-decomposition},
we have 
\begin{align*} 
	\fisher_m(\theta\opt) = A_m (t\opt) \cdot u\opt {u\opt}^\top +   B_m (t\opt) \cdot \proj_{u\opt}^\perp \Sigma \proj_{u\opt}^\perp 
\end{align*}
with
\begin{subequations}
\begin{align}
	A_m(t) & := \E \brk{\frac{\rho_m'(t|Z|)^2 Z^2}{\rho_m(t|Z|) ( 1 - \rho_m(t|Z|))}}, \label{eq:fisher-information-Am} \\
	B_m(t)& := \E \brk{\frac{\rho_m'(t|Z|)^2 }{\rho_m(t|Z|) ( 1 - \rho_m(t|Z|))}}. \label{eq:fisher-information-Bm}
\end{align}
\end{subequations}
It is crucial to understand the behavior of
$\rho_m$ and $\rho_m'$, which have the exact forms
\begin{align*}
  \rho_m(t) & 
  = \sum_{k>m/2}\binom{m}{k} \frac{e^{kt}}{(1+e^t)^m}, \\
  \rho_m'(t) & 
  = \sum_{k > m/2} \binom{m}{k} \frac{ke^{kt} (1+e^t) - me^{(k+1)t}}{(1+e^t)^{m+1}}.
\end{align*}
Throughout this proof, we assume without loss of generality that $m$ is
odd. When $m$ is even, we replace $\binom{m}{\frac{m+1}{2}}$ with
$\binom{m}{\frac{m}{2}+1}$, but otherwise all arguments are identical.
\begin{lemma} \label{lem:rho-m-prime-asymptotics}
	The function $\rho_m'$ is even, and for any $t > 0$,
	\begin{align*}
		\rho_m'(t) &= \binom{m}{\frac{m+1}{2}} \cdot \frac{m+1}{2} \frac{e^{\frac{m+1}{2}t}}{(1+e^t)^{m+1}}.
	\end{align*}
\end{lemma}
\begin{proof}
	By direct calculations
	\begin{align*}
		\rho_m'(t) &  = \sum_{k > m/2} \binom{m}{k} \frac{ke^{kt} (1+e^t) - me^{(k+1)t}}{(1+e^t)^{m+1}} \\
		&  \stackrel{(i)}{=} \sum_{k > m/2} \left\{  \binom{m}{k} \frac{ke^{kt}(1+e^t)}{(1+e^t)^{m+1}} - \binom{m}{k} \frac{ke^{(k+1)t}}{(1+e^t)^{m+1}} - \binom{m}{k+1} \frac{(k+1)e^{(k+1)t}}{(1+e^t)^{m+1}} \right\} \nonumber \\
		& = \frac{1}{(1+e^t)^{m+1}}\sum_{k > m/2} \left\{ \binom{m}{k} ke^{kt} - \binom{m}{k+1} (k+1)e^{(k+1)t}\right\}, 
	\end{align*}
	where (i) uses the combinatorial identity
	\begin{align*}
		m \binom{m}{k} = k \binom{m}{k} + (k+1)\binom{m}{k+1}.
	\end{align*}
	This yields a telescoping sum and establishes the result.
\end{proof}
Next,  we define the function
\begin{align*}
	\eta_m(t) = \frac{\rho_m'(t)^2}{\rho_m(t)(1-\rho_m(t))}.
\end{align*}
We can then write for the Fisher information $\fisher_m(\theta\opt) = A_m (t\opt) \cdot u\opt {u\opt}^\top +   B_m (t\opt) \cdot \proj_{u\opt}^\perp \Sigma \proj_{u\opt}^\perp $ via
\begin{align*}
	A_m(t) = \E \brk{\eta_m(tZ)Z^2}, \qquad B_m(t) = \E \brk{\eta_m(tZ)},
\end{align*}
as $\eta_m$ is an even function. The next lemma provides the key asymptotic
guarantees for $\eta_m(t)$ in the above displays. See
Appendix~\ref{proof:rho-m-fisher-asymptotics} for a proof.

\begin{lemma} \label{lem:rho-m-fisher-asymptotics}
  There exists a function $c_m(t)$ satisfying the following:
  \begin{enumerate}[leftmargin=*,label=(\roman*)]
  \item \label{item:eta-via-c}
    For any $t > 0$, $c_m(t) \geq 1 - e^{-t}$ and
    \begin{align*}
      \eta_m(t) = \frac{(m+1)^2}{4} \binom{m}{\frac{m+1}{2}} \cdot \frac{e^{\frac{m+3}{2}t}}{(1+e^t)^{m+2}} \cdot c_m(t),
    \end{align*}
  \item \label{item:c-m-by-delta} For any $\delta \in (0,1)$, there exist
    $C=C(\delta) > 0$ and $M=M(\delta)$ such that for all $m \geq M$
    and all $t > 0$,
    \begin{align*}
      c_m(t) \leq
      \frac{C \cdot(1 - e^{-t})}{1 - e^{-(\lfloor \delta \sqrt{m} \rfloor + 1) t}}.
    \end{align*}
  \item \label{item:normal-ratio-limit}
    For any fixed $w > 0$,
    \begin{align*}
      \lim_{m \to \infty} \sqrt{m} c_m \prn{\frac{w}{\sqrt{m}}}  = e^{-\frac{w^2}{8}} \sqrt{\frac{2}{\pi}} \cdot \frac{1}{\Phi(-w/2) \Phi(w/2)}.
    \end{align*}
  \end{enumerate}
\end{lemma}

We use the asymptotics in Lemma~\ref{lem:rho-m-fisher-asymptotics} to
compute $A_m(t)$ and $B_m(t)$, respectively.

\paragraph{Part I: Asymptotics for $B_m(t)$.}
We begin by showing the claimed asymptotic formula for $B_m(t)$ in
the theorem's statement.
By Stirling's formula, the multiplicative factor in $\eta_m$ satisfies
\begin{align*}
	\frac{(m+1)^2}{4} \binom{m}{\frac{m+1}{2}} \sim \frac{m^2}{4} \cdot 2^m \sqrt{\frac{2}{\pi m}},
\end{align*}
and thus
\begin{align*}
	B_m(t) &\sim \frac{m^2}{2 \sqrt{2 \pi m}} \E \brk{2^m \frac{e^{\frac{m+3}{2}tZ}}{(1+e^{tZ})^{m+2}} \cdot c_m(tZ)} = \frac{m^2}{8 \sqrt{2 \pi m}} \E \brk{2^{m+2} \frac{e^{\frac{m+3}{2}tZ}}{(1+e^{tZ})^{m+2}} \cdot c_m(tZ)} \nonumber \\
	& = \frac{m^2}{8 \sqrt{2 \pi m}} \E \brk{\prn{\frac{2 e^{\frac{tZ}{2}}}{1 + e^{tZ}}}^{m+2}\cdot e^{\frac{tZ}{2}} c_m(tZ)} \sim \frac{\sqrt{m}}{8\sqrt{2\pi}} \cdot \underbrace{(m+2) \E \brk{\prn{\frac{2 e^{\frac{tZ}{2}}}{1 + e^{tZ}}}^{m+2}\cdot e^{\frac{tZ}{2}} c_m(tZ)}}_{=:J_{m+2}(t)}.
\end{align*}
Our goal now is to show that
\begin{equation}
  \label{eqn:J-m-limit-term}
  \lim_{m \to \infty} m^{\frac{\beta - 1}{2}}
  J_m(t)
  = \frac{c_Z}{t^\beta} \sqrt{\frac{8}{\pi}}
  \cdot \int_0^\infty \frac{e^{-z^2} z^{\beta - 1}}{\Phi(-z) \Phi(z)} dz,
\end{equation}
which then immediately implies the asymptotic
$B_m(t) \sim b m (t \sqrt{m})^{-\beta}$
for $b = \frac{c_Z}{4\pi} \int_0^\infty \frac{e^{-z^2} z^{\beta - 1}}{
  \Phi(-z) \Phi(z)} dz$, as claimed in the theorem.

The proof of the limit~\eqref{eqn:J-m-limit-term} involves
an argument via dominated convergence.
By the change of variables $w := \sqrt{m} u$,
\begin{align*}
	 J_m(t) & =  \int_0^\infty m \prn{\frac{2e^{\frac{tu}{2}}}{1 + e^{tu}}}^{m} e^{\frac{tu}{2}} c_{m-2}(tu) p(u) du \nonumber \\
	 & =  \int_0^\infty \sqrt{m} \prn{\frac{2e^{\frac{tw}{2\sqrt{m}}}}{1 + e^{\frac{tw}{\sqrt{m}}}}}^{m} e^{\frac{tw}{\sqrt{m}}} c_{m-2}\prn{\frac{tw}{\sqrt{m}}} p \prn{\frac{w}{\sqrt{m}}} dw \nonumber  \\
	 & = \prn{\frac{1}{\sqrt{m}}}^{\beta-1} \int_0^\infty \prn{\frac{2e^{\frac{tw}{2\sqrt{m}}}}{1 + e^{\frac{tw}{\sqrt{m}}}}}^{m} e^{\frac{tw}{\sqrt{m}}} w^{\beta-1} \cdot \sqrt{m}  \prn{\frac{w}{\sqrt{m}}}^{1-\beta} c_{m-2}\prn{\frac{tw}{\sqrt{m}}} p \prn{\frac{w}{\sqrt{m}}} dw.
\end{align*}
We now demonstrate dominating functions for the various terms in the above
integrand. To bound the lsat several terms,
we use Assumption~\ref{assumption:z-density} that $\kappa := \sup_{z
  \in (0, \infty)} z^{1-\beta} p(z) < \infty$ and
Lemma~\ref{lem:rho-m-fisher-asymptotics}, which combine
to give the upper bound
\begin{align*}
  \sqrt{m}  \prn{\frac{w}{\sqrt{m}}}^{1-\beta} c_{m-2}\prn{\frac{tw}{\sqrt{m}}} p \prn{\frac{w}{\sqrt{m}}}
  & \leq  \kappa  \sqrt{m} c_{m-2} \prn{\frac{tw}{\sqrt{m}}} \\
  & \leq  \frac{\kappa  c\sqrt{m}\prn{1 - e^{-\frac{tw}{\sqrt{m}}}}}{1 - e^{-(\lfloor \delta \sqrt{m-2} \rfloor + 1) \frac{tw}{\sqrt{m}}}}
  \leq c' \sup_{w \in (0, \infty)} \frac{ tw}{1 - e^{-\delta tw}} < \infty
\end{align*}
for some $c' = c(\delta, \kappa)$ and all sufficiently large $m$.  The next
lemma controls the first terms in the integral above. (We defer its proof to
Appendix~\ref{proof:Jm-integral-term-upper-bound}.)
\begin{lemma} \label{lem:Jm-integral-term-upper-bound}
  There exists a universal constant $0 < C$ such that for all $m$,
  \begin{align*}
    \prn{\frac{2e^{\frac{tw}{2\sqrt{m}}}}{1 + e^{\frac{tw}{\sqrt{m}}}}}^{m} e^{\frac{tw}{\sqrt{m}}} w^\beta   & \leq \frac{1}{C}
    w^\beta \exp \brc{-C \min \{t^2w^2, \sqrt{tw}\}}.
  \end{align*}
\end{lemma}

Certainly the right hand side term in
Lemma~\ref{lem:Jm-integral-term-upper-bound} is integrable on $(0, \infty)$.
Rewriting $J_m(t)$, we can therefore invoke dominated convergence to
exchange the limit in $m$ and integral to obtain
\begin{align*}
	\lim_{m \to \infty} m^{\frac{\beta-1}{2}} J_m(t) & = \int_0^\infty \lim_{m \to \infty} \prn{\frac{2e^{\frac{tw}{2\sqrt{m}}}}{1 + e^{\frac{tw}{\sqrt{m}}}}}^{m} e^{\frac{tw}{\sqrt{m}}} w^{\beta-1}  \cdot \sqrt{m}  \prn{\frac{w}{\sqrt{m}}}^{1-\beta} c_{m-2}\prn{\frac{tw}{\sqrt{m}}} p \prn{\frac{w}{\sqrt{m}}} dw \nonumber \\
	& \stackrel{(i)}{=}  c_Z \int_0^\infty \lim_{m \to \infty} w^{\beta-1} \exp \brc{-\frac{t^2w^2}{2} \cdot \prn{\frac{\sqrt{m}}{tw} \cdot \prn{1-e^{-\frac{tw}{2\sqrt{m}}}}}^2 } \cdot \lim_{m \to \infty} \sqrt{m} c_{m-2}\prn{\frac{tw}{\sqrt{m}}} dw \nonumber \\
	& \stackrel{(ii)}{=} c_Z \int_0^\infty e^{-\frac{1}{8}t^2w^2} \cdot  e^{-\frac{t^2 w^2}{8}} \sqrt{\frac{2}{\pi}} \cdot \frac{w^{\beta-1}}{\Phi(-tw/2) \Phi(tw/2)} dw \nonumber \\
	& = \frac{c_Z}{t^\beta} \sqrt{\frac{8}{\pi}} \cdot \int_0^\infty \frac{e^{-z^2} z^{\beta-1}}{\Phi(-z) \Phi(z)} dz,
\end{align*}
where in equality $(i)$ we invoke Assumption~\ref{assumption:z-density} and
in $(ii)$ we use Lemma~\ref{lem:rho-m-fisher-asymptotics}. This
gives the desired limit~\eqref{eqn:J-m-limit-term}.

\paragraph{Part II: Asymptotics for $A_m(t)$.}
The calculations are similar to those we have done to approximate
$B_m(t)$. In this case
\begin{align*}
	A_m(t) \sim \frac{1}{8\sqrt{2\pi m}} \cdot \underbrace{(m+2)^2 \E \brk{\prn{\frac{2 e^{\frac{tZ}{2}}}{1 + e^{tZ}}}^{m+2}\cdot e^{\frac{tZ}{2}} c_m(tZ) Z^2}}_{=:I_{m+2}(t)},
\end{align*}
and again invoking dominated convergence,
\begin{align*}
  \lim_{m \to \infty}  m^{\frac{\beta-1}{2}} I_m(t) & = \int_0^\infty \lim_{m \to \infty} \prn{\frac{2e^{\frac{tw}{2\sqrt{m}}}}{1 + e^{\frac{tw}{\sqrt{m}}}}}^{m} e^{\frac{tw}{\sqrt{m}}} w^{\beta-1}  \cdot \sqrt{m}  \prn{\frac{w}{\sqrt{m}}}^{1-\beta} c_{m-2}\prn{\frac{tw}{\sqrt{m}}} p \prn{\frac{w}{\sqrt{m}}} w^2 dw \nonumber \\
  & = c_Z \int_0^\infty e^{-\frac{1}{8}t^2w^2} \cdot e^{-\frac{t^2 w^2}{8}} \sqrt{\frac{2}{\pi}} \cdot \frac{w^{\beta+1}}{\Phi(-tw/2) \Phi(tw/2)} dw \nonumber  \\
  &  = \frac{8 c_Z}{t^{\beta+2}} \sqrt{\frac{2}{\pi}} \cdot \int_0^\infty \frac{e^{-z^2} z^{\beta+1}}{\Phi(-z) \Phi(z)} dz.
\end{align*}
Hence
\begin{align*}
	A_m(t) = \frac{a}{t^2} \prn{\frac{1}{t\sqrt{m}}}^{\beta} (1 + o_m(1)), \qquad a = \frac{c_Z}{ \pi} \int_0^\infty \frac{e^{-z^2} z^{\beta+1}}{\Phi(-z) \Phi(z)} dz.
\end{align*}

\subsection{Proof of Lemma~\ref{lem:rho-m-fisher-asymptotics}} \label{proof:rho-m-fisher-asymptotics}

We begin by proving part~\ref{item:eta-via-c}. Define
\begin{equation*}
  s_m(t) = \sum_{k <
    m/2} \binom{m}{k} e^{kt} \le \half (1+e^t)^m,
\end{equation*}
where the inequality is valid for $t \ge 0$.
Invoking Lemma~\ref{lem:rho-m-prime-asymptotics} yields
\begin{align*}
	\eta_m(t) & = \frac{\rho_m'(t)^2}{\rho_m(t)(1-\rho_m(t))} = \frac{\frac{(m+1)^2}{4} \binom{m}{\frac{m+1}{2}}^2 \cdot \frac{e^{(m+1)t}}{(1+e^t)^{2m+2}}}{\prn{\sum_{k > m/2} \binom{m}{k} e^{kt}} \prn{\sum_{k < m/2} \binom{m}{k} e^{kt}} \cdot \frac{1}{(1+e^t)^{2m}}} \nonumber \\
	& = \frac{(m+1)^2}{4} \binom{m}{\frac{m+1}{2}}^2 \cdot \frac{e^{(m+1)t}}{(1+e^t)^{m+2}} \cdot \frac{1}{(1 - s_m(t)/(1+e^t)^m) s_m(t)} \nonumber \\
	& = \frac{(m+1)^2}{4} \binom{m}{\frac{m+1}{2}} \cdot \frac{e^{\frac{m+3}{2}t}}{(1+e^t)^{m+2}} \cdot \underbrace{\frac{ \binom{m}{\frac{m+1}{2}} e^{\frac{m-1}{2}t}}{(1 - s_m(t)/(1+e^t)^m) s_m(t)}}_{:=c_m(t)},
\end{align*}
which gives the equality in part~\ref{item:eta-via-c}.
It then remains to show $c_m(t) \geq 1 - e^{-t}$. For this,
we upper bound
\begin{align*}
  e^{-\frac{m-1}{2}t} s_m(t) \leq \sum_{k < m/2}\binom{m}{\frac{m+1}{2}} e^{-\prn{\frac{m-1}{2}-k}t} \leq \binom{m}{\frac{m+1}{2}} \frac{1}{1-e^{-t}},
\end{align*}
and noting the inequality
$\half \le 1 - \frac{s_m(t)}{(1+e^t)^m} \leq 1$, it then follows
that
\begin{align*}
  c_m(t) \geq \frac{\binom{m}{\frac{m+1}{2}} e^{\frac{m-1}{2}t}}{s_m(t)} \geq 1 -e^{-t}.
\end{align*}

Proceeding to the proof of part~\ref{item:c-m-by-delta},
for any $\delta \in (0, 1)$,
\begin{align*}
  \binom{m}{\frac{m+1}{2}}^{-1} e^{-\frac{m-1}{2}t} s_m(t) & = \sum_{l=0}^\infty \binom{m}{\frac{m+1}{2}}^{-1} \binom{m}{\frac{m+1}{2} - l} e^{-lt} \geq  \underbrace{\binom{m}{\frac{m+1}{2}}^{-1} \binom{m}{\frac{m+1}{2} - \lfloor \delta \sqrt{m} \rfloor}}_{:=q_m(\delta)} \sum_{l=0}^{\lfloor \delta \sqrt{m} \rfloor} e^{-lt} \nonumber 	\\
  & =  q_m(\delta) \cdot \frac{1 - e^{-(\lfloor \delta \sqrt{m} \rfloor + 1) t}}{1 - e^{-t}},
\end{align*}
where the inequality uses that $\binom{m}{\frac{m+1}{2} - l}$ is decreasing
in $l$. We next establish that $q_m(\delta) \sim e^{-2\delta^2}$. Indeed,
applying Stirling's formula yields
\begin{align*}
 	q_m(\delta) &= \frac{\prn{\frac{m}{2}}^{\frac{m}{2}} \cdot \prn{\frac{m}{2}}^{\frac{m}{2}}}{m^m} \cdot  \frac{m^m}{\prn{\frac{m}{2}-\delta \sqrt{m}}^{\frac{m}{2}-\delta \sqrt{m}} \cdot \prn{\frac{m}{2}+\delta \sqrt{m}}^{\frac{m}{2}+\delta \sqrt{m}}} \cdot  (1+o_m(1)) \nonumber \\
 	& = \prn{1-\frac{2\delta}{\sqrt{m}}}^{-\frac{m}{2}+\delta \sqrt{m}} \cdot \prn{1+\frac{2\delta}{\sqrt{m}}}^{-\frac{m}{2}-\delta \sqrt{m}} \cdot (1+o_m(1)) \nonumber \\
 	& = \prn{1-\frac{4\delta^2}{m}}^{-\frac{m}{2}} \cdot \prn{1-\frac{2\delta}{\sqrt{m}}}^{\delta \sqrt{m}} \cdot \prn{1+\frac{2\delta}{\sqrt{m}}}^{-\delta \sqrt{m}} \cdot (1+o_m(1)) \nonumber \\
 	& = e^{2\delta^2} \cdot e^{-2\delta^2} \cdot e^{-2\delta^2} \cdot (1+o_m(1)) = e^{-2 \delta^2} \cdot (1+o_m(1)).
\end{align*}
Substituting the above displays into the definition $c_m(t) =
\binom{m}{\frac{m + 1}{2}} e^{\frac{m - 1}{2} t} s_m(t)^{-1} (1 -
\frac{s_m(t)}{(1 + e^t)^m})^{-1}$, we can then establish the upper bound
using $1 - s_m(t)/(1+e^t)^m \ge \half$, so
\begin{align*}
c_m(t) \leq \frac{2\binom{m}{\frac{m+1}{2}} e^{\frac{m-1}{2}t}}{s_m(t)} \leq \frac{2(1-e^{-t})}{q_m(\delta) \cdot (1 - e^{-(\lfloor \delta \sqrt{m} \rfloor + 1) t})}.
\end{align*}
Since $q_m(\delta) \sim e^{-2\delta^2}$, this inequality
establishes part~\ref{item:c-m-by-delta}.

For part~\ref{item:normal-ratio-limit}, we compute $\lim_{m \to \infty}
\sqrt{m} c_m(w/\sqrt{m})$. For any fixed $\delta > 0$, we explicitly expand
$s_m(w/\sqrt{m})$ to obtain
\begin{align*}
	\binom{m}{\frac{m+1}{2}}^{-1} e^{-\frac{m-1}{2} \cdot \frac{w}{\sqrt{m}}} s_m(w/\sqrt{m}) & = \sum_{l=0}^\infty \binom{m}{\frac{m+1}{2}}^{-1} \binom{m}{\frac{m+1}{2} - l} e^{-lw/\sqrt{m}} \nonumber \\
	& = \sqrt{m}  \sum_{k=0}^{\infty} \underbrace{\frac{1}{\sqrt{m} }\sum_{\lfloor k \delta \sqrt{m} \rfloor \leq l < \lfloor (k+1) \delta \sqrt{m}  \rfloor} \binom{m}{\frac{m+1}{2}}^{-1} \binom{m}{\frac{m+1}{2} - l} e^{-lw/\sqrt{m}}}_{:=a_{m,k}(\delta)}.
\end{align*}
Here, we partition the set of all nonnegative integers $\{l \in \N \mid 0
\leq l < \infty\}$ into sets $L_k = \{\lfloor k \delta \sqrt{m} \rfloor
\leq l < \lfloor (k+1) \delta \sqrt{m} \rfloor \}$ for $k = 0, 1, \ldots$.
Using the same asymptotics derived
above for $q_m(\delta)$ and that $e^{-(k+1) \delta w} \leq e^{-lw/\sqrt{m}}
\leq e^{-k \delta w} $ for all $l \in L_k$, it holds for some
remainder term $r_{m,k}$ satisfying
$|r_{m,k}(\delta)| \leq 1 - e^{-\delta w}$ that
\begin{align*}
	a_{m,k}(\delta) & = \delta \cdot \frac{1}{\delta \sqrt{m}} \sum_{\lfloor k \delta \sqrt{m} \rfloor \leq l < \lfloor (k+1) \delta \sqrt{m}  \rfloor} e^{-2k^2 \delta^2}  \cdot (1 + o_m(1)) \cdot  e^{-lw/\sqrt{m}} \nonumber \\
	& = \delta \cdot \frac{1 + r_{m,k}(\delta)}{\delta \sqrt{m}} \sum_{\lfloor k \delta \sqrt{m} \rfloor \leq l < \lfloor (k+1) \delta \sqrt{m}  \rfloor} e^{-2k^2 \delta^2}  \cdot (1 + o_m(1)) \cdot  e^{-k \delta w} \nonumber \\
	& = \prn{1 + r_{m,k}(\delta)} \cdot \delta e^{-2k^2 \delta^2-k\delta w} \cdot (1 + o_m(1)).
\end{align*}

We will invoke dominated convergence for series to allow
us to exchange summation and limits in $m$.
We observe the following termwise domination:
\begin{align*}
  a_{m,k}(\delta) & \leq \frac{1}{\sqrt{m} } \cdot \prn{\lfloor (k+1) \delta \sqrt{m}  \rfloor - \lfloor k \delta \sqrt{m} \rfloor } e^{-\lfloor k \delta \sqrt{m} \rfloor w / \sqrt{m}} \\
  & \leq \frac{\delta \sqrt{m} + 1}{\sqrt{m}} e^{-k\delta w + w/\sqrt{m}} \nonumber
  \leq (\delta +1) e^{-k\delta w + w},
\end{align*}
which is summable in $k$.  We can then invoke dominated convergence theorem
for infinite series to obtain that for some $|r(\delta)| \leq 1 - e^{-\delta
  w}$,
\begin{align*}
  \lim_{m \to \infty} \sum_{k=0}^\infty a_{m,k}(\delta) = (1 + r(\delta)) \cdot \sum_{k=0}^\infty \delta e^{-2k^2 \delta^2 - k\delta w},
\end{align*}
implying that for any $\delta > 0$,
\begin{align*}
	\binom{m}{\frac{m+1}{2}}^{-1} e^{-\frac{m-1}{2} \cdot \frac{w}{\sqrt{m}}} s_m(w/\sqrt{m}) \sim \sqrt{m} \cdot (1 + r(\delta)) \cdot \sum_{k=0}^\infty \delta e^{-2k^2 \delta^2 - k\delta w}.
\end{align*}
Since the left hand side does not depend on $\delta$, we can then send $\delta$ to zero on the right hand side and use the definition of Riemann integral to obtain
\begin{align*}	
 	\binom{m}{\frac{m+1}{2}}^{-1} e^{-\frac{m-1}{2} \cdot \frac{w}{\sqrt{m}}} s_m(w/\sqrt{m}) 
	& \sim \sqrt{m} \lim_{\delta \to 0} \sum_{k=0}^\infty \delta e^{-2 k^2 \delta^2 - k\delta w } \nonumber \\
	& = \sqrt{m} \int_0^\infty e^{-2x^2 -wx} dx = \sqrt{\frac{ \pi m}{2}} e^{\frac{w^2}{8}} \Phi(-w/2).
\end{align*}
%
%
Now we compute the limit of the remaining term in $c_m(t)$,
\begin{align*}
	s_m(w/\sqrt{m})/(1+e^{w/\sqrt{m}})^m & \sim \prn{\frac{1}{1 + e^{w/\sqrt{m}}}}^m  \cdot \binom{m}{\frac{m+1}{2}} e^{-w \sqrt{m}/2} \cdot \sqrt{\frac{ \pi m}{2}} e^{\frac{w^2}{8}} \Phi(-w/2) \nonumber \\
	& \sim \prn{\frac{2}{1 + e^{w/\sqrt{m}}}}^m \cdot  e^{-w\sqrt{m}/2} e^{\frac{w^2}{8}} \Phi(-w/2) \nonumber \\
	& =  \prn{\frac{2 e^{-\frac{w}{2\sqrt{m}}}}{1 + e^{\frac{w}{\sqrt{m}}}}}^m \cdot e^{\frac{w^2}{8}} \Phi(-w/2) = \prn{1 - \frac{\prn{1 - e^{-\frac{w}{2\sqrt{m}}}}^2}{1 + e^{\frac{w}{\sqrt{m}}}}}^m \cdot e^{\frac{w^2}{8}} \Phi(-w/2) \nonumber \\
	& \sim \prn{1 - \frac{1}{8} \frac{w^2}{m}}^m \cdot e^{\frac{w^2}{8}} \Phi(-w/2) \sim  \Phi(-w/2).
\end{align*}
Taking the above displays collectively into
\begin{align*}
	c_m(w/\sqrt{m}) = \frac{\binom{m}{\frac{m+1}{2}} e^{\frac{m-1}{2}\cdot \frac{w}{\sqrt{m}}}}{s_m(w/\sqrt{m})} \cdot \frac{1}{1 - s_m(w/\sqrt{m})/(1+e^{w/\sqrt{m}})^m}
\end{align*}
we establish part~\ref{item:normal-ratio-limit} of the lemma, that is,
\begin{align*}
	\lim_{m \to \infty} \sqrt{m} c_m \prn{\frac{w}{\sqrt{m}}} = \lim_{m \to \infty} \frac{\sqrt{m}}{\sqrt{\frac{ \pi m}{2}} e^{\frac{w^2}{8}} \Phi(-w/2)} \cdot \frac{1}{1 - \Phi(-w/2)} = e^{-\frac{w^2}{8}} \sqrt{\frac{2}{\pi}} \cdot \frac{1}{\Phi(-w/2) \Phi(w/2)}.
\end{align*}

\subsection{Proof of Lemma~\ref{lem:Jm-integral-term-upper-bound}} \label{proof:Jm-integral-term-upper-bound}
We pick some constant $C > 0$. When $w \leq C\sqrt{m}/t$,
\begin{align*}
	& \prn{\frac{2e^{\frac{tw}{2\sqrt{m}}}}{1 + e^{\frac{tw}{\sqrt{m}}}}}^{m} e^{\frac{tw}{\sqrt{m}}} w^\beta \nonumber \\
	& = \prn{1 - \frac{\prn{1-e^{-\frac{tw}{2\sqrt{m}}}}^2}{1 + e^{-\frac{tw}{\sqrt{m}}}}}^{m} e^{\frac{tw}{\sqrt{m}}} w^\beta \leq e^C  w^\beta  \exp \brc{-\frac{m\prn{1-e^{-\frac{tw}{2\sqrt{m}}}}^2}{1 + e^{-\frac{tw}{\sqrt{m}}}} }  \nonumber \\
	&  \leq e^C  w^\beta \exp \brc{-\frac{1}{2} \cdot m\prn{1-e^{-\frac{tw}{2\sqrt{m}}}}^2 } = e^C  w^\beta \exp \brc{-\frac{t^2w^2}{2} \cdot \prn{\frac{\sqrt{m}}{tw} \cdot \prn{1-e^{-\frac{tw}{2\sqrt{m}}}}}^2 }  \nonumber \\
	& \leq e^C w^\beta \cdot \exp \left\{-\frac{t^2 w^2}{2} \cdot \prn{\inf_{z \in (0, C]} \frac{1-e^{-z/2}}{z} }^2 \right\} ,
\end{align*}
and when $w > C\sqrt{m}/t$,
\begin{align*}
	\prn{\frac{2e^{\frac{tw}{2\sqrt{m}}}}{1 + e^{\frac{tw}{\sqrt{m}}}}}^{m} e^{\frac{tw}{\sqrt{m}}} w^\beta & = \prn{\frac{2e^{-\frac{tw}{2\sqrt{m}}}}{1 + e^{-\frac{tw}{\sqrt{m}}}}}^{m} e^{\frac{tw}{\sqrt{m}}} w^\beta \nonumber \\
	& \leq \prn{2e^{-\frac{tw}{4\sqrt{m}}}}^m \cdot \exp \brc{-\frac{\sqrt{tw}}{4\sqrt{m}} \cdot \sqrt{tw}} e^{\frac{tw}{\sqrt{m}}}w^\beta \nonumber \\
	& \leq \prn{2e^{-\frac{tw}{4\sqrt{m}} + \frac{tw}{m\sqrt{m}}}}^m \cdot w^\beta \exp \brc{-\frac{C}{4} \sqrt{tw}} \nonumber \\
	& \leq \prn{2e^{-C\prn{\frac{1}{4} - \frac{1}{m}}}}^m  w^\beta \exp \brc{-\frac{C}{4} \sqrt{tw}} .
\end{align*}
Then for $m \geq 8$ and $C > 8 > \log 2$, it holds $2e^{-C\prn{\frac{1}{4} - \frac{1}{m}}} < 1$ and consequently
\begin{align*}
	\prn{\frac{2e^{\frac{tw}{2\sqrt{m}}}}{1 + e^{\frac{tw}{\sqrt{m}}}}}^{m} e^{\frac{tw}{\sqrt{m}}} w^\beta \le w^\beta \exp \brc{-\frac{C}{4} \sqrt{tw}}.
\end{align*}

\subsection{Proof of Corollary~\ref{cor:LAM-lower-bounds}} \label{proof:LAM-lower-bounds}
As we have the explicit formula for the density $p_\theta(X,\majY) = \Prb(\majY \mid X) q(X)$, with $\Prb(\majY = \sgn(\<x, \theta^\star\>) \mid X = x) = \rho_m(|\<x, \theta^\star\>|) $ and $q$ being the $d$-dimensional centered density with covariance $\Sigma$. It is differentiable in quadratic mean and allows us to invoke \citet[Thm.~8.9]{VanDerVaart98}. It only remains to show the asymptotic rate in $m$ for $\Sigma_{\theta\opt}$ and $\Sigma_{u\opt}$. Indeed, by Lemma~\ref{lem:Sigma-inverse},
\begin{align*}
	\Sigma_{\theta\opt} = \fisher_m(\theta\opt)^{-1} = A_m (t\opt)^{-1} \cdot u\opt {u\opt}^\top + B_m(t\opt)^{-1} \cdot \prn{\proj_{u\opt}^\perp \Sigma \proj_{u\opt}^\perp}^\dagger,
\end{align*}
and  Theorem~\ref{thm:lower-bound-fisher-information} implies
\begin{align*}
	\Sigma_{\theta\opt} \asymp {t\opt}^{\beta +2} m^{\frac{\beta}{2}} \cdot u\opt {u\opt}^\top +  {t\opt}^{\beta}m^{\frac{\beta-2}{2}}  \cdot \prn{\proj_{u\opt}^\perp \Sigma \proj_{u\opt}^\perp}^\dagger.
\end{align*}
Analogously we establish the proof for $\Sigma_{u\opt}$.

\section{Proof of Theorem~\ref{theorem:semiparametric-master}}
\label{proof:semiparametric-master}

We divide the theorem into two main parts, consistent with the typical
division of asymptotic normality results into a consistency result and a
distributional result. Recall the notation $L(\theta,
\vec{\sigma}) = \E[\loss_{\vec{\sigma},\theta}(Y \mid X)]$
and
$\LtwoP{\vec{\sigma} - \vec{g}}^2
= \E[\ltwo{\vec{\sigma}(Z) - \vec{g}(Z)}^2]$,
where $Z$ has the distribution
that Assumption~\ref{assumption:x-decomposition} specifies.

\subsection{Proof of consistency}

We demonstrate the consistency $\spe \cp u\opt$ in three parts, which we
present as Lemmas~\ref{lemma:super-generic-solutions},
\ref{lemma:vec-sigma-theta-continuity}, and
\ref{lemma:vec-sigma-probability}. The first presents an analogue of
Lemma~\ref{lemma:generic-solutions} generalized to the case in
which there are $m$ distinct link functions, allowing us to characterize
the link-dependent minimizers
\begin{equation*}
  \theta\opt_{\vec{\sigma}} \defeq \argmin_\theta L(\theta, \vec{\sigma})
\end{equation*}
via a one-dimensional scalar on the line $\{t u\opt \mid t \in \R_+\}$.  The
second, Lemma~\ref{lemma:vec-sigma-theta-continuity}, then shows that
$\theta\opt_{\vec{\sigma}} \to u\opt$ as $\vec{\sigma}$ approaches
$\vec{\sigma}\opt$, where the scaling that
$\ltwos{\theta\opt_{\vec{\sigma}}} \to 1$ follows by the normalization
Assumption~\ref{assumption:sp-general}.
Finally, the third lemma
demonstrates the probabilistic convergence $\spe -
\theta\opt_{\vec{\sigma}_n} \cp 0$ whenever $\vec{\sigma}_n \cp
\vec{\sigma}\opt$ in $L^2(P)$.

We begin with the promised analogue of
Lemma~\ref{lemma:generic-solutions}.
\begin{lemma}
  \label{lemma:super-generic-solutions}
  Define the calibration gap function
  \begin{equation}
    \label{eqn:general-h-m-func}
    h_{\vec{\sigma}}(t)
    \defeq \frac{1}{m} \sum_{j = 1}^m
    \E\left[\left(\sigma_j(tZ) - \sigma_j\opt(Z)\right) Z\right].
  \end{equation}
  Then the loss $L(\theta, \vec{\sigma})$ has unique
  minimizer $\theta\opt_{\vec{\sigma}} = t_{\vec{\sigma}} u\opt$ for the
  unique $t_{\vec{\sigma}} \in (0, \infty)$ solving
  $h_{\vec{\sigma}}(t) = 0$. Additionally,
  taking
  $\hessfunc_j(t, z) \defeq
  \sigma_j'(-tz) \sigma_j\opt(z) + \sigma_j'(tz) \sigma_j\opt(-z)$,
  we have
  \begin{equation*}
    \nabla^2 L(t u\opt, \vec{\sigma})
    = \frac{1}{m} \sum_{j = 1}^m \left(
    \E[\hessfunc_j(t, Z) Z^2] u\opt {u\opt}^\top
    + \E[\hessfunc_j(t, Z)] \projperp \Sigma \projperp \right)
    \succ 0.
  \end{equation*}
\end{lemma}
\begin{proof}
	We perform a derivation similar to that we used
	to derive the gap function~\eqref{eqn:h-m-func}, with a few
	modifications to allow collections of $m$ link functions.
	Note that
	\begin{align*}
		L(\theta, \vec{\sigma})
		& = \frac{1}{m} \sum_{j = 1}^m
		\E[\loss_{\sigma_j, \theta}(1 \mid X) \sigma\opt_j(\<X, u\opt\>)
		+ \loss_{\sigma_j,\theta}(-1 \mid X) \sigma\opt_j(-\<X, u\opt\>)],
	\end{align*}
	so that leveraging the usual ansatz that $\theta = t u\opt$, we have
	\begin{align*}
		\nabla L(\theta, \vec{\sigma})
		& = \E\left[\frac{1}{m}
		\sum_{j = 1}^m \left(\sigma\opt_j(-\<X, u\opt\>) \sigma_j(\<\theta, X\>)
		- \sigma\opt_j(\<X, u\opt\>) \sigma_j(-\<\theta, X\>)\right) X \right] \\
		& = \frac{1}{m} \sum_{j = 1}^m
		\E\left[\left(\sigma_j\opt(-Z) \sigma_j(t Z)
		- \sigma_j\opt(Z) \sigma_j(-tZ) Z\right)\right] u\opt \\
		& = h_{\vec{\sigma}}(t) u\opt,
	\end{align*}
	where $h_{\vec{\sigma}}(t) = \frac{1}{m} \sum_{j = 1}^m
	\E[(\sigma_j\opt(-Z) \sigma_j(tZ) - \sigma_j\opt(Z) \sigma_j(-tZ)) Z]$ is
	the immediate generalization of the gap~\eqref{eqn:h-m-func}. We
	now simplify it to the form~\eqref{eqn:general-h-m-func}. Using the
	symmetries $\sigma\opt(z) = 1 - \sigma\opt(-z)$ and $\sigma(z) = 1 -
	\sigma(-z)$ for any $\sigma, \sigma\opt \in \Flink$, we observe that
	\begin{equation*}
		\sigma(tz) (1 - \sigma\opt(z))
		- \sigma(-tz) \sigma\opt(z)
		= \sigma(tz) - (\sigma(tz) + \sigma(-tz)) \sigma\opt(z)
		= \sigma(tz) - \sigma\opt(z),
	\end{equation*}
	and so
	$h_{\vec{\sigma}}(t) = \frac{1}{m} \sum_{j = 1}^m\E[(\sigma_j(tZ) -
	\sigma_j\opt(Z))Z]$.
	Of course, $h_{\vec{\sigma}}(0) = -\frac{1}{m}
	\sum_{j = 1}^m \E[\sigma_j\opt(Z) Z] < 0$, as $\E[Z] = 0$, while
	$\lim_{t \to \infty} h_{\vec{\sigma}}(t) =
	\frac{1}{m} \sum_{j = 1}^m \E[|Z| - \sigma_j\opt(Z) Z] > 0$.
	Then as $h_{\vec{\sigma}}'(t) =
	\frac{1}{m} \sum_{j = 1}^m \E[\sigma_j'(tZ) Z^2] > 0$, there
	exists a unique $t_{\vec{\sigma}} \in (0, \infty)$ satisfying
	$h_{\vec{\sigma}}(t_{\vec{\sigma}}) = 0$.
	
	We turn to the Hessian derivation, where as in the proof of
	Lemma~\ref{lemma:generic-solutions}, we
	write for $\theta = tu\opt$ that
	\begin{align*}
		\nabla^2 L(\theta, \vec{\sigma})
		& = \frac{1}{m} \sum_{j = 1}^m \E\left[
		(\sigma_j'(-t Z) \sigma\opt_j(Z)
		+ \sigma_j'(tZ) \sigma\opt_j(-Z)) Z^2\right] u\opt{u\opt}^\top \\
		& \qquad ~ +
		\frac{1}{m} \sum_{j = 1}^m \E\left[
		(\sigma_j'(-t Z) \sigma\opt_j(Z)
		+ \sigma_j'(tZ) \sigma\opt_j(-Z))\right] \projperp \Sigma
		\projperp,
	\end{align*}
	and so $\nabla^2 L(\theta, \vec{\sigma}) \succ 0$ and $L(\theta,
	\vec{\sigma})$ has unique minimizer $\theta_{\vec{\sigma}}\opt =
	t_{\vec{\sigma}} u\opt$.
\end{proof}

With Lemma~\ref{lemma:super-generic-solutions} serving as the analogue
of Lemma~\ref{lemma:generic-solutions}, we can now show
the continuity of the optimizing parameter $\theta\opt_{\vec{\sigma}}$
in $\vec{\sigma}$:
\begin{lemma}
	\label{lemma:vec-sigma-theta-continuity}
	As $\LtwoP{\vec{\sigma} - \vec{\sigma}\opt} \to 0$,
	we have $\theta\opt_{\vec{\sigma}} - u\opt \to 0$.
\end{lemma}
\begin{proof}
	Via Lemma~\ref{lemma:super-generic-solutions}, it is evidently
	sufficient to show that the solution $t_{\vec{\sigma}}$ to
	$h_{\vec{\sigma}}(t) = 0$ converges to $1$.
	To show this,
	note the expansion
	\begin{equation}
		\label{eqn:expand-h-m}
		m h_{\vec{\sigma}}(t) =
		\sum_{j = 1}^m \left(\E[(\sigma_j(tZ) - \sigma_j\opt(tZ)) Z]
		+ \E[(\sigma_j\opt(tZ) - \sigma_j\opt(Z))Z]
		\right).
	\end{equation}
	We use the following claim, which shows that
	the first term tends to 0 uniformly in $t$ near $1$:
	\begin{claim}
		\label{claim:taco}
		For any $\sigma, \sigma\opt \in \Flink$, we have
		$\sup_{t \in [\half,2]} \LtwoP{\sigma(tZ) - \sigma\opt(tZ)}
		\to 0$ whenever $\LtwoP{\sigma\opt - \sigma} \to 0$.
	\end{claim}
	\begin{proof}
		We use that the density $p(z)$ of $|Z|$ is continuous
		and nonzero on $(0, \infty)$. Take any $0 < M_0 < M_1 < \infty$.
		Then
		\begin{align*}
			\lefteqn{\sup_{t \in [\half, 2]}
				\LtwoP{\sigma\opt(tZ) - \sigma(tZ)}
				=  \sup_{t \in [\half, 2]} \sqrt{\int_0^\infty \prn{\sigma\opt(tz) - \sigma(tz)}^2 p(z) dz}} \\
			& \stackrel{(i)}{\leq}
			\Prb(|Z| \leq M_0) + \Prb(|Z| \geq M_1) +
			\sup_{t \in [\half, 2]} \sqrt{\int_{M_0}^{M_1}
				\prn{\sigma\opt(tz) - \sigma(tz)}^2 p(tz) \cdot
				\frac{p(z)}{p(tz)} dz} \nonumber \\
			& \leq \Prb(|Z| \leq M_0) + \Prb(|Z| \geq M_1) + \sup_{
				\frac{M_0}{2} \le z, z'
				\le 2 M_1}
			\sqrt{\frac{p(z)}{p(z')}}
			\sup_{t \in [\half, 2]} \sqrt{\int_{M_0}^{M_1}
				\prn{\sigma\opt(tz) - \sigma(tz)}^2 p(tz) dz}  ,
		\end{align*}
		where in $(i)$ we use that the link functions $\sigma\opt$ and $\sigma$
		are bounded within $[0,1]$. For any fixed $0 < M_0 \le M_1 < \infty$,
		the ratio $\frac{p(z)}{p(z')}$ is bounded
		for $z, z' \in [\half M_0, 2 M_1]$, and using
		the substitution $tz \mapsto z$ we have
		\begin{align*}
			\sup_{t \in [\half, 2]} \sqrt{
				\int_{M_0}^{M_1}
				\prn{\sigma\opt(tz) - \sigma(tz)}^2 p(tz) dz}
			\leq \sqrt{2} \cdot \LtwoP{\sigma\opt - \sigma}.
		\end{align*}
		We thus have that
		$\LtwoP{\sigma(tZ) - \sigma\opt(tZ)}
		\le K \LtwoP{\sigma - \sigma\opt}
		+ \P(|Z| \not \in [M_0, M_1])$,
		where $K$ depends only on $M_0, M_1$ and the
		distribution of $Z$.
		Take $M_0 \downarrow 0$ and $M_1 \uparrow \infty$.
	\end{proof}
	
	Leveraging the expansion~\eqref{eqn:expand-h-m} preceding
	Claim~\ref{claim:taco} and the claim itself, we see that
	\begin{equation*}
		h_{\vec{\sigma}}(t) = \E[Z^2] \cdot o(1) +
		\frac{1}{m} \sum_{j = 1}^m
		\E[(\sigma_j\opt(tZ) - \sigma_j\opt(Z))Z]
	\end{equation*}
	uniformly in $t \in [\half, 2]$ as $\LtwoP{\vec{\sigma}\opt -
		\vec{\sigma}} \to 0$. The monotonicity of each $\sigma_j\opt$ guarantees
	that if $f_j(t) = \E[(\sigma_j\opt(tZ) - \sigma_j\opt(Z))Z]$, then
	$f'_j(t) = \E[{\sigma_j\opt}'(tZ) Z^2] > 0$, and so $t = 1$ uniquely
	solves $f_j(t) = 0$ and we must have $t_{\vec{\sigma}} \to 1$ as
	$\LtwoP{\vec{\sigma} - \vec{\sigma}\opt} \to 0$.
\end{proof}

Finally, we proceed to the third part of the consistency argument:
the convergence in probability.
\begin{lemma}
	\label{lemma:vec-sigma-probability}
	If $\LtwoP{\vec{\sigma}_n - \vec{\sigma}} \cp 0$,
	then $\spe - \theta_{\vec{\sigma}_n}\opt \cp 0$.
\end{lemma}
\begin{proof}
	By Lemma~\ref{lemma:super-generic-solutions}
	and the assumed continuity of the population Hessian,
	there exists $\delta > 0$ such that
	\begin{equation}
		L(\theta, \vec{\sigma}) \ge L(\theta\opt_{\vec{\sigma}}, \vec{\sigma})
		+ \frac{\lambda}{2} \ltwo{\theta - \theta\opt_{\vec{\sigma}}}^2
		\label{eqn:hessian-lb}
	\end{equation}
	whenever both $\LtwoP{\vec{\sigma} - \vec{\sigma}\opt} \le \delta$ and
	$\ltwo{\theta\opt_{\vec{\sigma}} - \theta} \le \delta$. Applying
	the uniform convergence
	Lemma~\ref{lemma:rademacher-symmetrization}, we see that
	for any $r < \infty$, we have
	\begin{equation*}
		\sup_{\ltwo{\theta} \le r, \vec{\sigma} \in \Flink^m}
		|P_n \loss_{\vec{\sigma},\theta} - L(\theta, \vec{\sigma})| \cp 0.
	\end{equation*}
	For $\delta > 0$, define the events
	\begin{equation*}
		\event_n(\delta) \defeq \left \{
		\ltwo{\theta\opt_{\vec{\sigma}} - u\opt} \le \delta,
		\LtwoP{\vec{\sigma}_n - \vec{\sigma}\opt} \le \delta \right\},
	\end{equation*}
	where
	Lemma~\ref{lemma:vec-sigma-theta-continuity}
	and the assumption that $\LtwoP{\vec{\sigma}_n - \vec{\sigma}\opt}
	\cp 0$ imply that $\P(\event_n(\delta)) \to 1$ for all $\delta > 0$.
	By the growth condition~\eqref{eqn:hessian-lb}
	and uniform convergence
	$P_n \loss_{\vec{\sigma},\theta} - L(\theta,\vec{\sigma}) \cp 0$
	over $\ltwo{\theta} \le r$,
	we therefore have that with probability tending to 1, 
	\begin{equation*}
		\inf_{\ltwos{\theta - \theta\opt_{\vec{\sigma}_n}} = \delta}
		\left\{P_n \loss_{\theta,\vec{\sigma}_n} - P_n \loss_{\theta\opt_{\vec{\sigma}_n},
			\vec{\sigma}_n} \right\}
		\ge \frac{\lambda}{4} \delta^2.
	\end{equation*}
	The convexity of the losses $\loss_{\theta,\vec{\sigma}}$ in $\theta$ and that
	$\spe$ minimizes $P_n \loss_{\theta,\vec{\sigma}_n}$ then
	guarantee the desired convergence $\spe - \theta\opt_{\vec{\sigma}_n} \cp 0$.
\end{proof}

\subsection{Asymptotic normality via Donsker classes}

\providecommand{\G}{\mathbb{G}}

While we do not have $\LtwoP{\vec{\sigma}_n - \vec{\sigma}\opt} =
o_P(n^{-1/2})$, which allows the cleanest and simplest asymptotic normality
results with nuisance parameters
(e.g.~\cite[Thm.~25.54]{VanDerVaart98}), we still expect
$\sqrt{n}(\spe - \theta\opt_{\vec{\sigma}_n})$ to be asymptotically
normal, and therefore, as $\theta\opt_{\vec{\sigma}} = t u\opt$ for some
$t > 0$, the normalized estimators $\spe / \ltwos{\spe}$
should be asymptotically normal to $u\opt$.
To develop the asymptotic normality results, we perform an analysis
of the empirical process centered at the \emph{estimators}
$\spe$, rather than the ``true'' parameter $u\opt$ as would
be typical.

We begin with an expansion. We let $\theta\opt_n =
\theta\opt_{\vec{\sigma}_n}$ for shorthand.
Then as $\P_{n,m} \nabla_\theta \ell_{\spe,
	\vec{\sigma}_n} = 0$ and $\P \nabla_\theta \ell_{\theta\opt_n,
	\vec{\sigma}_n} = 0$, we can derive
\begin{align*}
	\G_{n,m} \nabla_\theta\ell_{\spe, \vec{\sigma}_n} & = \sqrt{n} \prn{\P_{n,m} \nabla_\theta \ell_{\spe, \vec{\sigma}_n} - \P \nabla_\theta \ell_{\spe, \vec{\sigma}_n}} \nonumber \\
	& = \sqrt{n} \prn{\P \nabla_\theta\ell_{\theta\opt_n, \vec{\sigma}_n} - \P \nabla_\theta\ell_{\spe, \vec{\sigma}_n}} \nonumber \\
	& = \sqrt{n} \prn{\nabla_\theta \poploss(\theta\opt_n, \vec{\sigma}_n) - \nabla_\theta \poploss (\spe, \vec{\sigma}_n) } \nonumber \\
	& = \prn{\int_0^1 \nabla_{\theta}^2 \poploss((1-t) \spe + t \theta\opt_n , \vec{\sigma}_n ) d t} \cdot \sqrt{n}(\theta\opt_n - \spe).
\end{align*}
The assumed continuity of $\nabla_\theta^2 \poploss(\theta, \vec{\sigma})$
at $(u\opt, \vec{\sigma}\opt)$ (recall
Assumption~\ref{assumption:sp-general}) then implies that
\begin{align}
	\sqrt{n}(\theta\opt_n - \spe)
	& = \prn{\int_0^1 \nabla_{\theta}^2 \poploss((1-t) \theta\opt_n + t \spe, \vec{\sigma}_n ) d t}^{-1}
	\cdot\, \G_{n,m} \nabla_\theta\loss_{\spe, \vec{\sigma}_n} \nonumber \\
	& =
	\prn{\nabla_\theta^2 \poploss(u\opt, \vec{\sigma}\opt)  + o_P(1)}^{-1}
	\cdot \G_{n,m} \nabla_\theta \loss_{\spe, \vec{\sigma}_n},
	\label{eqn:sp-mle-mid-1}
\end{align}
where we have used the consistency guarantees $\theta\opt_n \cp u\opt, \spe
\cp u\opt$ by Lemmas~\ref{lemma:vec-sigma-theta-continuity}
and~\ref{lemma:vec-sigma-probability} and that $\LtwoP{\vec{\sigma}_n -
	\vec{\sigma}} \cp 0$ by assumption.

The expansion~\eqref{eqn:sp-mle-mid-1} forms the basis of our asymptotic
normality result; while $\spe$ and $\vec{\sigma}_n$ may be data
dependent, by leveraging uniform central limit theorems and the
theory of Donsker function classes, we can show that
$\G_{n,m} \nabla \loss$ has an appropriate normal limit.
To that end,
define the
function classes
\begin{align*}
	\mc{F}_{\delta}
	\defeq \left\{\nabla_\theta \loss_{\theta, \sigma}
	\mid \ltwo{\theta - u\opt} \leq \delta,
	\sigma \in \Flink \right\},
\end{align*}
where we leave the Lipschitz constant $\lipconst$ in $\Flink$ tacit.
By Assumption~\ref{assumption:sp-general} and that $\spe \cp u\opt$, we know
with probability tending to $1$ we have the membership $
\nabla_\theta\ell_{\spe, \vec{\sigma}_n} \in \mc{F}_\delta$. The key
result is then that $\mathcal{F}_\delta$ is a Donsker class:
\begin{lemma}
  \label{lemma:donsker-sp}
  Assume that $\Ep [\ltwo{X}^4] < \infty$.  Then $\mc{F}_\delta$ is a
  Donsker class, and moreover, if $d_{\Flink}((\spe, \vec{\sigma}_n),
  (u\opt, \vec{\sigma}\opt)) \cp 0$, then
  \begin{equation*}
    \G_{n,m} \nabla_\theta\ell_{\spe, \vec{\sigma}_n}
    - \G_{n,m} \nabla_\theta\ell_{u\opt, \vec{\sigma}\opt} \cp 0.
  \end{equation*}
\end{lemma}

Temporarily deferring the proof of Lemma~\ref{lemma:donsker-sp}, let us see
how it leads to the proof of Theorem~\ref{theorem:semiparametric-master}.
Using Lemma~\ref{lemma:donsker-sp} and Slutsky's lemmas in the
equality~\eqref{eqn:sp-mle-mid-1}, we obtain
\begin{align*}
	\sqrt{n}(\spe - \theta_n\opt)
	& = 
	-\prn{\nabla_\theta^2 \poploss(u\opt, \vec{\sigma}\opt)  + o_P(1)}^{-1}
	\cdot \G_{n,m} \nabla_\theta \loss_{u\opt, \vec{\sigma}\opt}
	+ o_P(1) \\
	& \cd \normal\left(0,
	\nabla^2 L(u\opt, \vec{\sigma}\opt)
	\cov\bigg(\frac{1}{m}
	\sum_{j = 1}^m \nabla \loss_{u\opt, \sigma_j\opt}(Y_j \mid X)\bigg)
	\nabla^2 L(u\opt, \vec{\sigma}\opt)\right).
\end{align*}
Calculations completely similar to those we use in the proof of
Theorem~\ref{theorem:multilabel-asymptotics} then give
Theorem~\ref{theorem:semiparametric-master}: because $\P(Y_j = y \mid X = x)
= \sigma_j\opt(y \<u\opt, x\>)$ by assumption,
\begin{equation*}
	\cov\bigg(\frac{1}{m}
	\sum_{j = 1}^m \nabla \loss_{u\opt, \sigma_j\opt}(Y_j \mid X)\bigg)
	= \frac{1}{m^2} \sum_{j = 1}^m
	\cov(\nabla \loss_{u\opt, \sigma_j\opt}(Y_j \mid X))
\end{equation*}
because
$Y_j \mid X$ are conditionally independent, while
\begin{align*}
	\cov(\nabla \loss_{u\opt,\sigma_j\opt}(Y_j \mid X))
	& = \E[\sigma_j\opt(Z)(1 - \sigma_j\opt(Z)) XX^\top] \\
	& = \E[\sigma_j\opt(Z)(1 - \sigma_j\opt(Z)) Z^2]
	u\opt{u\opt}^\top
	+ \E[\sigma_j\opt(Z)(1 - \sigma_j\opt(Z))]
	\projperp \Sigma \projperp.
\end{align*}
When
$\sigma_j = \sigma_j\opt$, the Hessian function $\hessfunc_j$
in Lemma~\ref{lemma:super-generic-solutions}
simplifies to $\hessfunc_j(1, z) = {\sigma_j\opt}'(z)$ as
$\sigma_j\opt$ is symmetric about 0.
We then apply the delta method
as in the proof of Theorem~\ref{theorem:multilabel-asymptotics}.

Finally, we return to the proof of Lemma~\ref{lemma:donsker-sp}.

\paragraph{Proof of Lemma~\ref{lemma:donsker-sp}.}

\newcommand{\ball}{\mathbb{B}}
\newcommand{\Fflip}{\mc{F}_{\mathsf{flip}}}

To prove $\mathcal{F}_\delta$ is Donsker, we show that each coordinate of
$\nabla_{\theta} \loss_{\theta, \sigma} \in \mathcal{F}_\delta$ is, and as
$\nabla_\theta \loss_{\theta, \sigma} (y \mid x) = -y\sigma(-y\<x,
\theta\>)x$, this amounts to showing the coordinate functions $f_{\theta,
	\sigma}^{(i)}(v) = v_i\sigma(-\<v, \theta\>)$ form a Donsker class when
$v$ has distribution $V = YX$. Let
\begin{align*}
	\mathcal{F}_\delta^{(i)}
	& \defeq \left\{ f_{\theta, \sigma}^{(i)}(\cdot) \mid
	\ltwo{\theta - u\opt} \leq \delta, ~
	\sigma \in \Flink \right\},
\end{align*}
so it is evidently sufficient to prove
that $\mathcal{F}_\delta^{(1)}$ forms a Donsker class.

We use bracketing and entropy numbers~\cite{VanDerVaartWe96} to control the
$\mc{F}_\delta^{(i)}$. Recall that for a function class $\mc{F}$, an
\emph{$\epsilon$-bracket} of $\mc{F}$ in $L^q(\P)$ is a collection of
functions $\{(l_i, u_i)\}$ such that for each $f \in \mc{F}$, there exists
$i$ such that $\ell_i \le f \le u_i$ and $\norm{u_i - l_i}_{L^q(\P)} \le
\epsilon$. The \emph{bracketing number} $N_{[\,]}(\epsilon, \mc{F}, L^q(\P))$
is the cardinality of the smallest such $\epsilon$-bracket, and
the \emph{bracketing entropy} is
\begin{align*}
	J_{[\,]}(\mc{F}, L^q(\P))
	\defeq \int_0^\infty
	\sqrt{\log N_{[\,]}(\epsilon , \mathcal{F}, L^q(\Prb))} d \epsilon.
\end{align*}
To show that $\mc{F}_\delta^{(i)}$ is Donsker, it is
sufficient~\cite[Ch.~2.5.2]{VanDerVaartWe96} to show that
$J_{[\,]}(\mc{F}_\delta^{(i)}, L^2(\P)) < \infty$.
Our approach to demonstrate that $\mc{F}_\delta^{(1)}$ is Donsker is thus to
construct an appropriate $\epsilon$-bracket of $\mc{F}_\delta^{(1)}$, which
we do
by first covering $\ell_2$-balls in $\R^d$, then for vectors
$\theta \in \R^d$, constructing a bracketing of the induced function class
$\{f_{\theta,\sigma}^{(1)}\}_{\sigma \in \Flink}$, which we combine
to give the final bracketing of $\mc{F}_\delta^{(1)}$.

We proceed with this two-stage covering and bracketing.
Let $\epsilon, \gamma > 0$ be small numbers whose values we determine
later.
Define the $\ell_2$ ball $\ball_2^d = \{x \in \R^d \mid  \ltwo{x} \le 1\}$,
and for any $0 < \epsilon < \delta$, let
$\mc{N}_\epsilon$ be a minimal $\epsilon$-cover of $\delta \ball_2^d$
in the Euclidean norm
of size $N = N(\epsilon, \delta \ball_2^d, \ltwo{\cdot})$,
$\mc{N}_\epsilon = \{\theta_1, \ldots, \theta_N\}$, so that for
any $\theta$ with $\ltwo{\theta} \le \delta$ there exists
$\theta_i \in \mc{N}_\epsilon$ with $\ltwo{\theta - \theta_i} \le \epsilon$.
Standard bounds~\cite[Lemma~5.7]{Wainwright19}) give
\begin{align}
	\label{eqn:ball-covering-number}
	\log N(\epsilon, \delta \ball_2^d, \ltwo{\cdot})
	\le d \log \prn{1 + \frac{2\delta}{\epsilon}}.
\end{align}
For simplicity of notation and to avoid certain tedious negations, we define
the ``flipped'' monotone function family
\begin{equation*}
	\Fflip \defeq \{g : \R \to \R \mid g(t) = \sigma(-t)\}_{
		\sigma \in \Flink}.
\end{equation*}
Now, for any $\theta \in \R^d$, let
$\mu_\theta$
denote the pushforward measure of $\<V, \theta\> = Y\<X, \theta\>$.
For $\theta \in \R^d$, we then let
$\mathcal{N}_{[\,],\gamma, \theta}$ be a minimal
$\gamma$-bracketing of
$\Fflip$
in the  $L^4(\mu_{\theta})$ norm. That is,
for $N = N_{[\,]}(\gamma, \Fflip, L^2(\mu_\theta))$, we have
$\mathcal{N}_{[\,], \gamma, \theta}
= \{(l_{\theta, i}, u_{\theta, i})\}_{i = 1}^{N}$,
and for each $\sigma \in \Flink$,
there exists $i = i(\sigma)$ such that
\begin{align*}
	l_{\theta, i}(t) \leq
	\sigma(-t) \leq
	u_{\theta, i}(t)
	~~ \mbox{and} ~~ \norm{u_{\theta, i} - l_{\theta, i}}_{L^4(\mu_{\theta})}
	\leq \gamma.
\end{align*}
Because elements of $\Fflip \subset \R \to [0, 1]$ are monotone,
\citet[Thm.~2.7.5]{VanDerVaartWe96} guarantee there
exists a universal constant $K < \infty$ such that
\begin{align}
	\label{eqn:monotone-covering-bound}
	\sup_Q \log N_{[\,]}(\gamma, \Fflip, L^4(Q)) \leq \frac{K}{\gamma},
\end{align}
and in particular,
$\log N_{[\,]}(\gamma, \Fflip, L^4(\mu_\theta)) \le \frac{K}{\gamma}$ for
each $\theta \in \R^d$.

With the covering $\mc{N}_\epsilon$ and induced bracketing collections
$\mc{N}_{[\,],\gamma,\theta}$, we now turn to a construction of the
actual bracketing of the class $\mc{F}_\delta^{(1)}$.
For any $\theta \in \R^d$ and bracket $(l_{\theta,i},
u_{\theta, i}) \in \mathcal{N}_{[\,], \gamma, \theta}$, define
the functionals $\what{l}_{\theta,j}, \what{u}_{\theta,j} : \R^d \to \R$ by
\begin{align*}
	\what{l}_{\theta, j}(v)
	& \defeq \hinge{v_1}
	\max\left\{l_{\theta,j}(\<v, \theta\>) - \mathsf{L} \ltwo{v} \epsilon,
	0 \right\}
	- \hinge{-v_1} \min\left\{u_{\theta, j}(\<v, \theta\>) + \mathsf{L} \ltwo{v} \epsilon, 1
	\right\}  , \\
	\what{u}_{\theta, j}(v)
	& \defeq \hinge{v_1} \min\left\{u_{\theta, j}(\<v, \theta\>)
	+ \lipconst \ltwo{v} \epsilon, 1 \right\}
	- \hinge{-v_1} \max\left\{l_{\theta, j}(\<v, \theta\>) -
	\lipconst \ltwo{v} \epsilon, 0 \right\}.
\end{align*}
The key is that these functions form a bracketing
of $\mc{F}_\delta^{(1)}$:
\begin{lemma}
	\label{lemma:bracketing-construction}
	Define the set
	\begin{equation*}
		\mc{B}_{\epsilon,\gamma}
		\defeq
		\left\{ \left(\what{l}_{\theta_i, j},
		\what{u}_{\theta_i, j}\right)
		\mid
		\theta_i \in \mc{N}_\epsilon,
		~ 1 \le j \le N_{[\,]}(\gamma, \Fflip, L^4(\mu_{\theta_i}))\right\}.
	\end{equation*}
	Then
	$\mc{B}_{\epsilon,\gamma}$ is a
	\begin{equation*}
		2 \lipconst \E[\ltwo{X}^4]^{1/2} \cdot \epsilon
		+ \E[\ltwo{X}^4]^{1/4} \cdot \gamma
	\end{equation*}
	bracketing
	of $\mc{F}_\delta^{(1)}$ with cardinality
	at most
	$\log \card(\mc{B}_{\epsilon,\gamma})
	\le \frac{K}{\gamma} + d \log(1 + \frac{\delta}{\epsilon})$.
\end{lemma}
\begin{proof}
	Let $f_{\theta, \sigma}^{(1)}(v) \in \mathcal{F}_\delta^{(1)}$.
	Take $\theta_i \in \mathcal{N}_{\epsilon}$ satisfying
	$\ltwo{\theta - \theta_i} \leq \epsilon$ and
	$(l_{\theta_i, j}, u_{\theta,j}) \in \mc{N}_{[\,], \gamma, \theta_i}$
	such that $l_{\theta_i, j}(t) \leq  \sigma(-t) \leq u_{\theta_i, j}(t)$
	for all $t$, where $\norm{u_{\theta_i,j} - l_{\theta_i,j}}_{L^4(\mu_{\theta_i})}
	\le \gamma$. We first demonstrate the bracketing guarantee
	\begin{align*}
		\what{l}_{\theta_i, j}(v)
		\leq f_{\theta, \sigma}^{(1)}(v) = v_1 \sigma(-\<v, \theta\>)
		\leq \what{u}_{\theta_i, j}(v)
		~~ \mbox{for~all~} v \in \R^d.
	\end{align*}
	For the upper bound, we have
	\begin{align*}
		\lefteqn{f_{\theta, \sigma}^{(1)}(v) = v_1 \sigma(-\<v, \theta\>)}
		\\
		& \stackrel{\mathrm{(i)}}{\leq}
		\hinge{v_1} \min \left\{\sigma(-\<v, \theta_i\>) + \lipconst |\<v, \theta_i - \theta\>|, 1\right\}
		- \hinge{-v_1} \max \left\{\sigma(-\<v, \theta_i\>) - \lipconst |\<v, \theta_i - \theta\>|, 0\right\} \\
		& \stackrel{\mathrm{(ii)}}{\leq} \hinge{v_1}
		\min \left\{\sigma(-\<v, \theta_i\>) + \lipconst \ltwo{v} \epsilon, 1\right\}
		- \hinge{-v_1} \max \left\{\sigma(-\<v, \theta_i\>) - \lipconst \ltwo{v} \epsilon, 0\right\} \\
		& \stackrel{\mathrm{(iii)}}{\leq}
		\hinge{v_1} \min\left\{u_{\theta_i, j}(\<v, \theta_i\>) + \lipconst \ltwo{v} \epsilon, 1 \right\}
		- \hinge{-v_1} \max\left\{l_{\theta_i, j}(\<v, \theta_i\>) - \lipconst \ltwo{v} \epsilon, 0 \right\} \\
		& = \what{u}_{\theta_i, j}(v),
	\end{align*}
	where step (i) follows from the $\lipconst$-Lipschitz continuity of
	$\sigma$, (ii) from the Cauchy-Schwarz inequality and
	that $\ltwo{\theta - \theta_i} \le \epsilon$,
	while step (iii) follows
	by the construction that
	$l_{\theta_i, j}(t) \leq \sigma(-t) \leq u_{\theta_i,
		j}(t)$ for all $t \in \R$.
	Similarly, we obtain the lower bound
	\begin{align*}
		f_{\theta, \sigma}^{(1)}(v) = v_1 \sigma(-\<v, \theta\>) 
		\geq \what{l}_{\theta_i, j}(v),
	\end{align*}
	again valid for all $v \in \R^d$.
	
	The second part of the proof is to bound the distance between the upper
	and lower elements in the bracketing.  By definition,
	$\what{u}_{\theta_i, j} - \what{l}_{\theta_i, j}$ has the pointwise
	upper bound
	\begin{align*}
		\prn{\what{u}_{\theta_i, j}(v) - \what{l}_{\theta_i, j}(v)}^2 &\leq \prn{|v_1| \prn{u_{\theta_i, j}(\<v, \theta_i\>) - l_{\theta_i, j}(\<v, \theta_i\>) + 2 \lipconst \ltwo{v} \epsilon }}^2   .
	\end{align*}
	Recalling that $V = YX$,
	by the Minkowski and Cauchy-Schwarz inequalities, we thus obtain
	\begin{align*}
		\norm{\what{u}_{\theta_i, j}(V) - \what{l}_{\theta_i, j}(V)}_{L^2(\Prb)} & \leq \norm{|V_1| \prn{u_{\theta_i, j}(\<V, \theta_i\>) - l_{\theta_i, j}(\<V, \theta_i\>)} }_{L^2(\Prb)} + \norm{|V_1| \cdot 2 \lipconst \ltwo{V} \epsilon }_{L^2(\Prb)} \nonumber \\
		& \leq \norm{|V_1|}_{L^4(\Prb)} \cdot \prn{\norm{u_{\theta_i, j}(\<V, \theta_i\>) - l_{\theta_i, j}(\<V, \theta_i\>)}_{L^4(\Prb)} + 2\lipconst \epsilon \cdot \norm{\ltwo{V}}_{L^4(\Prb)}  }  .
	\end{align*}
	Noting the trivial bounds $ \norm{|V_1|}_{L^4(\Prb)} \le
	\norm{X}_{L^4(\Prb)} < \infty$ and the assumed bracketing distance
	\begin{align*}
		\norm{u_{\theta_i, j}(\<V, \theta_i\>) - l_{\theta_i, j}(\<V, \theta_i\>)}_{L^4(\Prb)} 	& = \norm{u_{\theta_i, j} - l_{\theta_i, j}}_{L^4(\mu_{\theta_i})}
		\le \gamma,
	\end{align*}
	we have
	the desired bracketing
	distance $\norms{\what{u}_{\theta_i, j} - \what{l}_{\theta_i, j}}_{L^2(\Prb)}
	\leq 2\lipconst \Ep[\ltwo{X}^4]^{1/2} \epsilon
	+ \Ep[\ltwo{X}^4]^{1/4} \gamma$.
	
	The final cardinality bound is immediate via
	inequalities~\eqref{eqn:ball-covering-number}
	and~\eqref{eqn:monotone-covering-bound}.
\end{proof}

Lemma~\ref{lemma:bracketing-construction} will yield the desired entropy
integral bound. Fix any $t > 0$, and note that if we take $\epsilon
=\epsilon(t) \defeq t / (4 \lipconst \E[\ltwo{X}^4]^{1/2})$ and $\gamma =
\gamma(t) \defeq t / (2 \E[\ltwo{X}^4]^{1/4})$, then the set
$\mc{B}_{\epsilon, \gamma}$ is a $t$-bracketing of $\mc{F}_\delta^{(1)}$ in
$L^2$, and moreover, we have the cardinality bound
\begin{align*}
	\log N_{[\,]}(t, \mathcal{F}_{\delta}^{(1)} , L^2(\Prb))
	&
	\leq d \log \left(1 + \frac{2 \delta}{\epsilon(t)}\right)
	+ \frac{K}{\gamma(t)}
	\leq \frac{8\lipconst d \delta \cdot \Ep [\ltwo{X}^4]^{1/2}
		+ 2K \Ep[\ltwo{X}^4]^{1/4}}{t}.
\end{align*}
Additionally, as covering numbers are necessarily integer,
we have $\log N_{[\,]}(t) = 0$
whenever
$t > (8\lipconst d \delta \cdot \Ep [\ltwo{X}^4]^{1/2}
+ 2K \Ep[\ltwo{X}^4]^{1/4}) / \log 2$.
This gives the entropy integral bound
\begin{align*}
	J_{[\, ]}(\mathcal{F}_{\delta}^{(1)}, L^2(\Prb))
	& = \int_0^\infty \sqrt{\log N_{[\,]}(t ,\mathcal{F}_{\delta}^{(1)} ,
		L^2(\Prb))} d t < \infty,
\end{align*}
and consequently (cf.~\cite[Thm.~19.5]{VanDerVaart98} or
\cite[Ch.~2.5.2]{VanDerVaartWe96}), $\mathcal{F}_{\delta}^{(1)}$ is a
Donsker class.  A completely identical argument shows that
$\mathcal{F}_{\delta}^{(i)}, i=2,3,\dots, d$ are Donsker, and so
$\mathcal{F}_\delta$ is a Donsker class, completing the proof of the
first claim in Lemma~\ref{lemma:donsker-sp}.

To complete the proof of Lemma~\ref{lemma:donsker-sp},
we need to show that
\begin{align*}
	\G_{n,m} (\nabla_\theta\loss_{\spe, \vec{\sigma}_n}
	- \nabla_\theta\loss_{u\opt, \vec{\sigma}\opt})
	=
	\frac{1}{m}
	\sum_{j = 1}^m
	\G_n^{(j)} (\nabla_\theta\loss_{\spe, \sigma_{n,j}}
	- \nabla_\theta \loss_{u\opt, \sigma\opt_j})
	\cp 0,
\end{align*}
where $\G_n^{(j)}$ denotes the empirical process on
$(X_i, Y_{ij})_{i=  1}^n$. Notably, because $m$ is finite, it is sufficient
to show that
\begin{equation*}
	\G_n^{(j)} (\nabla_\theta\loss_{\spe, \sigma_{n,j}}
	- \nabla_\theta \loss_{u\opt, \sigma\opt_j})
	\cp 0,
	~~~
	j = 1, \ldots, m.
\end{equation*}
To that end, we suppress dependence on $j$ for notational
simplicity and simply write $\G_n$ and $\loss_{\spe, \sigma_n}$, where
$\LtwoP{\sigma_n - \sigma\opt} \cp 0$.
For any Donsker class $\mc{F} \subset \mc{X} \to \R^d$ and
$\epsilon > 0$,
we have
\begin{equation*}
	\limsup_{\delta \downarrow 0}
	\limsup_{n \to \infty}
	\P\left(\sup_{\LtwoP{f - g} \le \delta}
	\G_n(f - g) \ge \epsilon\right) = 0,
\end{equation*}
(see~\cite[Thm.~3.7.31]{GineNi21}), and so in turn it is sufficient to prove
that
\begin{equation}
	\label{eqn:ltwo-p-convergence-sp}
	\LtwoPbold{\nabla \loss_{\spe, \sigma_n} - \nabla\loss_{u\opt, \sigma\opt}}
	\cp 0.
\end{equation}

To demonstrate the convergence~\eqref{eqn:ltwo-p-convergence-sp}, let $M$ be
finite, and note that for any fixed $\theta, \sigma \in \Flink$ that for $V
= YX$ we have
\begin{align*}
	\LtwoPbold{\nabla \loss_{\theta, \sigma}
		- \nabla \loss_{u\opt, \sigma\opt}}^2
	& = \Ep \brk{\ltwo{V\sigma(-\<V, \theta\>) -
			V\sigma\opt(-\<V,  u\opt\>)}^2} \\
	& \leq \Ep \brk{\ltwo{V}^2 \indic{\ltwo{V} \geq M}}
	+ M^2 \LtwoPbold{\sigma(-\<V, \theta\>) - \sigma\opt(-\<V, u\opt\>)}^2 \\
	& \leq \Ep \brk{\ltwo{X}^2 \indic{\ltwo{X} \geq M}}
	+ 2M^2 \LtwoPbold{\sigma(-\<V, u\opt\>) - \sigma\opt(-\<V, u\opt\>)}^2 \\
	& \qquad\qquad\qquad\qquad\qquad\qquad ~
	+ 2M^2 \LtwoPbold{\sigma(-\<V, u\opt\>) - \sigma(-\<V, \theta\>)}^2
\end{align*}
by the triangle inequality.
As 
\begin{align*}
	\LtwoPbold{\sigma(-\<V, u\opt\>) - \sigma(-\<V, \theta\>)}
	\leq \lipconst \ltwo{\theta - u\opt} \cdot \LtwoPbold{V}
	= \lipconst \ltwo{\theta - u\opt} \cdot \LtwoPbold{X}
\end{align*}
and $\spe - u\opt \cp 0$ and
\begin{align*}
	\LtwoPbold{\sigma_n(-\<V, u\opt\>) - \sigma\opt(-\<V, u\opt\>)}
	= \LtwoPbold{\sigma_n - \sigma\opt} \cp 0
\end{align*}
by assumption,
it follows that for any $\epsilon > 0$ that
\begin{align*}
	\P\left(\LtwoPbold{\nabla_{\theta} \loss_{\spe, \sigma_n}
		- \nabla_{\theta} \loss_{u\opt, \sigma\opt}}
	\ge \E[\ltwo{X}^2 \indic{\ltwo{X} \geq M}] + \epsilon
	\right) \to 0.
\end{align*}
Taking $M \uparrow \infty$ gives the
convergence~\eqref{eqn:ltwo-p-convergence-sp}, completing the proof.

%
%

\section{Proofs for semiparametric approaches}

\subsection{Proof of Lemma~\ref{lem:d-continuity-guarantees}}
\label{proof:d-continuity-guarantees}

As $\sigma$ and $\sigma\opt$ are $\lipconst$-Lipschitz, we may without loss
of generality assume that $\lipconst = 1$ and so $\linf{\sigma_j'} \le 1$
and $\linf{{\sigma\opt_j}'} \le 1$.  We can compute the Hessian at any
$\theta \in \R^d$ and $\vec{\sigma} = (\sigma_1, \dots, \sigma_m)$,
\begin{align*}
  \nabla^2 L(\theta, \vec{\sigma}) & = \E\left[\frac{1}{m}
    \sum_{j = 1}^m \left(\sigma\opt_j(-\<X, u\opt\>) \sigma_j'(\<\theta, X\>)
    + \sigma\opt_j(\<X, u\opt\>) \sigma_j'(-\<\theta, X\>)\right) X X^\top \right] \\
  & = \frac{1}{m} \sum_{j = 1}^m \E\left[\sigma_j'(-Y\<X, \theta\>) XX^\top\right],
\end{align*}
where in the last line we use that $\sigma_j\opt$ and $\sigma_j$ are symmetric for $j=1,\dots, m$. Therefore we can upper bound the distance between Hessians by
\begin{align*}
	\norm{\nabla^2 L(\theta, \vec{\sigma}) - \nabla^2 L(u\opt, \vec{\sigma}\opt)} & \leq \frac{1}{m} \sum_{j = 1}^m \underbrace{\norm{\E\left[\sigma_j'(-Y\<X, \theta\>) XX^\top\right] - \E\left[{\sigma_j\opt}'(-Y\<X, u\opt\>) XX^\top\right]  }}_{:=\delta_j},
\end{align*}
and thus we only need to prove each quantity $\delta_j \to 0$ if
$d_{\fsymlinksp}((\theta, \vec{\sigma}), (u\opt, \vec{\sigma}\opt)) \to 0$,
that is, if $\ltwo{\theta - u\opt} \to 0$ and $\LtwoP{\sigma_j(-Y\<X,
  u\opt\>) - \sigma_j\opt(-Y\<X, u\opt\>)} \to 0$. In the following, we will
show $\delta_j \to 0$ under the two different conditions. To further
simplify the quantity, we claim it is sufficient to show $\xi_j \defeq
\norm{\E[\sigma_j'(-Y\<X, \theta\>) XX^\top] -
  \E[\sigma_j'(-Y\<X, u\opt\>) XX^\top]} \to 0$. Indeed, we have
\begin{lemma}
  If $\xi_j \to 0$ and $\LtwoP{\sigma_j(-Y\<X, u\opt\>) -
    \sigma_j\opt(-Y\<X, u\opt\>)} \to 0$, then $\delta_j \to 0$.
\end{lemma}
\begin{proof}
  By the triangle inequality and the independent decomposition $X = Z u\opt
  + W$, we have
  \begin{align*}
    \delta_j & \leq \xi_j + \norm{\E\left[\sigma_j'(-Y\<X, u\opt\>) XX^\top\right] - \E\left[{\sigma_j\opt}'(-Y\<X, u\opt\>) XX^\top\right] } \nonumber \\
    & = \xi_j + \norm{\E\left[(\sigma_j'(Z)-{\sigma_j\opt}'(Z)) Z^2\right] \cdot u\opt {u\opt}^\top } = \xi_j + \left|\E\left[(\sigma_j'(Z)-{\sigma_j\opt}'(Z)) Z^2\right] \right|.
  \end{align*}
  It remains to show $\E\left[(\sigma_j'(Z)-{\sigma_j\opt}'(Z)) Z^2\right]
  \to 0$. Using the symmetry of $\sigma_j$ and $\sigma\opt_j$, so
  $\sigma_j'(t) = \sigma_j'(-t)$, we can replace $Z$ by $|Z|$. Then
  integrating by parts, for any $0 < \epsilon < M < \infty$ we have
  \begin{align*}
    \Ep \brk{\sigma_j'(Z) Z^2 \ind \{\epsilon \leq |Z| \leq M\}} & = \int_\epsilon^M \sigma_j'(z) z^2 p(z) dz \nonumber \\
    &= \sigma_j(M) M^2p(M) - \sigma_j(\epsilon) \epsilon^2 p(\epsilon)  - \int_\epsilon^M \sigma_j(z) (2z p(z) + z^2 p'(z)) dz.
  \end{align*}
  By our w.l.o.g.\ assumption that $\linf{\sigma'} \leq 1$, we have $|\Ep
  [\sigma_j'(Z) Z^2] - \Ep [\sigma_j'(Z) Z^2 \ind \{\epsilon \leq |Z| \leq
    M\}]| \leq \Ep[Z^2 \ind\{|Z| < \epsilon \text{ or } |Z| >
    M\}]$. Thus, recognizing the trivial
  bound $\linf{\sigma_j} \leq 1$, we have
  \begin{align*}
    \left|\E\left[(\sigma_j'(Z)-{\sigma_j\opt}'(Z)) Z^2\right] \right| &\leq 2 \Ep[Z^2 \ind\{|Z| < \epsilon \text{ or } |Z| > M\}] + 2 \prn{\epsilon^2 p(\epsilon) + M^2p(M)} + \nonumber \\
    & \qquad + \underbrace{\left|\int_\epsilon^M \sigma_j(z) (2z p(z) + z^2 p'(z)) -  \int_\epsilon^M \sigma_j\opt(z) (2z p(z) + z^2 p'(z))\right|}_{(\star)}.
  \end{align*}
 	
  We show for any fixed $0 < \epsilon < M < \infty$, $(\star) \to 0$. Applying
  the Cauchy-Schwarz inequality twice, we have the bounds
  \begin{align*}
    \left|\int_\epsilon^M (\sigma_j(z) - \sigma_j\opt(z)) z p(z) dz \right| &\leq \LtwoP{\sigma_j(Z) - \sigma_j\opt(Z)} \cdot \sqrt{\Ep[Z^2]} \to 0, \\
    \left|\int_\epsilon^M (\sigma_j(z) - \sigma_j\opt(z)) z^2 p'(z) dz \right| & \leq  \LtwoP{\sigma_j(Z) - \sigma_j\opt(Z)} \cdot \sqrt{\int_\epsilon^M z^4 \prn{\frac{p'(z)}{p(z)}}^2 p(z) dz} \nonumber \\
    & \leq \LtwoP{\sigma_j(Z) - \sigma_j\opt(Z)} \cdot \sup_{z \in [\epsilon, M]} \left|\frac{p'(z)}{p(z)}\right| \sqrt{\Ep[Z^4]} \to 0,
  \end{align*}
  where for the final inequality we use that $p(z)$ is nonzero and
  continuously differentiable.

  As $(\star) \to 0$, we evidently have
  $\limsup |\E[(\sigma_j'(Z) - {\sigma_j\opt}'(Z)) Z^2]|
  \le 2 \E[Z^2 (\ind\{|Z| < \epsilon\} + \ind\{|Z| > M\})]
  + 2 (\epsilon^2 p(\epsilon) + M^2 p(M))$ for arbitrary
  $0 < \epsilon < M < \infty$.
  Using the assumptions that $\Ep[Z^2] \leq \Ep[\ltwo{X}^2] < \infty$
  and $\lim_{z \to s} z^2p(z) = 0$ for $s \in \{0, \infty\}$, we
  conclude the proof by taking $\epsilon \to 0$ and $M \to \infty$.
\end{proof}

Finally we prove $\xi_j \defeq \norm{\E[\sigma_j'(-Y\<X, \theta\>) XX^\top]
  - \E[\sigma_j'(-Y\<X, u\opt\>) XX^\top]} \to 0$ under the two conditions
in the statement of Lemma~\ref{lem:d-continuity-guarantees}: that
$\sigma'_j$ are Lipschitz or $X$ has a continuous density.
\paragraph{Condition 1. The links have Lipschitz derivatives}
We apply Jensen's inequality to write
\begin{align*}
  \xi_j \leq \E\left[|\sigma_j'(-Y\<X, \theta\>) - \sigma_j'(-Y\<X, u\opt\>)| \cdot \norm{XX^\top}\right]
  & = \E\left[|\sigma_j'(-Y\<X, \theta\>) - \sigma_j'(-Y\<X, u\opt\>)| \cdot \ltwo{X}^2\right] \\
  & \le \norm{\sigma_j'}_{\textup{Lip}}
  \ltwo{\theta - u\opt} \cdot \E\left[\ltwo{X}^3\right],
\end{align*}
where $\norm{\cdot}_{\textup{Lip}}$ denotes the Lipschitz
constant of its argument. Taking
$\theta \to u\opt$ completes the proof for this case.
\paragraph{Condition 2. The covariates have continuous density}
Let $X$ have density $q(x)$. We rewrite the convergence
$\theta \to u\opt$ instead as $\theta
= V u\opt$ where $V \to I_d$ is invertible. We again divide the expectation
into large $\ltwo{X}$ part and small $\ltwo{X}$ part. Let
$M < \infty$ be large enough that $\E[\ltwo{X}^2 \ind\{\ltwo{X} > M\}]
\le \epsilon$. Then using $\linfs{\sigma_j'} < \infty$,
we obtain
\begin{align*}
  \xi_j & \leq \norm{\E\left[(\sigma_j'(-Y\<X, \theta\>) -{\sigma_j}'(-Y\<X, u\opt\>)) XX^\top \right] \ind \{\ltwo{X} \leq M\}} + 2 \lipconst'
  \Ep \brk{\ltwo{X}^2 \ind \{\ltwo{X} > M\}} \nonumber \\
  & \le \norm{ \prn{V^{-\top} \E\left[\sigma_j'(-\<V^\top X, u\opt\>) V^\top X XV \right] V - \E\left[{\sigma_j}'(-\<X, u\opt\>)XX^\top \right]}  \ind \{\ltwo{X} \leq M\}}  + 2 \linf{\sigma_j'} \epsilon.
\end{align*}
By the linear transformation of variables $X' \defeq V^\top X$, the first
term in the above display is
\begin{align*}
	&\norm{ V^{-\top} \E\left[\sigma_j'(-Y\<V^\top X, u\opt\>) V^\top X XV \right] V - \E\left[{\sigma_j}'(-Y\<X, u\opt\>)XX^\top \right]} \nonumber \\
	&= \norm{\det(V^{-1}) \cdot V^{-\top} \prn{\int_{\ltwo{V^\top x} \leq M} \sigma_j'(-x^\top u\opt) xx^\top p(V^{-\top} x) dx}  V - \int_{\ltwo{x} \leq M} \sigma_j'(-x^\top u\opt) xx^\top p(x) dx}.
\end{align*}
Since $p(x)$ is absolutely continuous on any compact set, the above term
converges to $0$ as $V \to I_d$. Finally, as $\epsilon > 0$ was arbitrary,
we take $M \to \infty$ to conclude that $\xi_j \to 0$.

\subsection{Proof of Proposition~\ref{prop:crowd-sourcing-estimator}} \label{proof:crowd-sourcing-estimator}

We apply Theorem~\ref{theorem:semiparametric-master}.
We first give the specialization
of $C_{m,\vec{\sigma}\opt}$ that the assumptions of the proposition
imply, recognizing that as
${\sigma\lr}' = \sigma\lr(1 - \sigma\lr)$, we have
\begin{align*}
  C_{m, \vec{\sigma}\opt} & =
  \frac{\sum_{j = 1}^m
    \E[\sigma\lr(\alpha_j\opt Z) (1 - \sigma\lr(\alpha_j\opt Z))]}{
    (\sum_{j = 1}^m\E[{\sigma\lr}'(\alpha_j\opt Z)])^2}
  = \prn{ \sum_{j = 1}^m \E[\sigma\lr(\alpha_j\opt Z) (1 - \sigma\lr(\alpha_j\opt Z))]}^{-1}.
\end{align*} 
We now argue that we can actually invoke
Theorem~\ref{theorem:semiparametric-master}, which requires verification of
Assumption~\ref{assumption:sp-general}. Because for any $M < \infty$, the
link functions $t \mapsto \sigma\lr(\alpha t)$ have Lipschitz continuous
derivatives for $|\alpha| \le M$, so when $\vec{\sigma}_n$ has form
$\vec{\sigma}_n = [\sigma\lr(\alpha_{n,j}\cdot)]_{j = 1}^m$,
Lemma~\ref{lem:d-continuity-guarantees} implies the continuity of the
mapping $(\theta, \vec{\sigma}) \mapsto \nabla_\theta^2 \poploss(\theta,
\vec{\sigma})$ for $d_{\fsymlinksp}$ at $(u\opt, \vec{\sigma}\opt)$.
Then recognizing that by
Lipschitz continuity of $\sigma\lr$ we have
\begin{equation*}
  \LtwoPbold{\vec{\sigma}_n(Z)
    - \vec{\sigma}\opt(Z)}
  \le \ltwo{\alpha_n - \alpha\opt}^2 \LtwoPbold{Z}^2
  \cp 0
\end{equation*}
whenever $\alpha_n \in \R^m$ satisfies $\alpha_n \cp \alpha\opt$, we
obtain the proposition.

\section{Proofs of nonparametric convergence results}
\label{sec:proof-nonparametric}

\newcommand{\noise}{\xi}
\newcommand{\localcomplexity}{\mc{R}}

In this technical appendix, we include proofs of the results from
Section~\ref{sec:semiparametric} as well as a few additional results, which
are essentially corollaries of results on localized complexities and
nonparametric regression models, though we require a few modifications
because our setting is slightly non-standard.

\subsection{Preliminary results}

To set notation and to make reading it self-contained, we provide some
definitions. The $L^r(P)$ norm of a function or random vector is
$\norm{f}_{L^r(P)} = (\int |f|^r dP)^{1/r}$, so that its $L^2(P_n)$-norm is
$\LtwoPn{f}^2 = \frac{1}{n} \sum_{i = 1}^n f(X_i)^2$.  We consider the
following abstract nonparametric regression setting: we have a function
class $\mc{F} \subset \{\R \to \R\}$ with $f\opt \in \mc{F}$, and our
observations follow the model
\begin{equation}
  \label{eqn:nonparametric-model}
  Y_i = f\opt(X_i) + \noise_i,
\end{equation}
but instead of observing $(X_i, Y_i)$ pairs we observe $(\wt{X}_i, Y_i)$
pairs, where $\wt{X}_i$ may not be identical to $X_i$ (these play the roll
of $\<u\opt, X_i\>$ versus $\<\uinit, X_i\>$ in the results to come).  We
assume that $\noise_i$ are bounded so that $\sup \noise - \inf \noise \le
1$, independent, and satisfy the conditional mean-zero property that
$\E[\noise_i \mid X_i] = 0$ (though $\noise_i$ may not be independent of
$X_i$).  For a (thus far unspecified) function class $\mc{F}$ we set
\begin{equation*}
  \what{f} = \argmin_{f \in \mc{F}} P_n(Y - f(\wt{X}))^2.
\end{equation*}

We now demonstrate that the error $\LtwoPns{\what{f} - f\opt}^2 =
\frac{1}{n} \sum_{i = 1}^n (\what{f}(X_i) - f\opt(X_i))^2$ can be bounded by
a combination of the errors $\wt{X}_i - X_i$ and local complexities of the
function class $\mc{F}$.  Our starting point is a local complexity bound
analagous to the localization results available for in-sample prediction
error in nonparametric regression~\cite[cf.][Thm.~13.5]{Wainwright19}.  To
present the results, for a function class $\mc{H}$ we define the
localized $\noise$-complexity
\begin{equation*}
  \localcomplexity_n(u; \mc{H})
  \defeq \E\left[ \sup_{h \in \mc{H},
      \LtwoPns{h} \le u}
    \left|n^{-1} \sum_{i = 1}^n \noise_i h(x_i) \right| \right],
\end{equation*}
where we treat the expectation conditionally on $X_i$ and $\noise_i$ are
random (i.e.\ $\noise_i = Y_i - f\opt(X_i)$).  For the
model~\eqref{eqn:nonparametric-model}, we define the centered class
$\mc{F}\opt = \{f - f\opt \mid f \in \mc{F}\}$, which is star-shaped as
$\mc{F}$ is a convex set.\footnote{ A set $\mc{H}$ is \emph{star-shaped} if
  for all $h \in \mc{H}$, if $\alpha \in [0, 1]$ then $\alpha h \in
  \mc{H}$.}  We say that $\delta$ satisfies the \emph{critical radius
  inequality} if
\begin{equation}
  \frac{1}{\delta} \localcomplexity_n(\delta; \mc{F}\opt)
  \le \delta.
  \label{eqn:critical-inequality}
\end{equation}
With this, we can provide a proposition giving a high-probability
bound on the in-sample prediction
error of the empirical estimator $\what{f}$, which is essentially
identical to~\cite[Thm.~13.5]{Wainwright19}, though we require a few
modifications to address that we observe $\wt{X}_i$ and not $X_i$
and that the noise $\noise_i$ are bounded but not Gaussian.
\begin{proposition}
  \label{proposition:in-sample-convergence}
  Let $\mc{F}$ be a convex function class, $\delta_n > 0$ satisfy the
  critical inequality~\eqref{eqn:critical-inequality},
  and let
  $\gamma^2 = \sup_{f \in \mc{F}}
  \frac{1}{n} \sum_{i = 1}^n (f(X_i) - f(\wt{X}_i))^2$.
  Then for $t \ge \delta_n$,
  \begin{equation*}
    \P\left(\LtwoPns{\what{f} - f\opt}^2 \ge 30 t \delta_n
    + 25 \gamma \max\{\gamma, \LtwoPns{\noise}\}
    \mid X_1^n, \wt{X}_1^n
    \right)
    \le \exp\left(-\frac{n t \delta_n}{4}\right).
  \end{equation*}
\end{proposition}
\noindent
See Section~\ref{sec:proof-in-sample-convergence} for a proof of the
proposition.

Revisiting the critical radius inequality~\eqref{eqn:critical-inequality},
we can also apply \cite[Corollary~13.7]{Wainwright19}, which allows
us to use an entropy integral to guarantee bounds on the critical
radius. Here, we again fix $X_1^n = x_1^n$, and for a function
class $\mc{H}$ we let
$\mc{B}_n(\delta; \mc{H}) = \{h \in \starhull(\mc{H})
\mid \LtwoPn{h} \le \delta\}$. Let $N_n(t; \mc{B})$ denote the
$t$-covering number of $\mc{B}$ in $\LtwoPn{n\cdot}$-norm. Then
modifying a few numerical constants, we have
\begin{corollary}[Wainwright~\cite{Wainwright19}, Corollary 13.7]
  \label{corollary:localization-via-dudley}
  Let the conditions of
  Proposition~\ref{proposition:in-sample-convergence} hold. Then
  for a numerical constant $C \le 16$,
  any $\delta \in [0, 1]$ satisfying
  \begin{equation*}
    \frac{C}{\sqrt{n}} \int_{\delta^2 / 4}^\delta
    \sqrt{\log N_n(t; \mc{B}_n(\delta, \mc{F}\opt))} dt
    \le \frac{\delta^2}{2}
  \end{equation*}
  satisfies the critical inequality~\eqref{eqn:critical-inequality}.
\end{corollary}
\noindent
As an immediate consequence of this inequality, we have the following:
\begin{corollary}
  \label{corollary:lipschitz-criticality}
  Assume $|x_i| \le b$ for all $i \in [n]$ and that
  $\mc{F}$ is contained in the class of $M$-Lipschitz functions
  with $f(0) = 0$ (or any other fixed constant).
  Then for a numerical
  constant $c < \infty$, the choice
  \begin{equation*}
    \delta_n = c \left(\frac{M b}{n}\right)^{1/3}
  \end{equation*}
  satisfies the critical inequality~\eqref{eqn:critical-inequality}.
\end{corollary}
\begin{proof}
  The covering numbers $N_\infty$
  for the class $\mc{F}$ of $M$-Lipschitz functions
  on $[-b, b]$
  satisfying $f(0) = 0$
  in supremum norm $\linf{f} = \sup_{x \in [-b, b]} |f(x)|$ satisfy
  $\log N_\infty(t; \mc{F}) \lesssim \frac{M b}{t}$
  \cite[cf.][Example 5.10]{Wainwright19}. Using that $N_n \le N_\infty$,
  we thus have
  \begin{equation*}
    \int_{\delta^2/4}^\delta
    \sqrt{\log N_n(t; \mc{B}_n(\delta, \mc{F}\opt))}
    \lesssim
    \int_{\delta^2/4}^\delta
    \sqrt{\frac{Mb}{t}} dt
    = 2 \sqrt{Mb} \left(\sqrt{\delta} - \delta/4\right)
    \le 2 \sqrt{Mb \delta}
  \end{equation*}
  whenever $\delta \le 1$.
  For a numerical constant $c > 0$ it suffices in
  Corollary~\ref{corollary:localization-via-dudley} to choose
  $\delta$ satisfying
  $c \frac{1}{\sqrt{n}} \sqrt{Mb \delta} \le \delta^2$, or
  $\delta = c (\frac{Mb}{n})^{1/3}$.
\end{proof}

\subsection{Proof of Proposition~\ref{proposition:consistency-of-sigma}}
\label{sec:proof-consistency-of-sigma}

We assume without loss of generality that $m = 1$, as nothing in the
proof changes except notationally (as we assume $m$ is fixed).
We apply Proposition~\ref{proposition:in-sample-convergence} and
Corollary~\ref{corollary:lipschitz-criticality}.  For
notational simplicity, let $Q$ denote the measure
on $\R$ that $Y \<u\opt, X\>$ induces for $X \sim P$, and $Q_n$ similarly for
$P_n$.
We first show that
$\LtwoQns{\sigma_n - \sigma\subopt}$ converges.  First, we
recall~\cite[Lemma 3]{Owen90} that $\max_{i \le n} n^{-1/k} \ltwo{X_i} \cas
0$ as $n \to \infty$. Thus there is a (random) $B < \infty$ such that
$\max_{i \le n} \ltwo{X_i} \le B n^{1/k}$ for all $n$. Therefore,
Corollary~\ref{corollary:lipschitz-criticality} implies that the choice
$\delta_n = c (\frac{\lipconst B}{n^{1-1/k}})^{1/3}$ satisfies the critical
inequality~\eqref{eqn:critical-inequality}, and taking $\gamma^2_n =
\frac{\lipconst^2}{n} \sum_{i = 1}^n \<\uinit - u\opt, X_i\>^2 \le M^2 \uerror_n^2
\LtwoPn{X}^2$, we have that for $t \ge \delta_n$,
\begin{equation*}
  \LtwoPns{\sigma_n - \sigma\subopt}^2
  \lesssim t \delta_n + \lipconst^2 \uerror_n^2 \LtwoPn{X}^2
  ~~~ \mbox{with~probability~at~least}~
  1 - \exp\left(-\frac{n t \delta_n}{4}\right)
\end{equation*}
on the event that $\max_{i \le n} \ltwo{X_i} \le B n^{1/k}$ for all $n$,
where we have conditioned (in the probability) on $X_i$.
As $\LtwoPn{X}^2 \cas \E[\ltwo{X}^2]$, we may choose
$t = \delta_n \gg 1/n^{1/3}$ and find that the Borell-Cantelli lemma
then implies that with probability $1$, there is a random $C < \infty$
such that
\begin{equation}
  \label{eqn:ltwopn-sigma-converge}
  \LtwoQn{\sigma_n - \sigma\opt}^2 \le C \left(
  n^{\frac{2}{3k} - \frac{2}{3}} + \epsilon_n^2 \right)
\end{equation}
for all $n$ with probability $1$.

Finally we argue that $\LtwoQ{\sigma_n - \sigma\opt} \cas 0$.
Let $b < \infty$ be otherwise arbitrary, and let
$\mc{G}_b$ be the collection of $2\lipconst$-Lipschitz functions on
$[-b, b]$ with $\linf{g} \le 1$ and $g(0)$ for $g \in \mc{G}_b$, noting
that $\sigma_n - \sigma\opt$ restricted to $[-b, b]$ evidently
belongs to $\mc{G}_b$. Then we have
\begin{align}
  \lefteqn{\LtwoQ{\sigma_n - \sigma\subopt}^2} \nonumber \\
  & = \LtwoQn{\sigma_n - \sigma\subopt}^2
  + \int_{|t| \le b} (\sigma_n(t) - \sigma\subopt(t))^2
  (dQ(t) - dQ_n(t))
  + \int_{|t| > b} (\sigma_n(t) - \sigma\subopt(t))^2 
  (dQ(t) - dQ_n(t)) \nonumber \\
  & \le \LtwoQn{\sigma_n - \sigma\subopt}^2
  + \sup_{g \in \mc{G}_b} |Q g^2 - Q_n g^2|
  + Q([-b, b]^c) + Q_n([-b, b]^c),
  \label{eqn:three-terms-bike}
\end{align}
where the inequality used that $\linf{\sigma_n - \sigma\subopt} \le 1$
by construction.
The first term in inequality~\eqref{eqn:three-terms-bike} we have
already controlled. We may control the second supremum term
almost immediately using Dudley's entropy integral
and a Rademacher contraction inequality.
Indeed, we have
\begin{equation*}
  \E[\sup_{g \in \mc{G}_b} |Q_n g^2 - Q g^2|]
  \stackrel{(i)}{\le} 2 \E[\sup_{g \in \mc{G}_b} |Q_n^0 g^2|]
  \stackrel{(ii)}{\le} 4 \E[\sup_{g \in \mc{G}_b} |Q_n^0|]
  \stackrel{(iii)}{\lesssim} \frac{1}{\sqrt{n}} \int_0^\infty
  \sqrt{\log N_\infty(t; \mc{G}_b)} dt,
\end{equation*}
where inequality~$(i)$ is a standard symmetrization inequality, $(ii)$ is
the Rademacher contraction inequality~\cite[Ch.~4]{LedouxTa91} applied to
the function $t \mapsto t^2$, which is $2$ Lipschitz for $t \in [-1, 1]$,
and $(iii)$ is Dudley's entropy integral bound. As the $\sup$-norm
log-covering numbers of $\lipconst$-Lipschitz functions on $[-b, b]$ scale
as $\frac{\lipconst b}{t}$ for $t \le \lipconst b$ and are 0 otherwise, we
obtain $\E[\sup_{g \in \mc{G}_b} |Q_n g^2 - Q g^2|] \lesssim
\frac{\lipconst b}{\sqrt{n}}$.  The bounded-differences inequality then implies that
for any $t > 0$,
\begin{equation*}
  \P\left(\sup_{g \in \mc{G}_b} |(Q_n - Q)g^2|
  \ge c \frac{\lipconst b}{\sqrt{n}} + t \right) \le
  \P\left(\sup_{g \in \mc{G}_b} |(Q_n - Q) g^2|
  \ge \E[\sup_{g \in \mc{G}_b} |(Q_n - Q) g^2|] + t \right)
  \le \exp(-c n t^2).
\end{equation*}
Finally, the final term in the bound~\eqref{eqn:three-terms-bike}
evidently satisfies $Q([-b, b]^c) \le \frac{\E[\ltwo{X}^k]}{b^k}$
and $\sup_b |Q_n([-b,b]^c) - Q([-b, b])| \le 2 \sqrt{t / n}$
with probability at least $1 - e^{-2t^2}$ by the DKW inequality.
We thus find by the Borel-Cantelli lemma that simultaneously
for all $b < \infty$, with probability at least
$1 - e^{-nt^2}$ we have
\begin{equation*}
  \LtwoQ{\sigma_n - \sigma\subopt}^2 \le
  \LtwoQn{\sigma_n - \sigma\subopt}^2
  + \frac{C \lipconst b}{\sqrt{n}} + C t + \frac{\E[\ltwo{X}^k]}{b^k},
\end{equation*}
where $C$ is a numerical constant.

Substituting inequality~\eqref{eqn:ltwopn-sigma-converge} into
the preceding display and taking $b = n^{-\frac{1}{2(k + 1)}}$, we
get the result.

\subsection{Proof of Proposition~\ref{proposition:in-sample-convergence}}
\label{sec:proof-in-sample-convergence}

We begin with an extension of the
familiar basic inequality~\cite[e.g.][Eq.~(13.18)]{Wainwright19}, where
we see by convexity that for any $\eta > 0$ we have
\begin{align*}
  \sum_{i = 1}^n (Y_i - \what{f}(X_i))^2
  & \le (1 + \eta) \sum_{i = 1}^n (Y_i - \what{f}(\wt{X}_i))^2
  + (1 + 1/\eta) \sum_{i = 1}^n (\what{f}(\wt{X}_i) - \what{f}(X_i))^2 \\
  & \le (1 + \eta) \sum_{i = 1}^n (Y_i - f\opt(\wt{X}_i))^2
  + (1 + 1/\eta) \sum_{i = 1}^n (\what{f}(\wt{X}_i) - \what{f}(X_i))^2 \\
  & \le (1 + \eta)^2 \sum_{i = 1}^n (Y_i - f\opt(X_i))^2
  + (2 + \eta + 1/\eta) \sum_{i = 1}^n
  \left[(\what{f}(\wt{X}_i) - \what{f}(X_i))^2
    + (f\opt(\wt{X}_i) - f\opt(X_i))^2 \right].
\end{align*}
Noting that $Y_i = f\opt(X_i) + \noise_i$, algebraic manipulations yield
that if
$\Delta = [f\opt(X_i) - \what{f}(X_i)]_{i = 1}^n$ is the error vector, then
\begin{align*}
  \LtwoPn{\Delta}^2 - \frac{2}{n} \noise^T \Delta
  + \LtwoPn{\noise}^2
  &\le (1 + \eta)^2 \LtwoPn{\noise}^2 \\
  &\quad + \frac{2 + \eta + 1/\eta}{n} \sum_{i = 1}^n
  \left[(\what{f}(\wt{X}_i) - \what{f}(X_i))^2 \right. \\
  &\hspace{5.8cm}\left. + (f\opt(\wt{X}_i) - f\opt(X_i))^2 \right].
\end{align*}
Simplifying, we obtain the following:
\begin{lemma}
  \label{lemma:basic-inequality}
  Let $\gamma^2 = \sup_{f \in \mc{F}}
  \frac{1}{n} \sum_{i = 1}^n (f(X_i) - f(\wt{X}_i))^2$
  and $\Delta = [\what{f}(X_i) - f\opt(X_i)]_{i = 1}^n$.
  Then
  \begin{align*}
    \LtwoPn{\Delta}^2
    & \le \inf_\eta \left\{(2 \eta + \eta^2) \LtwoPn{\noise}^2
    + \frac{2}{n} \noise^T \Delta
    + (4 + 2 \eta + 2/\eta) \gamma^2 \right\} \\
    & \le \frac{2}{n} \noise^T \Delta
    + 11 \gamma \max\left\{\gamma, \LtwoPn{\noise}\right\}.
  \end{align*}
\end{lemma}
\begin{proof}
  The first inequality is algebraic manipulations and
  uses that our choice of $\eta$ was arbitrary.
  For the second, we consider two cases:
  that $\LtwoPn{\noise} \ge \gamma$ and that $\LtwoPn{\noise} \le \gamma$.
  In the first case, we consider $\eta \le 1$,
  yielding the simplified bound
  that
  $\LtwoPn{\Delta}^2 \le \frac{2}{n} \noise^T \Delta
  + 3 \eta \LtwoPn{\noise}^2 + \frac{8 \gamma^2}{\eta}$ for
  $\eta \le 1$. Taking $\eta = \gamma / \LtwoPn{\noise}$ gives
  that
  \begin{equation*}
    \LtwoPn{\Delta}^2
    \le \frac{2}{n} \noise^T \Delta
    + 11 \LtwoPn{\noise} \gamma.
  \end{equation*}
  In the case that $\gamma \ge \LtwoPn{\noise}$, we choose $\eta = 1$,
  and the bound simplifies to
  $\LtwoPn{\Delta}^2 \le \frac{2}{n} \noise^T \Delta
  + 3 \LtwoPn{\noise}^2 + 8 \gamma^2
  \le \frac{2}{n} \noise^T \Delta + 11 \gamma^2$.
\end{proof}

We now return to the proof of the proposition proper.
We begin with an essentially immediate extension of the result~\cite[Lemma
  13.12]{Wainwright19}.  We let $\mc{H}$ be an arbitrary star-shaped
function class.
Define the event
\begin{equation}
  \label{eqn:A-wainwright-event}
  A(u) \defeq \left\{
  \mbox{there~exists~} g \in \mc{H} ~ \mbox{s.t.}~
  \LtwoPn{g} \ge u ~ \mbox{and} ~
  \left|P_n \noise g\right| \ge 2 \LtwoPn{g} u
  \right\},
\end{equation}
which we treat conditionally on $X_1^n = x_1^n$ as in the definition
of the local complexity $\localcomplexity_n$.
(Here the noise $\noise_i$ are still random, taken conditionally
on $X_1^n = x_1^n$.)

\begin{lemma}[Modification of Lemma 13.12 of~\citet{Wainwright19}]
  \label{lemma:star-shaped-localization}
  Let $\mc{H}$ be a start-shaped function class
  and let $\delta_n > 0$ satisfy the critical radius inequality
  \begin{equation*}
    \frac{1}{\delta} \localcomplexity_n(\delta; \mc{H}) \le
    \delta.
  \end{equation*}
  Then for all $u \ge \delta_n$, we have
  \begin{equation*}
    \P(A(u)) \le \exp\left(-\frac{n u^2}{4} \right).
  \end{equation*}
\end{lemma}
\noindent
Deferring the proof of Lemma~\ref{lemma:star-shaped-localization}
to Section~\ref{sec:proof-star-shaped-localization},
we can parallel the argument for~\cite[Thm.~13.5]{Wainwright19}
to obtain our proposition.

Let $\mc{H} = \mc{F}\opt$ in Lemma~\ref{lemma:star-shaped-localization}.
Whenever $t \ge \delta_n$, we have
$\P(A(\sqrt{t \delta_n})) \le e^{-nt \delta_n / 4}$. We consider the two
cases that $\LtwoPns{\Delta} \lessgtr \sqrt{t \delta_n}$. In the
former that $\LtwoPns{\Delta} \le \sqrt{t \delta_n}$, we have
nothing to do. In the latter, we have $\what{f} - f\opt \in \mc{F}\opt$
while $\LtwoPns{\Delta} > \sqrt{t \delta_n}$, so that if
$A(\sqrt{t \delta_n})$ fails then we must have
\begin{equation*}
  \left|\frac{1}{n} \sum_{i = 1}^n \noise_i \Delta(x_i)\right|
  \le 2 \LtwoPn{\Delta} \sqrt{t \delta_n}.
\end{equation*}
From the extension of the basic inequality in
Lemma~\ref{lemma:basic-inequality} we see that
\begin{equation*}
  \LtwoPn{\Delta}^2 \le 
  4 \LtwoPn{\Delta} \sqrt{t \delta_n} +
  11 \max\left\{\gamma^2, \gamma \LtwoPn{\noise}\right\}.
\end{equation*}
Solving for $\LtwoPn{\Delta}$ then yields
\begin{equation*}
  \LtwoPn{\Delta}
  \le \frac{4 \sqrt{t \delta_n} + \sqrt{16 t \delta_n
      + 44 \gamma \max\{\gamma, \LtwoPn{\noise}\}}}{2}
  \le 4 \sqrt{t \delta_n} +
  \sqrt{11 \gamma \max\{\gamma, \LtwoPn{\noise}\}}.
\end{equation*}
Simplifying gives Proposition~\ref{proposition:in-sample-convergence}.

\subsubsection{Proof of Lemma~\ref{lemma:star-shaped-localization}}
\label{sec:proof-star-shaped-localization}

Mimicking the proof of~\cite[Lemma 13.12]{Wainwright19}, we
begin with~\cite[Eq.~(13.40)]{Wainwright19}:
\begin{equation*}
  \P(A(u))
  \le \P(Z_n(u) \ge 2 u^2)
  ~~ \mbox{for} ~~
  Z_n(u) \defeq \sup_{g \in \mc{H},
    \LtwoPns{g} \le u} \left|n^{-1} \sum_{i=1}^n \noise_i g(x_i)\right|.
\end{equation*}
Now note that if $\LtwoPns{g} \le u$, then the function $\noise \mapsto
|n^{-1} \sum_{i = 1}^n \noise_i g(x_i)|$ is $u / \sqrt{n}$-Lipschitz
with respect to the $\ell_2$-norm, so that convex concentration
inequalities~\cite[e.g.][Theorem 3.4]{Wainwright19}
imply that
$\P(Z_n(u) \ge \E[Z_n(u)] + t)
\le \exp(-\frac{t^2 n}{4 b^2 u^2})$
whenever $\sup \noise - \inf \noise \le b$, and so
for $b = 1$ we have
\begin{equation*}
  \P(Z_n(u) \ge \E[Z_n(u)] + u^2) \le
  \exp\left(-\frac{n u^2}{4}\right).
\end{equation*}
As $\E[Z_n(u)] = \localcomplexity_n(u)$, we
finally use that the normalized complexity $t \mapsto
\frac{\localcomplexity_n(t)}{t}$ is non-decreasing~\cite[Lemma
  13.6]{Wainwright19} to obtain that
for $u \ge \delta_n$,
$\frac{1}{u} \localcomplexity_n(u)
\le \frac{1}{\delta_n} \localcomplexity_n(\delta_n)
\le \delta_n$, the last inequality by assumption.
In particular, we find that for $u \ge \delta_n$ we have
$\E[Z_n(u)] = \localcomplexity_n(u)
\le u \delta_n \le u^2$, and so
\begin{equation*}
  \P(Z_n(u) \ge 2 u^2) \le
  \P(Z_n(u) \ge \E[Z_n(u)] + u^2)
  \le \exp\left(-\frac{n u^2}{4} \right),
\end{equation*}
as desired.

\end{document}